\documentclass[11pt]{report}
\usepackage{titlesec}
\titleformat{\chapter}
  {\normalfont\LARGE\bfseries}{\thechapter}{1em}{}
\titlespacing*{\chapter}{0pt}{3.5ex plus 1ex minus .2ex}{2.3ex plus .2ex}

\usepackage{amsmath}
\usepackage{amssymb}
\usepackage{amsfonts}
\usepackage{mathrsfs}
\usepackage[usenames]{color}
\usepackage[numbers]{natbib}
\usepackage{enumitem} 
\usepackage{amscd} 
\usepackage{authblk}
\usepackage{hyperref}

\usepackage{graphicx}
\usepackage{epsfig}
\usepackage{amsthm}
\usepackage{xcolor}

\usepackage{aurical}

\usepackage{latexsym}
\newsavebox{\toy}
\savebox{\toy}{\framebox[0.65em]{\rule{0cm}{1ex}}}
\newcommand{\QED}{\usebox{\toy}\end{demo}}
\newtheorem{theorem}{Theorem}[section]
\newtheorem*{theorem*}{Theorem}
\newtheorem{lemma}[theorem]{Lemma}
\newtheorem{proposition}[theorem]{Proposition}
\newtheorem{corollary}[theorem]{Corollary}
\newtheorem{assumption}[theorem]{Assumption}
\newtheorem{conjecture}[theorem]{Conjecture}
\newtheorem{definition}[theorem]{Definition}
\newtheorem{remark}[theorem]{Remark}
\newtheorem{exercise}[theorem]{Exercise}

\newtheorem{property}[theorem]{Property}

\newcommand{\Theorem}[1]{\begin{theorem}\label{Thm.#1}}
\newcommand{\Lemma}[1]{\begin{lemma}\label{Lem.#1}}
\newcommand{\Proposition}[1]{\begin{proposition}\label{Prop.#1}}
\newcommand{\Corollary}[1]{\begin{corollary}\label{Cor.#1}}
\newcommand{\Assumption}[1]{\begin{assumption}\label{Ass.#1}\rm}
\newcommand{\Definition}[1]{\begin{definition}\label{Def.#1}\rm}
\newcommand{\Remark}[1]{\begin{remark}\label{Rem.#1}\rm }
\newcommand{\Exercise}[1]{\begin{exercise}\label{Exe.#1}\rm }

\def\qed{\hfill\rule{.2cm}{.2cm}\par\medskip\par\relax}

\newcommand{\bd}{\begin{displaymath}}
\newcommand{\ed}{\end{displaymath}}
\newcommand{\bdn}{\begin{equation}}
\newcommand{\bdnl}{\begin{equation}\label}
\newcommand{\edn}{\end{equation}}
\newcommand{\barray}{\begin{array}}
\newcommand{\earray}{\end{array}}
\newcommand{\bds}{\begin{description}}
\newcommand{\eds}{\end{description}}
\newcommand{\bitemize}{\begin{itemize}}
\newcommand{\eitemize}{\end{itemize}}
\newcommand{\benumerate}{\begin{enumerate}}
\newcommand{\eenumerate}{\end{enumerate}}
\newcommand{\btabbing}{\begin{tabbing}}
\newcommand{\etabbing}{\end{tabbing}}
\newcommand{\bcenter}{\begin{center}}
\newcommand{\ecenter}{\end{center}}
\newcommand{\bflushright}{\begin{flushright}}
\newcommand{\bflushleft}{\begin{flushleft}}
\newcommand{\eflushright}{\end{flushright}}
\newcommand{\eflushleft}{\end{flushleft}}
\newcommand{\bdnn }{\begin{eqnarray*}}
\newcommand{\ednn }{\end{eqnarray*}}
\newcommand{\bdmn}{\begin{eqnarray}}
\newcommand{\edmn}{\end{eqnarray}}

\newcommand{\nn}{\nonumber}
\newcounter{biblio}
\newenvironment{references}%
{\begin{list}{[\arabic{biblio}]}{\usecounter{biblio}%
\setlength{\leftmargin}{2.5em}\setlength{\rightmargin}{0pt}%
\setlength{\labelwidth}{2em}\setlength{\itemsep}{0pt}}}{\end{list}}
\newcommand{\References}%
{\vspace{2.8ex plus .3ex minus .3ex}%
\begin{center}{\bf References}\end{center}\begin{references}}

\newcommand{\Q}{{\mathbb{Q}}}
\newcommand{\R}{{\mathbb{R}}}
\newcommand{\rd}{\R^d}

\newcommand{\lef}{\left}
\newcommand{\rig}{\right}
\newcommand{\ri}{\right}

\newcommand{\8}{\infty}

\newcommand{\6}{\partial}

\newcommand{\E}{{\bf E}}
\def\P{{\mathbb P}}

\renewcommand{\a}{\alpha}
\renewcommand{\b}{\beta}

\newcommand{\rmd}{\mathrm{d}}
\newcommand{\D}{\Delta}

\newcommand{\h}{\eta}

\newcommand{\la}{\lambda}

\newcommand{\n}{\nu}

\newcommand{\om}{\omega}

\newcommand{\cC }{{\cal C}}

\newcommand{\cF }{{\cal F}}
\newcommand{\cG }{{\cal G}}

\newcommand{\cM }{{\cal M}}

\newcommand{\cO }{{\cal O}}

\makeatletter
\def\section{\@startsection{section}{1}{\z@}{-3.5ex plus -1ex minus 
 -.2ex}{2.3ex plus .2ex}{\bf}}
\def\subsection{\@startsection{subsection}{2}{\z@}{-3.25ex plus -1ex minus 
 -.2ex}{1.5ex plus .2ex}{\bf}}
\makeatother

\newcommand{\cvlaw}{\stackrel{\rm{ law}}{\longrightarrow}}
\newcommand{\eqlaw}{\stackrel{\rm{ law}}{=}}
\newcommand*\cvLdeux{\overset{L^2}{\longrightarrow}}

\def\1#1{{\bf 1}{\{#1\}}}

\newcommand{\IP}{{\mathbb P}}

\newcommand{\be}{{\beta}}
\newcommand{\eps}{{\varepsilon}}

\newcommand{\bx}{{\bf x}}

\newcommand{\Yt}{{\mathsf Y}^{(t)}}

\AtBeginDocument{
   \def\MR#1{}  }

\usepackage{variations}

\begin{document}

\pagestyle{myheadings}
\markboth{FC-CC}{BPPE}

\author{Francis Comets, Cl\'ement Cosco}

\affil[]{\small{Universit\'e Paris Diderot\\
Laboratoire de Probabilit\'es, Statistique et Mod\'elisation\\ LPSM (UMR 8001 CNRS, SU, UPD)\\
B\^atiment Sophie Germain, 8 place Aur\'elie Nemours, 75013 Paris\\
{\tt comets@lpsm.paris,  ccosco@lpsm.paris}}
}

\title{Brownian Polymers in Poissonian Environment: \\ a survey. }
\maketitle
\begin{abstract}
We consider a space-time continuous directed polymer in random environment. The path is Brownian and the medium is Poissonian.
We review many results obtained in the last decade, and also we present  new ones. In this fundamental setup, we can make use of fine formulas and strong tools 
from stochastic analysis for Gaussian or Poisson measure, together with martingale techniques. 
 These notes cover the matter of a course presented during the Jean-Morlet chair 2017 of CIRM
"Random Structures in Statistical Mechanics and Mathematical Physics" in Marseille.
 \\[.3cm]\textbf{Keywords:} Directed polymers, random environment; weak disorder, intermediate disorder, strong disorder; free energy; Poisson processes, martingales.
\\[.3cm]\textbf{AMS 2010 subject classifications:}
Primary 60K37. Secondary 60Hxx, 82A51, 82D30.
\end{abstract}

{\parskip=0pt {
\tableofcontents}}
\chapter{Introduction}

This survey is based on a course presented by the first author at the Research School in Marseille, March 6-10, 2017. 
The school was organized by the chair holders, Kostya Khanin and Senya Shlosman, of the Jean-Morlet chair 2017 of CIRM,
\begin{quotation}
Random Structures in Statistical Mechanics and Mathematical Physics.
\end{quotation}
The model is a space-time continuous directed polymer in random environment. 
In this regard, it is one of the most basic such model and it plays a fundamental
role. 
Directed polymers are described by random paths, which are influenced by randomly located impurities which may be attractive or repellent.
Such models have been widely considered in statistical physics, disordered systems and stochastic processes. 

\medskip

As an informal definition we model the polymer by a random path $\bx=(\bx(t); t \geq 0)$ taking values in $\rd$ and interacting with time-space Poisson points $(t_i, x_i)$ called environment. The path sees such a point if at time $t_i$ it is located within a fixed distance $r$ from $x_i$. Denoting by $\#_t(\bx)= \sum_{i: t_i\leq t}  {\bf 1}_{|\bx(t_i)-x_i|\leq r}$ the number of  Poisson points seen by the path $\bx$ up to time $t$,
the model with time horizon $t$ at inverse-temperature parameter $\b$ is associated to the Hamiltonian 
$$
-\frac 12 \int_0^t |\dot \bx (s)|^2 ds + \b \#_t(\bx)\;.
$$

In this model where the path is Brownian and the medium is Poissonian,  we benefit from nice formulas and strong tools from stochastic calculus for
Gaussian or Poisson measure and martingale techniques. 
\medskip

The notes are essentially based on references \cite{CY05, CYkokyuroku, CYBMPO2,CoscoIntermediate}, gathering and unifying the matter scattered in these references, and containing novel contributions and perspectives as emphasized below. It also parallels the book \cite{CStFlour} which deals similar models in the  discrete framework, and we warn the reader of the existence of many results available for one 
particular model but  not for the others. We do not reproduce all details or computations, but we rather try to give the general picture and the essential arguments. 
\medskip

Let us mention the main highlights in this survey and also the new results:
 \begin{enumerate}
 \item We establish in section \ref{ch:3} a fine continuity estimate under spatial shifts for the limit of the martingale. This is achieved by a smart use of mirror coupling. 
\item Section \ref{ch:dir} contains a nice original account on directional free energy. We develop a full approach of
disorder strength based on directional free energy.
\item In section \ref{ch:CM} we develop an original approach to diffusivity at weak disorder, based on Camer\'on-Martin transformation (see theorem \ref{prop:diff}).
\item Section \ref{sec:IR} is dedicated to the intermediate disorder regime and KPZ equation. We give a synthetic account with all the central ideas. 
\end{enumerate}

The detailed matter and the organization appear most clearly in the table of contents,
which is a useful source to follow the line all through the notes.

\newcommand{\gibbs}{P_t^{\b,\om}}
\newcommand{\gibbss}{P_s^{\b,\om}}
\newcommand{\gibbsDeux}{P_t^{\b,\om,\otimes 2}}

\newcommand{\gibbsni}{P_{n-i}^{\b,\theta_{i,x}\om}}
\newcommand{\gibbsmn}{P_{m+n}^{\b,\om}}
\newcommand{\gibbsmdn}{P_{m}^{\b,\theta_{n,y}\om}}
\newcommand{\gibbsdn}{P_{2n}^{\b,\om}}
\newcommand{\gibbsbu}{P_n^{\b_1,\om}}
\newcommand{\gibbsbd}{P_n^{\b_2,\om}}
\newcommand{\gibbsbp}{P_n^{\b',\om}}
\newcommand{\gibbsdnpu}{P_{2n+1}^{\b,\om}}
\newcommand{\gibbsmu}{P_{t-1}^{\b,\om}}
\newcommand{\gibbsc}{P_{t}^{\b,\om,2}}
\newcommand{\gibbsmuc}{P_{t-1}^{\b,\om, 2}}
\newcommand{\gibbspc}{[P_{t}^{\b,\om}]^{\otimes 2}}
\newcommand{\gibbsmupc}{[P_{t-1}^{\b,\om}]^{\otimes 2}}
\newcommand{\gibbsp}{P_n^{\b',\om}}

 \chapter{Free energy and phase transition} \label{sec:thermo}


 Notations and conventions: all through the notes, we will use the same symbols $P, \P, \dots$ to denote probability measures and mathematical expectations; e.g., $P[X]$ is the 
 $P$-expectation of the random variable $X$.  
\medskip

In this section, we introduce the model and two central thermodynamic quantities, the quenched and the annealed free energies.

\section{Polymer model}

The model is defined as a Brownian motion in a random potential. 
\medskip

$\bullet$ {\it The free measure :}   
$(B=\{ B_t\}_{t \geq 0}, P_x)$ is a Brownian motion on 
the $d$-dimensional Euclidean space $\rd$ starting from $x \in \rd$.
We will use short notation  $P_0=P$.

$\bullet$ {\it The random environment:}  
$\om = \sum_i \delta_{(T_i,X_i)}$ is a Poisson point process on $\R_+ \times \rd$ with intensity measure $\nu dt dx$, where $\nu$ is a positive parameter. We suppose that $\om$ is defined on some probability space $(\Omega,\mathcal{G},\IP)$, and we define $\mathcal{G}_t$ to be the $\sigma$-field generated by the environment up to time $t$:
\begin{equation}
\om_t=\om_{|(0,t]\times \rd} \;,\quad \mathcal{G}_t = \sigma\left(\om_t(A) ; A \in \mathcal{B}(\mathbb{R}_+ \times \mathbb{R}^d) \right),
\end{equation}
where $\mathcal{B}(\mathbb{R}_+ \times \mathbb{R}^d)$ denotes the Borel sets of $\mathbb{R}_+ \times \mathbb{R}^d$.
\medskip

From these two basic ingredients, we define the object we consider in the notes. Fix $r>0$, and let $U(x)$ denote Euclidean (closed) ball in $\rd$ with radius $\gamma_d^{-1/d} r$,
$$
U(x)= {\mathrm B}  (x , \gamma_d^{-1/d} r).
$$ with $\gamma_d$ the volume of the unit ball, so $U(x)$ has volume $r^d$.  The tube around path $B$ is the following subset of $(0,t]\times\R^d$: 
\begin{equation}
V_t(B)= \left\{ (s,x): s \in (0,t], x \in U(B_s)\right\}\,.
\end{equation}
When the indicator function
 \begin{equation} \label{def:chitx}
\chi_{s,x}= {\mathbf 1} \{x \in U(B_s) \} = {\mathbf 1} \{|x-B_s| \leq \gamma_d^{-1/d} r\}
\end{equation}
has value 1 [resp., 0],  the path $B$ does see [resp. does not see] the point $(s,x)$. For a fixed path $B$, the quantity defined by
\begin{equation} \label{eq:214}
\om (V_t) = \int_{(0,t] \times \rd} \chi_{s,x} \, \om(ds,dx),
\end{equation}
is the number of Poisson points seen by the path $B$ up to time $t$, playing the role of $\#_t$ in the Introduction. Note that under $\P$, the variable $\om (V_t)$ is Poisson distributed with mean $\nu t r^d$.

$\bullet$ {\it The polymer measure:} Fixing a realization $\om$ of the Poisson point process and a value of the time horizon
$t>0$, we define the probability measure $\gibbs$ on the path space 
${\cC}(\R_+; \rd)$ equipped with its Borel field by  
\begin{equation} \label{mnen}
d \gibbs =\frac{1}{Z_{t}(\om,\b,r)}
\exp \{\beta \om(V_t)\} \;dP,
\end{equation}
where $\b \in \R$ is a parameter (the inverse temperature), where 
\begin{equation} \label{Zt}
Z_t=Z_{t}(\om,\b,r)
=P\lef[\exp \lef(\b \om(V_t)
\rig) \rig]
\end{equation}
is the normalizing constant making $\gibbs$ a probability measure on the path space.

 The model has been introduced by Nobuo Yoshida as a polymer model, and first appeared in \cite{CY05} in the literature.  For $\b >0$ the path is attracted by the Poisson points, and repelled otherwise. The Poisson environment represents randomly dispatched impurities. 
For negative $\b$ the model relates to Brownian motion in Poissonian obstacles
\cite{donskervaradhan-wienersaus, Sznitman} which can be traced back to works of Smoluchowski \cite{Smoluchowski}. Here we consider  a directed version, in contrast to crossings \cite{Wuthrich1, Wuthrich2, Wuthrich-geodesics, Wuthrich-bornes} where the path is stretched ballistically.
Our model with $\b \to +\8$ is related to Euclidean first passage percolation \cite{HowardNewman97, Howard00} with exponent $\alpha=2$ therein.

Also, for a  branching Brownian motion in random medium \cite{Shiozawa-clt, Shiozawa-loc}, $Z_t$ is equal to the mean population size in the medium given by $\omega$.

\section{Some key formulas and notations}
We first recall three basic formulas that we will use repeatedly.
\begin{itemize}
\item 
For all non-negative and all non-positive measurable functions $h$ on $\mathbb{R}_+\times \mathbb{R}^d$, the Poisson formula for exponential moments (chapter 3. of \cite{LastPenrose}) writes
\begin{equation} \label{eq:PoissonExpMoment}
\P\left[ e^{\int h(s,x) \om_t(ds dx)}\right] = \exp \int_{]0,t]\times \rd} \nu ds dx \left( e^{h(s,x)}-1 \right)   \;.
\end{equation}
The formula remains true when $h$ is replaced by $i h$, for any real integrable function $h$. 
\item Introducing the notation 
\begin{equation}
\lambda(\b)= e^\b-1\;,
\end{equation}

the linearization formula for Bernoulli writes
\begin{equation}\label{eq:linBer}
e^{\b {\bf 1}_A}-1 = (e^\b-1) {\bf 1}_A =  \lambda(\b) {\bf 1}_A\;.
\end{equation}
\item For all $s\geq 0$, we have 
\begin{equation}\int_{\mathbb{R}^d} \chi_{s,x} \hspace{0.3mm} \rmd x = r^d.
\end{equation}
\end{itemize}

\section{Quenched free energy}
It is defined as the rate of growth of the partition function, and it is a self-averaging property.

\begin{theorem} \label{th:freeenergy}
The quenched free energy
$$ p(\b, \nu) =  \lim_{t \to \8} \frac{1}{t} \ln Z_t (\om, \b, r)
$$
exists a.s. and in $L^p$-norm for all $p \geq 1$, and is deterministic,
$$
 p(\b, \nu) = \sup_{t >0}  \frac{1}{t}\P[\ln Z_t]\;.
$$
\end{theorem}
\begin{remark} We omit the parameter $r>0$  from the notation for the free energy.
The reason is that,  in contrast to $\b$ and $ \nu$, it is kept fixed most of the time.
\end{remark}

$\Box$ Let $\theta_{t,x}$ the space-time shift operator on the environment space, 
$$\theta_{t,x} \big(\sum_i \delta_{(T_i,X_i)}\big)= \sum_i \delta_{(T_i-t,X_i-x)}\;.$$ By Markov property of the Brownian motion,  we have for $s, t \geq 0$,
\begin{eqnarray} \nn
{Z_{t+s}} &=& P\left[ e^{\b \om(V_t)} e^{\b \om(V_{t+s} \setminus V_t)} \right] \\ \nn
&=& P\left[ e^{\b \om(V_t)} P\left[ e^{\b \om(V_{t+s} \setminus V_t)} \big \vert B_t\right] \right] \\ \label{eq:Markov2}
&=& P\left[ e^{\b \om(V_t)}  Z_s \circ \theta_{t,B_t}
\right] \\ \label{eq:Markov}
&=& {Z_t} \times \gibbs[ Z_s \circ \theta_{t,B_t}] \;,
\end{eqnarray}
a remarkable identity expressing the Markov structure of the model.
Let $u(t)= \P[ \ln Z_t]$. By the independence property of Poisson points, $\om_{|]s,t]}$ is independent of $\cG_s$ for all $0\leq s \leq t$.  Then, denoting by $\P^{\cG_t}$ the conditional expectation and  conditional
probability given $\cG_t$, we have 
\begin{eqnarray*}
u(t+s) 
&=&  \P \big[ \ln \gibbs[ Z_s \circ \theta_{t,B_t)}]  \big]+ \P \ln Z_t 
\\
&\stackrel{\rm Jensen}{\geq} 
& \P  \gibbs[   \ln  Z_s \circ \theta_{t,B_t)} ]+u(t)  \\
&=&  
 \P \P^{\cG_t} \gibbs[   \ln  Z_s \circ \theta_{t,B_t)} ]+u(t)  \\
&\stackrel{\rm Fubini}{=}&  \P  \big[  \gibbs[ \P^{\cG_t}[\ln Z_s \circ \theta_{t,B_t)}] ] \big]+u(t)\\
&=& \P \big[   \gibbs[ u(s) ]\big]+u(t) \qquad\qquad \qquad (\om {\rm \ shift\ invariant})\\
&=& u(s)+u(t)\;.
\end{eqnarray*}
Hence the function $u(t)$ is superadditive. By the superadditive lemma, we get the existence of the limit
\begin{equation}
\nn
\lim_{ t \to \8} \frac{u(t)}{t} = \sup_{t >0} \frac{u(t)}{t} \;.
\end{equation}
Now, anticipating the concentration inequality \eqref{eq:concentrationlogZt} and the continuous time bridging
\eqref{eq:OrderOfCenteredlogZt},
we derive that 
\begin{equation}
\nn
\frac{1}{t} \big( \ln Z_t- \P[\ln Z_t]\big) \longrightarrow 0
\end{equation}
almost surely and in $L^p$ for all $p\geq 1$ finite.  \qed
\section{Annealed free energy and hierarchy of moments}
We compute the expectation of the partition function over the medium using \eqref{eq:214} and Fubini,
\begin{eqnarray}
\P[ Z_t] &=&  P \P\left[ e^{\b \int \chi_{s,x} \om_t(ds,dx)} \right] \nn \qquad \qquad \qquad \qquad
\\&\stackrel{\eqref{eq:PoissonExpMoment}}{=}&
P \exp \int_{]0,t]\times \rd} \left( e^{\b \int \chi_{s,x}}-1 \right) \nu ds dx \nn \qquad 
 \\
&\stackrel{\eqref{eq:linBer}}{=}&P \left[  \exp \la(\b)  \int_{]0,t]\times \rd}  \chi_{s,x}  \nu ds dx \right] \nn  \qquad  \qquad 
 \\&= &\label{eq:annealed}
 \exp\{ t \nu \la r^d\} \;.
\end{eqnarray}
Hence $\P[ Z_t]$ grows in time at exponential rate  $p^{(1)}(\b, \nu)= \nu \lambda(\b)r^d$. More generally, it is natural to consider the rate of growth of the $s$-th moment 
of the partition function,
$$
p^{(s)}(\b, \nu)= \lim_{t \to \8} \frac{1}{st} \ln \P[ Z_t^s]\,,\quad s >0.
$$
By H\"older inequality, $\|Z\|_r \leq \|Z\|_s$ for $r \leq s$, these rates are non-decreasing in $s$, and for integer values, they can be expressed by handy variational formulas
using large deviation theory. By Jensen's inequality, we have for all $t$
\begin{equation} \label{eq: annealedbound}
\frac{1}{t}\P[\ln Z_t] \leq   \frac{1}{t} \;\ln \P[ Z_t] = p^{(1)}(\nu,\b)\;,
\end{equation}
yielding the so-called {\em annealed bound} : $$ p(\b, \nu) \leq  p^{(1)}(\nu,\b)\;.$$  Summarizing the above,  we have a chain of inequalities
\begin{equation}
\nn
  p(\b, \nu) \leq  p^{(1)}(\b, \nu) \leq \ldots \leq  p^{(k)}(\b, \nu) \leq  p^{(k+1)}(\b, \nu) \leq \ldots . 
\end{equation}
It is commun folklore that in a large class of models, the first inequalities in the above chain are equalities, while they become strict from
$k^*= \inf \{k \geq 0: p^{(k)}(\b, \nu) <  p^{(k+1)}(\b, \nu)$ (with the convention $p^{(0)}=p$). Considering the sequence of rates $(p^{(k)}; k \geq 1)$ is
classical approach to intermittency \cite{CarmonaMolchanov, khoshnevisan, Zeldovich} and sect. 2.4 of \cite{BertiniCancrini95}. 

 In the directed case, we focus at $k=0,1$ only, since the latter is explicit.
%
\begin{proposition} \label{prop:p}
Basic properties of the free energy:
\begin{enumerate}
\item \label{eq:p1} For $\b \neq 0, \nu >0$, we have
$ \b \nu r^d < p(\b, \nu) \leq \nu \lambda(\b)r^d$.
\item \label{eq:p2}
$ \b \to p(\b, \nu)$ is convex.
\item \label{eq:p3} The excess free energy 
\begin{equation}
\label{def:excessfreeenergy}   
 \psi(\b,\nu)=\nu \lambda(\b)r^d- p(\b, \nu)
\end{equation}
is non-decreasing in $|\b|$ and in $\nu$. It is jointly continuous. 
\end{enumerate}
\end{proposition}
\noindent
$\Box$ The second inequality in item \ref{eq:p1} is the annealed bound. The first one follows from an inifinite-dimensional version of Jensen's
inequality; this version being curiously overlooked in the literature, we recall the full statement:
\begin{lemma}[Lemma A.1 in \cite{RASaihp11}]\label{lem:RAS-J}
Let $g$ be a bounded measurable function on a product space ${\cal X} \times  {\cal Y}$, $\mu$ a probability measure on $\cal X$ and $\rho$ a probability measure on $\cal Y$. Then
$$
\ln \int_{\cal X} e^{\int_{\cal Y}  g(x,y) d\rho (y) } d \mu(x) \leq \int_{\cal Y}  \left[ \ln \int_{\cal X}  e^{g(x,y)} d\mu (x) \right] d \rho(y) .
$$
\end{lemma}
We apply it with $\rho=\P, \mu=P, g(x,y)=\b \om(V_t)$ to get the desired bound\footnote{We explain in this note 
why the Lemma is an infinite-dimensional version of Jensen's inequality:  
the functional $\psi(f)= \ln \int_{\cal X} e^{f(x)} d\mu(x)$ is convex, and the function $f(\cdot)=g(\cdot, y)$ is randomly chosen with $\rho(dy)$.}.
However this bound is not so great here, since the simple one $p(\b,\nu) \geq t^{-1} \P[ \ln Z_t]$ for a fixed $t$ (which comes from superadditivity of $u(t)$)
is not linear, but strictly convex in $\b$ and then already better.
%
%

Item \ref{eq:p2} is the standard convexity of free energy, 
$$
\frac{\partial^2}{\partial \b^2} \ln Z_t = {\rm Var}_{\gibbs} \big( \om(V_t)) \big) > 0\;,
$$
where ${\rm Var}_{\gibbs}$ denotes the variance under the polymer measure in a fixed environment $\om$.

We now turn towards item \ref{eq:p3}, in the case $\b \geq 0$ (the other case being similar). We use specific properties of the medium, infinite divisibility: 
for $\nu, \D>0$, we note that the superposition 
$\om+\hat \om$ of two independent PPP with intensities $\nu$ and $\D$ is  a PPP with intensity $\nu+\D$. Writing $\P$ the expectation 
over both variables ${\om,\hat \om}$, we compute by conditioning
\begin{eqnarray} \nn
\P \ln Z_t(\om) & \stackrel{\b \geq 0}{\leq} & 
\P \ln Z_t(\om + \hat \om) \\ &=& \P \P\left[ \ln Z_t(\om + \hat \om) \vert \om\right] \nn \\
&\stackrel{\rm Jensen}{\leq} & \P \ln \P\left[ Z_t(\om + \hat \om)  \vert \om\right] \nn \\
&=& \P \ln Z_t (\om) + t \D  \lambda(\b) r^d. \nn
\end{eqnarray}
This proves  monotonicity of $\psi$  in $\nu$. This proves  at the same time  continuity in $\nu$ (locally uniformly in $\b$) and the joint continuity in $(\b, \nu)$.

\medskip

The plain identity $\P [Yf(Y)]= \theta \P[f(Y+1)]$ for a r.v. $Y$ distributed as a Poisson law with mean $\theta$ has a counterpart for PPP,
an integration by parts formula known as Slivnyak-Mecke formula (e.g., p.50 in \cite{KendallStoyanMecke} or th.~4.1~in \cite{LastPenrose}):  Let $\cM$ be the space of point measures on $(0,t] \times \rd$ and $h: 
\R_+ \times \rd \times {\cM} \to \R_+$ measurable, then
\begin{equation}
\label{eq:IPP}
\P \left[ \int h(s,x; \om_t) \om_t(ds,dx) \right]= \int_{(0,t]\times \rd}  \P\big[ h(s,x; \om_t + \delta_{s,x}) \big]\; \nu ds dx  \;.
\end{equation}
With this in hand, we can show monotonicity of $\psi$ in $\b$: 
\begin{eqnarray} 
\frac{\partial }{\partial \b} \P \big[ \ln Z_t \big]
 \nn &=& \P \gibbs [\om(V_t)]   \\
 \nn &=& \P \int \om_t(ds dx)  \frac{ P\left[ \chi_{s,x} e^{\b \om(V_t)} \right]}{Z_t}\\
\nn &\stackrel{\eqref{eq:IPP}}{=}& \P \int_{(0,t]\times \rd} \nu ds dx  \frac{ P\left[ \chi_{s,x} e^{\b (\om(V_t)+\delta_{s,x})} \right]}{P\left[ 
e^{\b (\om(V_t)+\delta_{s,s})} \right]}  \\
\nn &\stackrel{\eqref{eq:linBer}}{=}& \P \int_{(0,t]\times \rd} \nu ds dx  \frac{ P\left[ e^\b \chi_{s,x} e^{\b \om(V_t)} \right]}{P\left[ \big(\lambda(\b)\chi_{s,x}
+1\big) e^{\b \om(V_t)} \right]} \\
\label{eq:meanenergy} 
&{=}& \P \int_{(0,t]\times \rd} \nu ds dx  \frac{ e^\b \gibbs[\chi_{s,x}]}{1+ \lambda(\b) \gibbs[\chi_{s,x}]} \;.
\end{eqnarray}
Define 
\begin{equation}
\label{def:psit}
\psi_t(\b,\nu) = t^{-1} \P \left[ \nu \lambda r^d -\ln Z_t \right] .
\end{equation}
With the identity
$$ \frac{\partial }{\partial \b} \nu \lambda(\b)r^dt = \nu e^{\b} r^d t =  e^{\b} \int_{(0,t]\times \rd} \nu ds dx \gibbs[ \chi_{s,x}]$$
we obtain
\begin{equation}
\label{eq:psi'}
\frac{\partial }{\partial \b}  \psi_t(\b,\nu) = \frac{1}{t}e^\b \lambda(\b) \nu 
 \int_{(0,t]\times \rd}  ds dx \; \P \frac{ \gibbs[  \chi_{s,x}]^2}{1+ \lambda(\b) \gibbs[\chi_{s,x}]} \;,
 \end{equation}
which has the sign of $\b$. So the limit $\psi$ of $\psi_t$ is increasing in $|\b |$. \qed
\section{Phase transition}
An important  consequence of  monotonicity and continuity of $\psi$ in $|\b|$ in Proposition \ref{prop:p} is the existence and uniqueness of 
the critical temperatures introduced in the next statement, which is a direct consequence of the above.
\begin{theorem}\label{defdef..}
There exist $\b_c^+(\nu), \b_c^-(\nu)$ with  $-\8 \leq \b_c^- \leq 0 \leq \b_c^+ \leq + \8$ such that 
%
\begin{equation} \label{eq:highlow}
\left\{ 
\begin{array}{ccc}
\psi(\b,\nu)=0 \quad & {\rm if}  & \b \in [\b_c^-,\b_c^+]  \\
\psi(\b,\nu)>0 \quad  & {\rm if}  & \b < \b_c^- \; {\rm or} \; \b > \b_c^+.  
\end{array}
\right.
\end{equation}
Moreover,  $|\b_c^\pm(\nu)|$ is non-increasing in $\nu$.
\end{theorem}
These values $\b_c^+(\nu), \b_c^-(\nu)$ are called \emph{ critical} (inverse)  \emph{ temperatures} at density $\nu$ (they depend on $r$ as well).
The domains in the $(\b,\nu)$-half-plane defined by the first and second line in \eqref{eq:highlow} are called {\em high and low temperature region}
respectively. The boundary between the two regions is called the {\em critical line}, and a {\em phase transition} in the statistical mechanics sense occurs:
the quenched free energy $p(\b,\nu)$ is equal to the annealed free energy $p^{(1)}(\b,\nu)=\nu \lambda r^d$ 
-- an analytic function -- but analyticity of  $p(\b,\nu)$ breaks down when crossing the critical line.  
\medskip

To summarize our finding, we define the  high temperature region and the  low temperature region
\begin{equation} \nn
{\cal D} = \{ (\b,\nu): \psi(\b, \nu)=0 \}, \qquad {\cal L}  = \{ (\b,\nu): \psi(\b, \nu)>0\}.
\end{equation}
They are are delimited by the critical lines $\b_c^-(\nu)$ and $\b_c^+(\nu)$ from Definition \ref{defdef..}.
In the next sections we will discuss non-triviality of the critical lines, as well as fine properties. In section \ref{ch_overlap} we will understand that they correspond to delocalized  or localized
behavior respectively.

\chapter{Weak Disorder, Strong Disorder } \label{ch:3}
\section{The normalized partition function} In this section, we introduce a natural martingale that will play an important  role in many results concerning the asymptotic behavior of the polymer. 

For any fixed path of the brownian motion, $\{\om(V_t)\}_{t\geq 0}$ is a Poisson process of intensity $\nu r^d$ and has associated exponential martingales $\{\exp \lef(\b \om(V_t)-\lambda(\b)\nu r^d t\rig)\}_{t\geq 0}$. Hence, for	$t\geq 0$, the \emph{normalized partition function}
\begin{equation} \label{eq:Wt}
W_t = e^{-\lambda (\b)\nu r^d t} Z_t,
\end{equation}
defines a positive, mean $1$, c\`adl\`ag martingale with respect to $\{\mathcal{G}_t\}_{t\geq 0}$.

By Doob's martingale convergence theorem \cite[Chapter 2, Corollary 2.11]{revuzYor}, we get the existence of a random variable $W_\infty$ such that
\begin{equation}
W_\infty = \lim_{t\to\infty} W_t \  \text{ a.s.}
\end{equation} 

\begin{theorem}
There is a dichotomy: either the limit $W_\infty$ is almost-surely positive, or it is almost-surely zero. Otherwise stated, we have either
\begin{equation}  \label{eq:WD}
\IP \{W_\infty > 0\} = 1,
\end{equation}
or
\begin{equation} \label{eq:SD}
\IP \{W_\infty = 0\} = 1.
\end{equation}
\end{theorem}
\begin{proof}
Denote by $e_t$ the renormalized weight 
\begin{equation}
e_t = \exp(\b \om(V_t(B)) - \lambda(\b)\nu r^d t)
\end{equation}
By the Markovian property \eqref{eq:Markov2}, we get that for all positive times $t$ and $s$,
\begin{equation} \label{eq:WtIntegralEquation}
W_{s+t} = P [e_{t} \, W_s \circ \theta_{t,B_t} ].
\end{equation}
In Section \ref{sec:proofSelfconsistent}, we will justify that one can take the limit as $s\to\infty$ in this equality, in order to get that
\begin{equation} \label{eq:WinfinityIntegralEquation2}
W_\infty = P [e_{t} \, W_\infty \circ \theta_{t,B_t} ].
\end{equation}
Then, notice that \eqref{eq:WinfinityIntegralEquation2} also writes 
\begin{equation} \label{eq:WinfinityIntegralEquation}
W_\infty = W_t \, \int_{\mathbb{R}^d} P_t^{\b,\om} (B_t \in \mathrm{d} x)\, W_\infty \circ \theta_{t,x}.
\end{equation}
Since $W_t > 0$ $\IP$-a.s and since $P_t^{\b,\om}$ has positive density with respect to Lebesgue's measure, we obtain by \eqref{eq:WinfinityIntegralEquation} that
\[ \forall t>0, \quad \{W_\infty = 0\} = \{W_\infty \circ \theta_{t,x} = 0, \ x\text{-a.e.}\}, \]
or, equivalently,
\[  \{W_\infty = 0\} = \left\{ \int_{\mathbb{R}^d} P (B_t \in \mathrm{d} x)\, W_\infty \circ \theta_{t,x} =0 \right\}.
 \]
The event of the right-hand side belong 
to the $\sigma$-field $\mathcal{G}_{[t,\infty)} = \sigma\left(\om(A) ; A \in \mathcal{B}([t,\infty) \times \mathbb{R}^d) \right)$ completed by null sets, so
\[\{W_\infty = 0\} \in \bigcap_{t>0}^\infty \mathcal{G}_{[t,\infty)}.\]
The theorem now follows from Komogorov's 0-1 law.
\end{proof}

This dichotomy calls for a definition.

\begin{definition} We say that the polymer is in the \textbf{weak disorder} phase when $W_\infty > 0$ almost surely. We say it is in the \textbf{strong disorder} phase when $W_\infty = 0$ almost surely.
\end{definition}

The phase diagram is connected in the $\beta$-parameter space. 

\begin{theorem}\label{th:313} There exist two critical parameters ${\bar{\b}_c}^- \in [-\infty,0]$ and ${\bar{\b}_c}^+ \in [0,\infty]$, depending only on $\nu,r$ and $d$, such that
\begin{itemize}
\item For all $\b \in ({\bar{\b}_c}^-,{\bar{\b}_c}^+) \,\cup \{0\}$, the polymer belongs to the weak disorder phase.
\item For all $\b \in \mathbb{R}\setminus [{\bar{\b}_c}^-,{\bar{\b}_c}^+]$, the polymer belongs to the strong disorder phase.
\end{itemize}
\end{theorem}
\begin{proof}
Let $\theta$ be a real number in $(0,1)$ and denote $Y_t = W_t^{\theta}$ for all $t\geq 0$. The family $(Y_t)_{t\geq 0}$ is a collection of positive random variables verifying  \[ \sup_{t\geq 0} \IP \big[ Y_t^{1/\theta} \big] = \sup_{t\geq 0} \IP [W_t] = 1 < \infty.\] As $1/\theta$ is strictly greater than $1$, this relation implies the uniform integrablity of $(Y_t)_{t\geq 0}$. Since the process $(W_t^\theta)_{t\geq 0}$ converges almost surely to $W_\infty^\theta$, we get from uniform integrability that
\begin{equation} \label{eq:CV_Wtheta}
\lim_{t\to \infty} \IP \left[W_t^{\theta} \right] = \IP \left[ W_\infty^\theta \right].
\end{equation}
Now, one can observe that the right hand side term is positive if and only if \eqref{eq:WD} holds and that it is zero if and only if \eqref{eq:SD} holds. To prove the theorem, it is then enough to prove that $\b \mapsto \IP \left[ W_\infty^\theta \right]$ is a non-increasing function of $|\b|$ and choose for example 
\begin{equation}
\label{equation:barbetac}
{\bar{\b}_c}^+ = \inf\{\b \geq 0 : \IP \left[ W_\infty^\theta \right] = 0\} \;,
\end{equation}
which does not depend on $\theta \in (0,1)$.
Using \eqref{eq:CV_Wtheta}, we now just have to show that $\b \mapsto \IP \left[ W_t^\theta \right]$ is an non-increasing function of $|\beta|$ for all positive $t$.
By standard arguments, we get that
\begin{align*} \label{eq:derivee_PWt} \frac{\partial}{\partial \b} \IP \left[ W_t^\theta \right] 
& = \IP \left[ \theta W_t^{\theta-1} \frac{\partial}{\partial \b} W_t  \right]\\
& =  \IP \left[ \theta W_t^{\theta-1} P \left[ \left(\om(V_t) - \lambda'(\b) \nu r^d t \right) e^{\b \om(V_t) -\lambda (\b)\nu r^d t}\right] \right]\\
& = \theta \, P \, \IP \left[ W_t^{\theta-1} \left(\om(V_t) - \lambda'(\b) \nu r^d t \right) e^{\b \om(V_t) -\lambda (\b)\nu r^d t} \right].
\end{align*}
Introducing the probability measure $\IP^\b$ on point measures, given by
\[\mathrm d\IP^\b( \om) = e^{\b \om(V_t) - \lambda r^d \nu t} \mathrm d\IP(\om),\]
the derivative of $\IP[ W_t^\theta ]$ is now given by
\begin{equation} \label{eq:derivee_PWt}
\frac{\partial}{\partial \b} \IP \left[ W_t^\theta \right] = \theta \, P \  \IP^\b \left[W_t^{\theta - 1} (\om(V_t) - r^d \nu \lambda' t ) \right].
\end{equation}
In Proposition \ref{prop:PoissonGirsanov} just below, we will see that under the probability measure $\IP^\b$, $\om$ is a Poisson point process on $\mathbb{R}_+ \times \mathbb{R}$.

We can then use the \emph{Harris-FKG inequality} for Poisson processes \cite[th. 11 p. 31]{Last2016} in order to bound the above expectation. Indeed, the variable $\om(V_t) - r^d \nu \lambda' t$ is an increasing function of the point process and by definition, the process $W_t^{\theta - 1}$ is then a decreasing function of $\om$ when $\b \geq 0$ (resp. increasing when $\b < 0$). Applying the FKG inequality, we find that for positive $\b$
\begin{equation} \label{eq:FKG-application}
\IP^\b \left[W_t^{\theta - 1} (\om(V_t) - \lambda' r^d \nu t ) \right ] \leq  \  \IP^\b \left[W_t^{\theta -1} \right] \IP^\b \left[(\om(V_t) - r^d \nu \lambda' t )\right] = 0,
\end{equation}
where the last equality is a result of the relation
\[\IP [\om(V_t) e^{\b \om(V_t)}] =  \lambda'(\b) r^d \nu t.\]
The same result with opposite inequality comes when $\b < 0$.
Thus, we get from \eqref{eq:derivee_PWt} and \eqref{eq:FKG-application} that $\IP[W_t^\theta]$ is a non-increasing function of $|\b|$.
\end{proof}
We recall at this point that Poisson processes with mutually absolutely continuous intensity measures are themselves mutually absolutely continuous.
\begin{proposition} \label{prop:PoissonGirsanov} Let $\eta$ be a Poisson point process on a measurable space $E$, of intensity measure $\mu$. Let $f$ be a function such that $e^f-1 \in L^1(\mu)$. Then, under the probability measure $\mathbb{Q}$ defined by 
\[ \frac{\rmd \mathbb{Q}}{\rmd \mathbb{P}} = \exp\left(\int_E f(x) \, \eta(\rmd x) - \int_E (e^{f(x)}-1) \, \rmd \mu(x)\right),\]
the process $\eta$ is a Poisson point process of intensity measure $e^f \, \rmd \mu$.
\end{proposition}
\begin{proof}
Let $g$ be any non-negative measurable function. As the Laplace functional characterizes Poisson processes (theorem 3.9 in \cite{LastPenrose}), we compute it for the point process $\eta$ under the measure $\mathbb{Q}$:
\begin{align*}
\mathbb{Q} \exp\left\{-\int_E g(s) \, \eta(\rmd x)\right\} & = \mathbb{P} \exp \left\{ \int_{E} f(x)-g(x) \, \eta(\rmd x) \right\} e^{ -\int (e^{f(x)}-1) \, \rmd \mu(x) } \\
& =  \exp \left\{ \int_{E} (e^{ f(x)-g(x)}-1) \, \rmd \mu(x) \right\} e^{ -\int (e^{f(x)}-1) \, \rmd \mu(x) }\\
& =  \exp \left\{ \int_{E} (e^{-g(x)}-1) e^{f(x)} \, \rmd \mu(x) \right\},\\
\end{align*}
where the second equality is an application of \eqref{eq:PoissonExpMoment}.
The expression we obtain corresponds, as claimed, to a Poisson point process of intensity measure $e^f \, \rmd \mu$.
\end{proof}

\section{The self-consistency equation and UI properties in the weak disorder}
\subsection{Proof of the self-consistency equation on $W_\infty$} \label{sec:proofSelfconsistent}
In this section, we prove that one can take the limit in the identity $W_{s+t} = P [e_{t} \, W_s \circ \theta_{t,B_t}]$ and obtain the  equation of self-consistency:
\begin{equation}
\label{eq:self_consistent_equation}
W_\infty = P [e_{t} \, W_\infty \circ \theta_{t,B_t} ].
\end{equation}
A part of the problem is that we only have almost sure convergence of the $W_s \circ \theta_{t,x}$ for countable number of $x$'s. 
To deal with this issue, we show that the quantity
\begin{equation}W_t(x) := e^{-\lambda (\b)\nu r^d t} P_x\lef[\exp \lef(\b \om(V_t) \rig) \rig],
\end{equation} 
does not vary too much with $x$, in the sense of the following lemma:

\begin{lemma} \label{lm:distanceWtxWt0}
There exists a constant $C=C(\b,\nu,r)$, such that, for all $t\in [0,\infty]$ and $x,y\in\mathbb{R}^d$,
\begin{equation} \label{eq:distanceWtxWt0}
\IP\left[\left|W_t(x)-W_t(y)\right|\right] \leq C |x-y|.
\end{equation}
\end{lemma}
\begin{proof} To simplify the notations, we only consider the case where $y=0$. To recover the lemma, it is enough to argue that the Poisson environment is invariant in law under a translation in space of vector $y$.

To prove \eqref{eq:distanceWtxWt0} for $y=0$, we write the difference of the two martingales as expectations over two coupled Brownian motions, in the same environment.
The coupling we consider is the mirror coupling, which is defined as follows (see \cite{Hsu2013} for more details).

At time $0$, one of the Brownian motions is starting at $0$ and one is starting at $x \in \mathbb{R}^d$. Denote by $H$ be the hyperplane bisecting the segment $[0,x]$, which is the hyperplane passing by $x/2$ and orthogonal to the vector $x$. Let also
\[\tau = \inf\{t\geq 0 | B_t \in H\},\]
be the first hitting time of $H$ by $B$. Then, define $\widetilde{B}$ as the path that coincides with the reflection of path of $B$ with respect to $H$ for times before $\tau$, and that coincides with $B$ after $\tau$. 

\begin{figure} \centering
\includegraphics[scale=0.45]{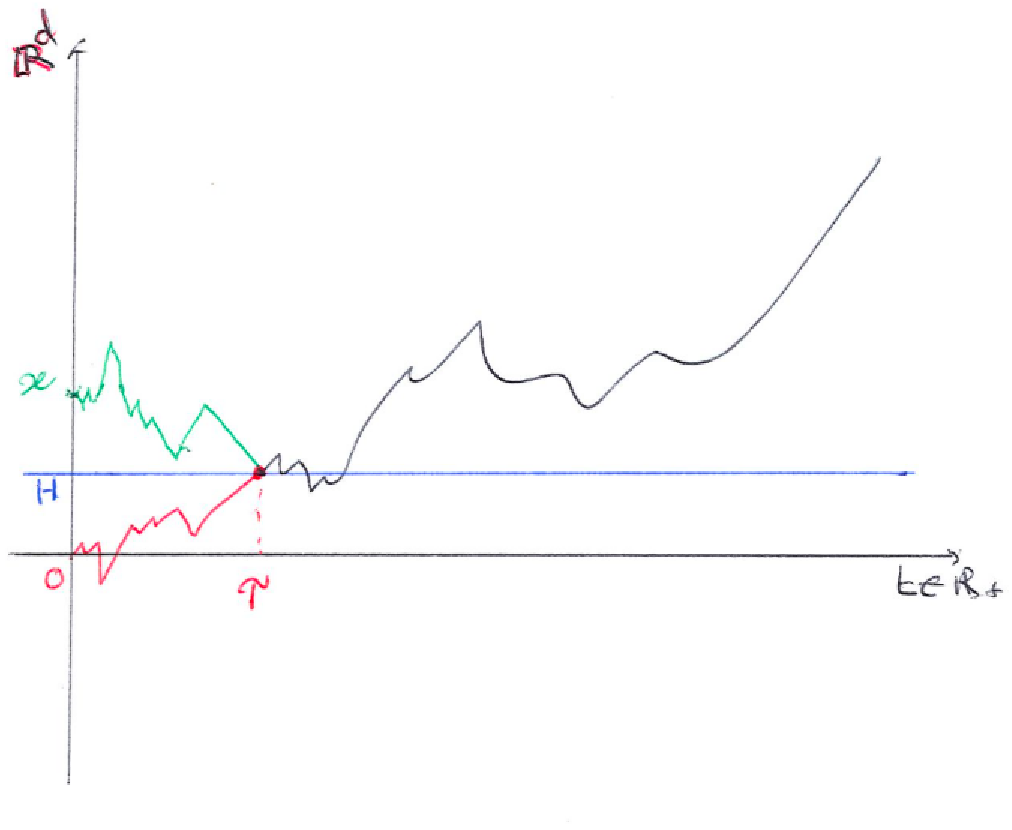}
\vskip -3cm
\caption{Mirror coupling in $d=1$.} 
 \label{fig:mirror}
\end{figure}
The process $\widetilde B$ has the law of a Brownian motion starting from $x$. Moreover, the time $\tau$ is the first time $B$ and $\widetilde B$ meet. After $\tau$, the processes coincide. The variable $\tau$ has the following cumulative distribution function:
\begin{equation} \label{eq:cumulativeOfTau}
P(\tau \geq z)=\phi_z(|x|),
\end{equation}
where, for positive $z$, \[\phi_z(|x|)=\frac{2}{\sqrt{2\pi z}}\int_{0}^{|x|/2} e^ {-u^2/2z}\rmd u,\] and where $|\cdot|$ is the Euclidean distance. In the litterature, a coupling that satisfies this relation is said to be maximal (see \cite{Hsu2013}).
Let $$e_t = \exp(\b \om(V_t(B)) - \lambda(\b)\nu r^d t), \quad \tilde{e}_t = \exp(\b \om(V_t(\widetilde B)) - \lambda(\b)\nu r^d t),$$ which we factorize in the contributions before and after $B$ and $\widetilde B$ coalesce, so that, for $t\in [0,\infty)$,
\begin{align*}
\IP[|W_t(x)-W_t(0)|] &\leq \IP P[|\tilde{e}_t - e_t|] \\
& = P\IP [|(\tilde{e}_{t\wedge \tau}-{e}_{t\wedge \tau})e_{(t-t\wedge \tau)^+}\circ \theta_{t\wedge \tau,B_{t\wedge \tau}}|]\\
& = P\IP [|\tilde{e}_{t\wedge \tau}-{e}_{t\wedge \tau}|],
\end{align*}
where the last equality is a result of the independance of the Poisson environment before and strictly after time $t\wedge \tau$. 

Then, we distinguish the cases where $B$ or $\widetilde{B}$ encounter a point of the environment before $t\wedge \tau$, and the cases where they don't. We get that $P\IP [|\tilde{e}_{t\wedge \tau}-{e}_{t\wedge \tau}|]$ writes:
\begin{align*}
& P\IP \left[|e_{t\wedge \tau} - \tilde{e}_ {t\wedge \tau}| \mathbf{1}_{\{\om(V_{t\wedge \tau}(B))>1,\,\om(V_{t\wedge \tau}( \widetilde B))>1\}}\right] + P\IP\left[(e_{t\wedge \tau} - e^{-\lambda \nu r^d t\wedge \tau}) \mathbf{1}_{\{\om(V_{t\wedge \tau}(B))>1,\,\om(V_{t\wedge \tau}( \widetilde B))=0\}}\right]\\
& \qquad + P\IP\left[(\tilde{e}_{t\wedge \tau} - e^{-\lambda \nu r^d t\wedge \tau}) \mathbf{1}_{\{\om(V_{t\wedge \tau}(B)=0,\,\om(V_{t\wedge \tau}( \widetilde B)>1\}}\right].
\end{align*}
We first use the triangle inequality in the first expectation of the sum, and neglect the negative terms in two other expectations. Then, recombining the terms, one obtains that

\begin{align}
\IP[|W_t(x)-W_t(0)|] & \leq 2 P\IP \left[e_{t\wedge \tau} \mathbf{1}_{\{\om(V_{t\wedge \tau}(B))>1\}}\right] + 2 P\IP \left[\tilde{e}_{t\wedge \tau} \mathbf{1}_{\{\om(V_{t\wedge \tau}(\widetilde{B}))>1\}}\right] \nonumber \\ 
& = 4 P\IP \left[e_{t\wedge \tau} \mathbf{1}_{\{\om(V_{t\wedge \tau}(B))>1\}}\right]  \label{eq:ineqInLemmaWtx}
\end{align}
where the equality is a consequence of invariance in law of the Poisson environment under the translation by $x$.

For any $t$, the variable $\om(V_t)$ is a Poisson r.v. of parameter $ \nu r^d t$, so that
\begin{equation*}
P\IP \left[e_{t\wedge\tau} \mathbf{1}_{\{\om(V_{t\wedge\tau}(B)>1\}}\right]  = \sum_{k=1}^\infty P \IP [e_{t\wedge\tau} \mathbf{1}_{\{\om(V_{t\wedge\tau})=k\}}]
 = \sum_{k=1}^\infty P \left[e^{\b k-\lambda\nu r^d (t\wedge\tau)} \frac{(\nu r^d (t\wedge\tau))^k}{k!}e^{-\nu r^d (t\wedge\tau)} \right],
\end{equation*}
from which we get by standard computations that
\begin{equation} \label{eq:expecInLemmaWtx}
P\IP \left[e_{t\wedge\tau} \mathbf{1}_{\{\om(V_{(t\wedge\tau)}(B)>1\}}\right] = P[1-\exp(-e^\b \nu r^d t\wedge\tau)] \leq P[1-\exp(-e^\b \nu r^d \tau)].\end{equation}

To control this last expectation, we will first notice that $\tau / |x|^2$ is independent of $|x|$, and we will compute its density. By equation \eqref{eq:cumulativeOfTau} and the change of variable $u=|x|\sqrt(z) v$, we get that
\begin{align*}
\P(\tau/|x|^2 > z) = \frac{1}{\sqrt{2\pi}} \int_0^{1/2\sqrt{z}} e^{-v^2/2} \rmd v.
\end{align*}
This function of $z$ is continuous and everywhere differentiable on $[0, \8)$. 
The variable $\tau/|x|^2$ hence admits a density with respect to the Lebesgue measure, given by
\[f(z) = \frac{1}{4\sqrt{2\pi}} \frac{e^{-1/8z}}{z^{3/2}}.\]
Therefore, one gets that the right-hand side of \eqref{eq:expecInLemmaWtx} can be written as
\begin{align*}
&P\left[\left(1-\exp(-e^\b \nu r^d {\tau}{|x|^{-2}} |x|^2)\right)\mathbf{1}_{\{\tau/|x|^2 \leq 1\}}\right] + P\left[\left(1-\exp(-e^\b \nu r^d {\tau}{|x|^{-2}} |x|^2)\right)\mathbf{1}_{\{\tau/|x|^2 > 1\}}\right]\\
&\qquad \qquad \qquad \qquad 
\leq e^\b \nu r^d |x|^2 + \int_1^\infty \left(1-\exp(-e^\b \nu r^d z |x|^2)\right) f(z) \rmd z.
\end{align*}
As there is some constant $C>0$ such that $f(z) \leq C z^{-3/2}$, after the change of variables $z|x|^2=u$, the integral above can be bounded by
\[|x| \int_{0}^\infty C \left(1-\exp\left(-e^\b \nu r^d u \right)\right) v^{-3/2} \rmd v,\]
where one can check that the integral converges. Since its value only depends on $\b,\nu$ and $r$, we finally get that there exists some constant $C'=C(\b,\nu,r)$, such that
\[\IP[|W_t(x)-W_t(0)|] \leq 4P[1-\exp(-e^\b \nu r^d \tau)] \leq C' (|x|^2 + |x|),\]
where the first inequality comes from combining \eqref{eq:ineqInLemmaWtx} and \eqref{eq:expecInLemmaWtx}. Since the $W_t(x)$ variables have bounded expectations, this proves the lemma in the case where $t<\infty$. The $t=\infty$ case is then a consequence of Fatou's lemma.
\end{proof}

\vspace{0.4cm}
We can now show that the self-consistency equation holds. Let $\delta >0$ be a parameter that will go to $0$. For all $q \in \mathbb{Z}^d$, define $\Delta( q)$ to be the cube of length $\delta$ centered at $\delta q$, so that all the cubes form a partition of the space $\mathbb{R}^d$.
We get that the right-hand side of \eqref{eq:WtIntegralEquation} satisfies
\begin{align}
P [e_{t} \, W_s \circ \theta_{t,B_t} ] & = \sum_{q\in\mathbb{Z}^d} P [e_{t} \, W_s \circ \theta_{t,B_t} ; B_t\in \Delta( q) ] \nonumber \\
& = \sum_{q\in\mathbb{Z}^d} P [e_{t} \, W_s \circ \theta_{t,\delta q} ; B_t\in \Delta( q) ] + A_s^\delta, \label{eq:approxInL1normForWSD}
\end{align}
where
$A_s^\delta = \sum_{q\in\mathbb{Z}^d} P [e_{t} \, (W_s \circ \theta_{t,B_t}-W_s \circ \theta_{t,\delta q}) ; B_t\in \Delta( q) ].$

First observe that $\IP$-almost surely, $W_s \circ \theta_{t,{\delta q}}$ converges  to $W_\infty \circ \theta_{t,{\delta q}}$ for all $q\in\mathbf{Z}^d$, so Fatou's lemma entails
\[\IP\text{-a.s.}, \quad \liminf_{s\to\infty} \sum_{q\in\mathbb{Z}^d} P [e_{t} \, W_s \circ \theta_{t,\delta q} ; B_t\in \Delta( q) ] \geq \sum_{q\in\mathbb{Z}^d} P [e_{t} \, W_\infty \circ \theta_{t,\delta q} ; B_t\in \Delta( q) ],\] so that, by \eqref{eq:approxInL1normForWSD} and letting $s\to\infty$ in \eqref{eq:WtIntegralEquation},
\begin{equation} \label{eq:ineqOnWinftyAfterFatou}
\IP\text{-a.s.}, \quad W_\infty \geq \sum_{q\in\mathbb{Z}^d} P [e_{t} \, W_\infty \circ \theta_{t,\delta q} ; B_t\in \Delta( q) ] + \liminf_{s\to\infty} A_s^\delta.
\end{equation}

Furthermore, using the fact that $W$ is a martingale, one can check that $s\mapsto A_s^\delta$ is also a martingale with respect to the filtration $\{\mathcal{G}_{t+s}\}_{s\geq 0}$.
For any time $S\geq 0$, Lemma \ref{lm:distanceWtxWt0} implies that it satisfies
\begin{align*}
\IP[|A_S^\delta|] & \leq \sum_{q\in\mathbb{Z}^d} P \left[\IP[e_{t}\,|W_S \circ \theta_{t,B_t}- W_S \circ \theta_{t,\delta q}|] \mathbf{1}_{B_t\in \Delta( q)} \right]\\
&\leq C(\delta^2 + \delta) \sum_{q\in\mathbb{Z}^d} P\left[ \mathbf{1}_{B_t\in \Delta( q)}\right]\\ &= C(\delta^2 + \delta),
\end{align*}
where, in the second inequality, we have factorized by $\IP[e_{t}] = 1$, using the independence under $\IP$ of the environment before and strictly after time $t$. Thus, by Doob's inequality \cite[Th. 3.8 (i), Ch.~1]{KS91}, we have for all $u>0$,
\[\IP\left[\sup_{0\leq s\leq S} |A_s^\delta| \geq u \right] \leq \frac{C(\delta^2 + \delta)}{u},\]
where we can let $S\to\infty$ by monotone convergence. This implies that $\sup_{s\geq 0} |A_s^\delta|$ converges in probability to $0$, when $\delta \to 0$, which in turn implies that
\[\liminf_{s\to\infty} |A_s^\delta| \overset{\IP}{\longrightarrow} 0.\]

Then, using Lemma \ref{lm:distanceWtxWt0} in the case where $t=\infty$, the same computation as above would show that
\[\IP\left[ \left| \sum_{q\in\mathbb{Z}^d}P \left[e_{t} \, W_\infty \circ \theta_{t,B_t} ; B_t\in \Delta( q)\right]- \sum_{q\in\mathbb{Z}^d} P \left[e_t \, W_\infty \circ \theta_{t,\delta q} ; B_t\in \Delta( q)\right] \right| \right] \leq C(\delta^2 + \delta),\]
and hence,
\[\sum_{q\in\mathbb{Z}^d} P \left[e_t \, W_\infty \circ \theta_{t,\delta q} ; B_t\in \Delta( q)\right]   \overset{L^1}{\longrightarrow} P [e_{t} \, W_\infty \circ \theta_{t,B_t}].\]
In particular, we have shown that the right-hand side of \eqref{eq:ineqOnWinftyAfterFatou} converges in probability to $P [e_{t} \, W_\infty \circ \theta_{t,B_t}]$, so that almost surely
\[ W_\infty \geq  P [e_{t} \, W_\infty \circ \theta_{t,B_t}].\]
Observe that the two quantities have the same expectations to conclude that \eqref{eq:self_consistent_equation} holds.

\subsection{Uniform integrability in the weak disorder}
\begin{proposition} \label{prop:UI}
The martingale $W_t$ is uniformly integrable if and only if the polymer is in the weak disorder phase, i.e. $W_\infty > 0,\, \IP$-almost surely.
\end{proposition}
\begin{proof}
If $W_t$ is UI, then $W_t$ converges in $L^1$ to $W_\infty$, so that $\IP[W_\infty] = 1 > 0$, therefore weak disorder must hold by the dichotomy.

Suppose now that the polymer is in weak disorder and set
\[X_{t,x} = \frac{W_\infty\circ \theta_{t,x}}{\IP [W_\infty]}.\]
The self-consistency equation \eqref{eq:self_consistent_equation} writes
$X_{0,0} = P[e_t X_{t,B_t}]$,
so that, as $X_{t,x}$ is independent of $\mathcal{G}_t$ for all $x$,
\[\IP[ X_{0,0} | \mathcal{G}_t] = P\left[e_t \IP[X_{t,B_t}]\right] = P[e_t].\]
This shows that for all $t\geq 0$, 
\[W_t = \IP[X_{0,0} | \mathcal{G}_t] \quad \text{a.s.}\]
Hence, $(W_t)_{t\geq 0}$ is uniformly integrable since the family of the right-hand side is a uniformly integrable martingale.
\end{proof}

\section{The $\bf{L^2}$-region}
\begin{theorem} \label{th:LTwoRegion} (i) There exist two critical parameters ${\b_2}^- \in [-\infty,0]$ and ${\b_2}^+ \in [0,\infty]$, depending only on $\nu,r$ and $d$, such that, if $\b \in ({\b_2}^-,{\b_2}^+) \, \cup \{0\}$ then
\begin{equation}\label{eq:supWt2} \sup_{t\in\mathbb{R}} \IP [W_t^2] < \infty,
\end{equation} and such that the supremum is infinite if $\b \in \mathbb{R}\setminus [{\b}_2^-,\b_2^+]$.

(ii) Furthermore, if $d\geq 3$, there exists a constant $c(d)\in (0,\infty)$, such that \eqref{eq:supWt2} holds whenever
\begin{equation} \label{eq:L2_condition}
\lambda(\b)^2 \nu r^{d+2} < c(d),
\end{equation}

(iii) In particular, ${\b_2}^- < 0$ and ${\b_2}^+ > 0$ whenever $d\geq 3$.

(iv) Also for  $d\geq 3$, when $\nu r^{d+2} < c(d)$ the constant in \eqref{eq:L2_condition}, we have ${\b_2}^- =-\8$.
\end{theorem}

\begin{definition}We call the $\,\textbf{$\bf{L^2}$-region}$ the set of parameters $\beta, \nu, r$ for which \eqref{eq:supWt2} holds.\end{definition}

\begin{proof}[Proof of Theorem \ref{th:LTwoRegion}]
We introduce the product measure $P^{\otimes 2}$ of two independent Brownian motions $B_t$ and $\tilde{B}_t$ starting from $0$ with respective tubes $V_t$ and $\tilde{V}_t$. The main idea is to write
\[W_t^2 = P^{\otimes 2} [e^{\b \om(V_t)} e^{\b \om(\tilde{V}_t)}]e^{-2\lambda \nu r^d t},\]
so that, using Fubini's theorem,
\begin{equation} \IP[W_t^2] = P^{\otimes 2} \, \IP [e^{\b (\om(V_t)+ \om(\tilde{V}_t)}]e^{-2\lambda \nu r^d t}.
\end{equation}
One can see that \[\om(V_t) + \om(\tilde{V}_t) = 2\om(V_t \cap \tilde{V}_t) + \om(V_t \Delta \tilde{V}t),\]
which is the sum of two independent Poisson random variables; computing their Laplace transforms leads us to
\begin{eqnarray}
 \IP[W_t^2]
\nonumber 
&=& \exp \left( \lambda(2\b) \nu |V_t \cap \tilde{V}_t| + \lambda \nu |V_t \Delta \tilde{V}_t| - 2\lambda \nu r^d t \right) \nonumber \\
&=& \exp\left(\lambda^2 \nu |V_t \cap \tilde{V}_t| \right) \nonumber \\
\label{eq:cvWt2} &\underset{t\to\infty}{\longrightarrow}& \exp\left(\lambda^2 \nu |V_\infty \cap \tilde{V}_\infty| \right), 
\end{eqnarray}
where the second equality is obtained using $|V_t|=|\tilde{V}_t|=tr^d$ and $\lambda(\b)^2=\lambda(2\b)-2\lambda(\b)$, while the limit is justified by monotone convergence.
Now,
\begin{align*} |V_\infty\cap \tilde{V}_\infty| &= \int_0^\infty |U(B_t)\cap U(\tilde{B}_t)|\rmd t \\
&= \int_0^\infty |U(0)\cap U(\tilde{B}_t-B_t)|\rmd t\\
& \eqlaw \int_0^\infty |U(0)\cap U(B_{2t})|\rmd t,
\end{align*}
since $(B_t + \tilde{B}_t)_{t\geq 0} \eqlaw (B_{2t})_{t\geq 0}$. Hence, using monotone convergence and \eqref{eq:cvWt2}, we get that
\begin{equation} \label{eq:supESPWt2}
\sup_{t\in\mathbb{R}} \IP [W_t^2] = \lim_{t\to\infty} \IP [W_t^2] = P\left[\exp\left(\frac{\lambda(\b)^2}{2} \nu \int_0^\infty |U(0)\cap U(B_{t})|  \rmd t\right) \right],
\end{equation}
where the first equality is a consequence of $(W_t^2)_{t\geq 0}$ being a submartingale. Equation \eqref{eq:supESPWt2} shows that $\sup_{t} \IP[W_t^2]$ is an increasing function of $|\b|$, which proves part (i) of the theorem.

To prove the second part, we will bound the right hand side of \eqref{eq:supESPWt2}. First observe that \[|U(0)\cap U(B_t)| \leq |U(0)| \, \mathbf{1}_{|B_t| \leq 2\gamma_d^{-1/d}r} \eqlaw r^d \mathbf{1} \left\{\left|B\left(\frac{t\gamma_d^{2/d}}{4r^{2}}\right)\right| \leq 1 \right\},\]
so that, by a change of variables,
\begin{equation} \label{eq:bound_supWt2} \sup_{t\in\mathbb{R}} \IP [W_t^2] \leq P\left[\exp\left(2\lambda(\b)^2 \nu r^{2+d}\gamma_d^{-2/d} \int_0^\infty \mathbf{1}_{|B_t| \leq 1} \, \rmd t\right) \right].
\end{equation}
When $d\geq 3$, the Brownian motion is transient and has the following property:
\[\alpha(d) = \sup_{x\in \mathbb{R}^d} P_x \left[ \int_0^\infty \mathbf{1}_{|B_t| \leq 1}\,\rmd t \right] < \infty.\]
By Khas'minskii's lemma \cite[p. 8, Lemma 2.1]{Sznitman}, this implies that
\[P \left[ \exp\left(u \int_0^\infty \mathbf{1}_{|B_t| \leq 1} \, \rmd t\right) \right] < (1-u\alpha)^{-1},\]
whenever $u\alpha < 1$. Looking back at \eqref{eq:bound_supWt2}, this condition finally leads to \eqref{eq:L2_condition}.

Part (iii) is obtained by observing that $\lambda(\beta) \to 0$ as $\beta \to 0$, so that condition \eqref{eq:L2_condition} is fulfilled for small enough $\b$.

To prove (iv), just note that the right-hand side of  \eqref{eq:L2_condition} tends to $\nu r^{d+2}$ as $\b \to -\8$.
\end{proof}

\section{Relations between the different critical temperatures}

The critical values $\b_2^\pm, {\bar{\b}_c}^\pm$ and $  \b_c^\pm$ defined in Theorems \ref{defdef..}, \ref{th:313} and
\ref{th:LTwoRegion}
are ordered. 

\begin{proposition} The following properties hold:
\begin{enumerate}[label=(\roman*)]
\item For all $d\geq 1$,
\begin{equation}
\b_2^+ \leq {\bar{\b}_c}^+ \leq \b_c^+ < \infty \quad \text{and} \quad \b_c^- \leq {\bar{\b}_c}^- \leq \b_2^-.
\end{equation}
\item When $d\geq 3$, these parameters are all non-zero.
\item When $d\geq 3$ and $\nu r^{d+2}$ is small enough, $\b_c^- = {\bar{\b}_c}^- = \b_2^- = -\infty$.
\end{enumerate} 
\end{proposition}
\begin{remark} Point (iii) tells us that if the intensity of the Poisson point process or the radius are sufficiently small, the polymer will not really be impacted by the environment.
\end{remark}
\begin{remark}\label{rk:criticd=12}
For $d=1,2$, it holds that
that $\b_2^\pm = {\bar{\b}_c}^\pm = \b_c^\pm = 0$. The reader is referred to the proofs in \cite{BeLa15, Bertin, CY05, Lacoin10} for other similar models, and can be convinced that the arguments go through in our case \cite{BertinUnp}.
\end{remark}
\begin{remark}
A long-standing conjecture  is that $\b_c^\pm = {\bar{\b}_c}^\pm$, i.e., that the weak/strong disorder transition coincide with the high/low temperature one. 
\end{remark}

\begin{proof}
We first show that $\b_2^+ \leq {\bar{\b}_c}^+$.
Suppose $\b_2^+>0$ and let $0\leq \b < \b_2^+$. By definition, we have
\[\sup_{t\geq 0} \IP [W_t^2] < \infty,\]
so that $(W_t)_{t\geq 0}$ is a martingale bounded in $L^2$. Thus, $(W_t)_{t\geq 0}$ converges in $L^2$ norm, which implies $L^1$ convergence.

Since $\IP [W_t] = 1$ for all $t$, we get that $\IP [W_\infty] = 1$, so \eqref{eq:WD} must hold and hence $\b \leq {\bar{\b}_c}^+$. As it is true for all $\b<\b_2^+$, the desired inequality follows directly.

We now turn to the proof of ${\bar{\b}_c}^+ \leq \b_c^+$. Again, suppose that ${\bar{\b}_c}^+>0$ and let $0\leq \b < {\bar{\b}_c}^+$. We have $W_\infty >0$ almost surely, so $\ln W_t \to \ln W_\infty$ almost surely, that is to say
\[\ln Z_t - \lambda \nu r^d t \underset{t\to\infty}{\longrightarrow} \ln W_\infty.\]
Dividing by $t$, we get that
\[\frac{1}{t}\ln Z_t - \lambda \nu r^d \underset{t\to\infty}{\longrightarrow} 0,\]
thus $p(\b, \nu, r) = \lambda \nu r^d$, i.e. $\b \leq \b_c$.
The same argument goes for the negative critical values. This ends the proof of (i).

Points (ii) and (iii) are repeated from Theorem \ref{th:LTwoRegion}. 
\end{proof}

 \chapter{Directional free energy} \label{ch:dir}
 
 In this section we make use of the Brownian nature of the polymer and the invariance of the medium under shear transformations, which induces a lot of symmetries in the model, culminating with quadratic shape function and the equality
 \eqref{eq:p0=}. 
 
 \section{Point-to-point partition function}

With $P_{s,x}^{t,y}$  the Brownian bridge  in $\R^d$ joining $(s,x)$ to $(t,y)$, we introduce the point-to-point partition (P2P) function
\begin{equation} \label{PtoPdefinition}
Z_t(\om, \b; x) = P_{0,0}^{t,x}\left[ \exp \{ \b \om(V_t)\}\right]\;,
\end{equation}
from which we can recover the point-to-level (P2L) partition function
 \begin{equation}
\label{eq:intZ}
Z_t(\om,\b)= \int_{\rd} Z_t(\om, \b; x) \rho(t,x) dx \; 
\end{equation}
by conditioning on $B_t$.
We use the standard notation 
$
\rho(t,x) = (2 \pi t)^{-d/2} \exp - |x|^2/2t
$ for the heat kernel in $\R^d$. 
For $\xi \in \R^d$, define the shear transformation $\tau_\xi: \R_+\times \R^d \to \R_+\times \R^d$ by
\begin{equation}\nn
\tau_\xi(s,x)=(s,x+s\xi)\;,
\end{equation}
which is one to one with $\tau_\xi^{-1}=\tau_{-\xi}$.
Since $\tau_\xi$ acts on the graph of functions $f: \R_+ \to \R^d$, we denote its action on functions
by
\begin{equation} \label{eq:tauxif}
\tau_\xi f : s \mapsto f(s)+s\xi \;,
\end{equation}
so that $\tau_\xi(s,f(s))=(s, \tau_\xi(f)(s)$.
The pushed forward of  a point measure by $\tau_\xi$  is defined by
\begin{equation}\label{eq:tau}
\tau_\xi \circ \big(\sum_i \delta_{(t_i, x_i)}\big) = \sum_i \delta_{(t_i, x_i + \xi t_i)}=  \sum_i \delta_{\tau_\xi (t_i, x_i)}
 \;,
\end{equation}
where it is clear that $\tau_\xi \circ \om$ is again a Poisson point process with intensity $\nu ds dx$,
i.e., $\tau_\xi \circ \om=\om$ in law. With $B$ the canonical process, under the measure $P_{0,0}^{t,t\xi}$
the process 
$W=\tau_{-\xi}(B)$
is a Brownian bridge $(0,0) \to (0,0)$. Therefore, for all $\om$,
\begin{eqnarray} \nn
Z_t(\om, \b; t\xi) &=&  P_{0,0}^{t,t\xi}\left[ \exp \{ \b \om(V_t (B) )\}\right] \\ \nn
&=&  P_{0,0}^{t,t\xi}\left[ \exp \{ \b \om(V_t (\tau_{\xi} (W)) )\}\right] \\ \nn
&=&  P_{0,0}^{0,0}\left[ \exp \{ \b \om(\tau_{\xi} (V_t (B))) \}\right] \\ \nn
&=&  P_{0,0}^{0,0}\left[ \exp \{ \b (\tau_{\xi} \circ \om)(V_t (B)) \}\right] \\ \label{eq:CMBB}
&=& Z_t(\tau_{\xi} \circ \om, \b; 0)\;.
\end{eqnarray}
This implies that $Z_t(\om, \b; x)$ has same law as $Z_t(\om, \b; 0)$.
%
%
We can prove that the {\em directional free energy},  in the direction $\xi \in \rd$, 
$$ p^{\rm dir}(\b, \nu; \xi ) =  \lim_{t \to \8} \frac{1}{t} \ln Z_t (\om, \b; t\xi)
$$
exists a.s. and in $L^p$-norm for all $p \geq 1$, and is equal to
$
\lim_{t \to \8}  t^{-1}\P[\ln Z_t(\om, \b; t \xi)]
$. The route is quite different from  Theorem \ref{th:freeenergy}, it follows the lines of chapter 5 in \cite{Sznitman} for the undirected case, that we briefly sketch now:
define the tube around the Brownian path between times $s \leq t$,  $V_{s,t}=V_{s,t}(B)=\cup_{u \in [s,t]} \{u\}\times U(B_u)$, and also
\begin{equation}\nn
{\mathfrak e}_{s,t}(x,y;\om) := P\left[ e^{\b \om(V_{s,t})} {\bf 1}_{B_t \in U(y)} \big\vert B_s=x\right] = P_{x} \left[ e^{\b \theta_{s,x} \!\circ \om(V_{t-s})} {\bf 1}_{B_{t-s} \in U(y)} \right],
\end{equation}
which is the integral of the P2P partition function over a ball of radius $r$. Then, the quantity 
\begin{equation} \nn
a_{s,t}(x,y;\om)= \inf_{z \in U(x)} \ln {\mathfrak e}_{s,t}(z, y; \om)
\end{equation}
is superadditive, in particular we have
$$a_{0,s+t}(0, (s+t)\xi; \om) \geq  a_{0,s}(0, s\xi; \om) + a_{s,s+t}(s\xi, (s+t)\xi; \om), $$
and the subadditive ergodic theorem shows the existence of the limit $t^{-1} a_{0,t}(0, t\xi; \om)$ as $t \to \8$, say $ p^{\rm dir}(\b, \nu; \xi )$, a.s. and in $L^1$. 
($L^p$-convergence will follow from the concentration inequality, which remains unchanged).
Then, one can show that the infimum over $z \in U(x) $ in the definition of $a_{0,t}(0,y;\om)$ can be dropped in the limit $t \to \8$, as well as the integration 
 in the definition of  $e_{0,t}(0,y;\om)$ on the fixed domain $ U(y)$. Proving these claims requires some work with quite a few technical estimates;
 we do not write the details here, the reader is referred to section 5.1 in \cite{Sznitman}.

By \eqref{eq:CMBB}, $Z_t(\om, \b; x)\eqlaw Z_t(\om, \b; x')$ and thus, for all $\xi, \xi' \in \rd$,
\begin{equation}\label{eq:samepdir}
p^{\rm dir}(\b, \nu; \xi ) = p^{\rm dir}(\b, \nu; \xi' ) \;.
\end{equation}
\section{Free energy does not depend on direction}
Let $P^h$ be the Wiener measure with drift $h \in \rd$, i.e., the probability measure on the path space ${\cal C}(\R_+, \rd)$ such that for all $t$,
\begin{equation} \nn
\left(\frac{ d P^h}{d P}\right)_{| \cF_t}= \exp \left\{ h \cdot B_t - \frac{t|h|^2}{2}\right\}\;.
\end{equation}
By Cameron-Martin formula, under $P^h$, the canonical process $B$ is a  Brownian motion  with drift $h$, i.e., $W=\tau_{-h}(B)$ is a standard Brownian motion under $P^h$
and has the same law as $B$ under $P$. Thus, the partition function for the drifted Brownian polymer
\begin{eqnarray}
Z_t^h(\om,\b) &\stackrel{\rm def.}{=}& P^h\left[ \exp \{ \b \om(V_t(B))\}\right] \nn \\
\nn &=& P^h\left[ \exp \{ \b \om(V_t(\tau_{h}(W)))\}\right] \\
\nn &=& P\left[ \exp \{ \b \om(V_t(\tau_{h}(B)))\}\right] \\
&=&   Z_t(\tau_{-h} \circ \om,\b) \label{eq:CMh}
\end{eqnarray}
as in \eqref{eq:CMBB}. It would be routine, and this time exactly as in the proof of  Theorem \ref{th:freeenergy}, to show the existence of  free energy for the drifted Brownian polymer 
\begin{equation} \nn
p^h(\b, \nu) := \lim_{t \to \8} t^{-1} \ln Z_t^h(\om,\b)
\end{equation}
a.s. and in $L^p$; But in fact, this is even unnecessary since \eqref{eq:CMh} yields the existence of the limit. 
Now, the previous display together with invariance of $\om$ under shear shifts imply that
\begin{equation}
p^h(\b, \nu) = p(\b, \nu) \;.
\end{equation}
On the other hand, similar to  \eqref{eq:intZ} we have
 \begin{equation}
\label{eq:intZh}
Z_t^h(\om,\b)= \int_{\rd} e^{h \cdot x - t|h|^2/2} Z_t(\om, \b; x) \rho(t,x) dx\;.
\end{equation}
By Laplace method, 
it follows from standard work that 
\begin{eqnarray} \nn
p^h(\b, \nu) 
\nn &=& \sup_{\xi \in \rd} \left\{ h \cdot \xi - |h|^2/2 - |\xi|^2/2  + p^{\rm dir}(\b, \nu; \xi )\right\}
\\ \label{eq:ph=} &\stackrel{\eqref{eq:samepdir}}{=}& p^{\rm dir}(\b, \nu; 0 ) - |h|^2/2 + \sup_{\xi \in \rd} \left\{ h \cdot \xi - |\xi|^2/2 \right\}\\
\nn &=& p^{\rm dir}(\b, \nu; 0 ).
\end{eqnarray}
Finally, all the above  notions of free energy coincide:
\begin{equation}
\label{eq:p0=}
\boxed{\quad p(\b, \nu) = p^{\rm dir}(\b, \nu; \xi ) = p^h(\b,\nu)\;.\quad }
\end{equation}
{\bf Conclusion:}  The critical values  $\beta_c^{\pm}$  for  equality of quenched and annealed free energy are the same for all free energies (P2P in all directions, P2L
with all drifts).

 \section{Local limit theorem}

In the $L^2$-region, a local limit theorem was discovered by Sina\"i  \cite{Si95} in the discrete case, and extended to our continuous model by Vargas \cite{Va04}. 
 \medskip

Define the time-space reversal  operator on the environment $\theta_{t,x}^\leftarrow $,
acting on point measures as 
\begin{equation}\label{eq:thetainv}
\theta_{t,x}^\leftarrow (\sum_i \delta_{(t_i, x_i)}\big) = \sum_i \delta_{(t-t_i, x_i-x)}
\end{equation}

\begin{theorem}[Local limit theorem; \cite{Va04}, Th. 2.9] \label{prop:LLT}
Assume $\b \in (-\b_2^-, \b_2^+)$. Then, for any constant $A>0$ and any positive function $\ell_t$ tending to $\8$ with $\ell_t = o(t^a)$ for some $a<1/2$,
\begin{equation} \label{eq:llt1}
Z_t(\om, \b; x) = W_\8 \times W_\8 \circ \theta_{t,x}^\leftarrow + \eps_t(x) \;,
\end{equation}
 and
 \begin{equation}\nn
Z_t(\om, \b; x) = W_{\ell_t} \times W_{\ell_t} \circ \theta_{t,x}^\leftarrow + \delta_t(x) \;,
\end{equation}
 with error terms vanishing as $t \to \8$,
 $$
 \sup_{|x| \leq A \sqrt{t}} \P[ |\eps_t(x)|] \to 0 \;, \qquad  \sup_{|x| \leq A \sqrt{t}} \P[ |\delta_t(x)|^2] \to 0 \;.
 $$
\end{theorem}

Intuitively, the local limit theorem states that, the polymer ending at $x$ at time $t$ only "feels" the environment at times $s$ close to 0  and locations close to 0  or close t
at "large" times $s$ close to $t$ and locations close to $x$. (See Figure \ref{fig:LLT}.) In between, it behaves like a Brownian bridge.
\medskip


\begin{conjecture}
We formulate two conjectures:

$\bullet$ It is natural to define another pair of critical inverse temperature, analogue to the weak/strong disorder transition:
\begin{eqnarray} \label{eq:betacdir}
\overline{\b}_c^{+,{dir}}&=& \sup\{ \b \geq 0: \lim_t \P[ (W_t^{{dir}, \xi})^{1/2}] >0\}, \\ \nn   \overline{\b}_c^{-,{dir}}&=& \inf\{ \b \leq 0: \lim_t \P[ (W_t^{{dir}, \xi})^{1/2}] >0\}.
\end{eqnarray}
Using Jensen inequality  in \eqref{eq:intZ}, it is not difficult to get $\overline{\b}_c^{+,{dir}} \in [\b_2^+,  \b_c^+ ]$, $\overline{\b}_c^{-,{dir}} \in [\b_c^{-}, \b_2^-]$.
We conjecture that 
the equality holds, i.e., 
$$\overline{\b}_c^\pm = \overline{\b}_c^{\pm,{dir}}. $$

$\bullet$ A long standing conjecture is that the local limit theorem  \eqref{eq:llt1} holds the way all through the weak disorder region. Note that the latter conjecture would imply that the former one holds.
\end{conjecture}
\begin{figure}[h!!!] \centering
\includegraphics[scale=0.4]{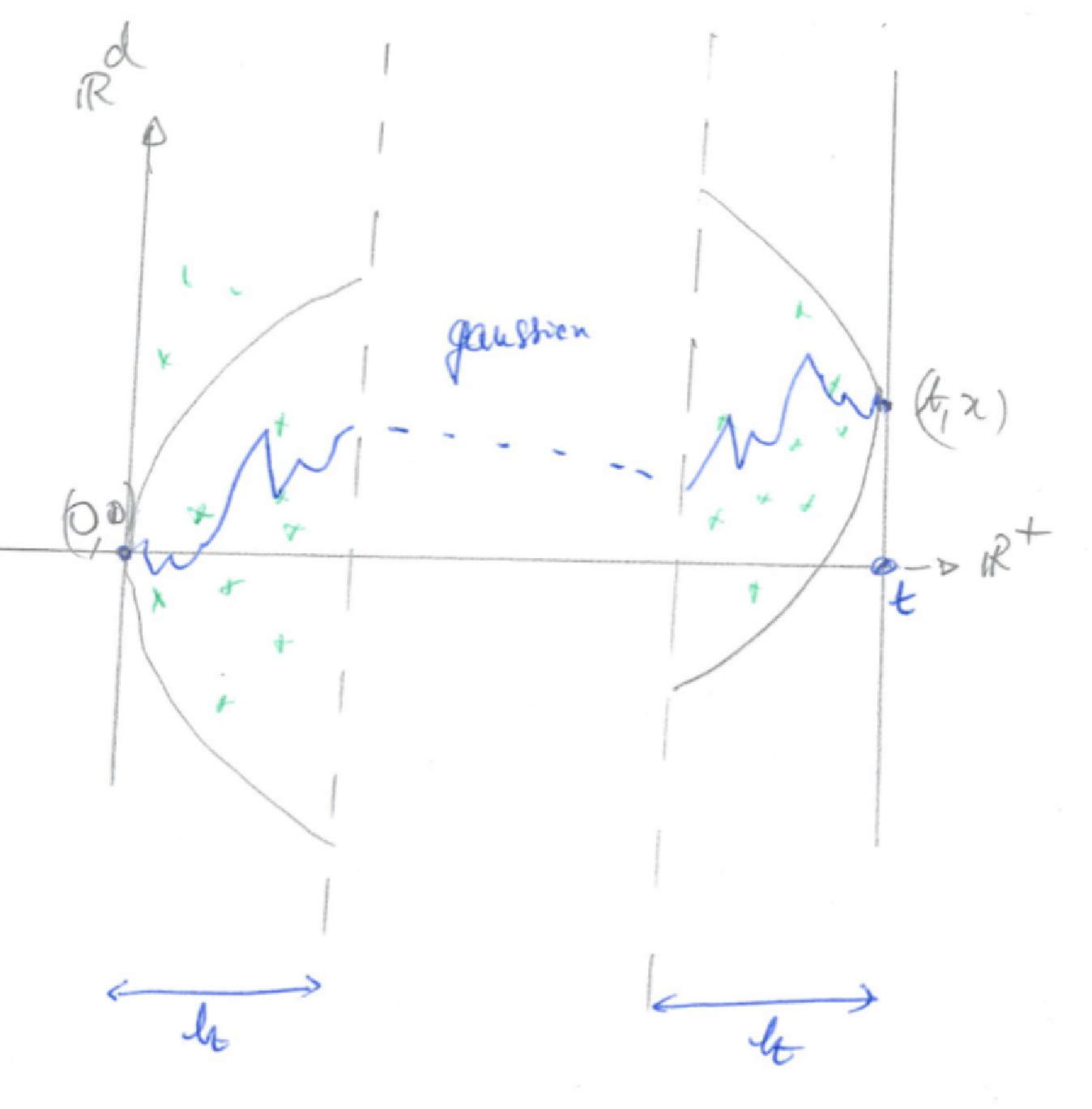}
\caption{The local limit theorem. The P2P partition function only feels details of the environment close to the space-time endpoints $(0,0)$ and $(t,x)$. In between, it behaves like the Gaussian propagator.} 
 \label{fig:LLT}
\end{figure}

\chapter{The replica overlap and localization} \label{ch_overlap}

The utlimate goal of this section is to show that the thermodynamic phase transition of section \ref{sec:thermo} is a  localization transition for the polymer.

To do that, we need fine tools from stochastic analysis. The starting point is Doob-Meyer decomposition, a natural and strong tool to study  stochastic processes in which the process is written as  the sum of a local martingale (the impredictable part) and a bounded variation predictable process (the tamed part).
We start by recalling some martingale properties of the Poisson environment that will prove usefull throughout the following chapters.

\section{The compensated Poisson measure and some associated martingales}
Given the Poisson point process $\om$, we introduce the compensated measure $\bar \om$,
\begin{equation}
\bar \om(\rmd s \rmd x)= \om(\rmd s \rmd x)-\nu\hspace{0.3mm} \rmd s \rmd x,
\end{equation}
and we abreviate its restriction to $(0,t]\times \rd$ by $\bar \om_t$. By definition,  for all function $f(s,x,\om)$ that verifies
\begin{equation} \label{eq:integrabiltyPoissonIntegrand}
\int_{[0,t]\times \mathbb{R}^d} \IP[|f(s,x,\cdot)|] \rmd s \rmd x < \infty,
\end{equation}
the compensated integral of $f$ is given by
\begin{equation}
\int f(s,x,\om) \bar \om_t(\rmd s \rmd x) = \int f(s,x,\om) \om_t(\rmd s \rmd x) - \int_{[0,t]\times\mathbb{R}^d} f(s,x,\om) \nu \hspace{0.3mm} \rmd s \rmd x.
\end{equation}

Furthermore, we say that a function $f(t,x,\om)$ is \emph{predictable}, if it belongs to the sigma-field generated by all the functions $g(t,x,\om)$ that satisfy the following properties:
\begin{enumerate}[label=(\roman*)]
\item for all $t>0$, $(x,\om)\to g(t,x,\om)$ is $\mathcal{B}(\mathbb{R}^d) \times \mathcal{G}_t$-measurable;
\item for all $(x,\om)$, $t\to g(t,x,\om)$ is left continuous.
\end{enumerate}
Then, if the function $f(t,x,\om)$ is predictable, provided that \eqref{eq:integrabiltyPoissonIntegrand} holds and that
\[\int_{[0,t]\times\mathbb{R}^d} \IP[f(s,x,\cdot)^2]\rmd s \rmd x < \infty,\]
the process $t\to \int f \hspace{0.3mm} \rmd \bar \om_t$ is a square-integrable martingale associated to $(\mathcal{G}_t)_{t\geq 0}$, of previsible bracket \cite[Section II.3.]{IkedaWatanabe}
\begin{equation} \label{eq:bracket}
\left\langle \int f \rmd \bar \om \right\rangle_t = \int_{[0,t]\times \mathbb{R}^d} f^2\, \nu \hspace{0.3mm} \rmd s \rmd x.
\end{equation}
The previsible bracket has the property that $ (\int f \hspace{0.3mm} \rmd \bar \om_t)^2 - \langle \int f \hspace{0.3mm} \rmd \bar \om \rangle_t$ is a martingale. In particular,
\begin{equation} \label{eq:MartingaleVarianceFormula}
\mathrm{Var} \left(\int f(s,x,\cdot) \om_t(\rmd s \rmd x)\right) = \int_{[0,t]\times \mathbb{R}^d} \IP[ f(s,x,\cdot)^2] \nu \hspace{0.3mm} \rmd s \rmd x.
\end{equation}

\section{The Doob-Meyer decomposition of $\ln \, Z_t$}
Since $W_t$ is a martingale and $\ln$ is concave, $-\ln(W_t)$ is a submartingale, for which we  
want to get a Doob-Meyer decomposition \cite[Ch.VI]{RogersWilliams}.

In the following, we will use the notation $\Delta_s X := X_s - X_{s-}$ for any c\`adl\`ag process $X$. Let $\zeta_t := e^{\b \om(V_t)}$. As $\om(V_t)$ can be expressed as a sum over $\om_t$, we get by telescopic sum that $\zeta_t = 1 + \int \om_t(\rmd s\rmd x)\Delta_{s} \zeta$, an integral over $\R_+\times\rd$. Averaging over the Brownian path, we also get that $Z_t$ can be expressed as a sum over the process. Therefore, by telescopic sum,
\[\ln Z_{t} = \int \om_t(\rmd s \rmd x) \Delta_s \ln Z.\]
Now, let $(t,x)\in  \R_+ \times\mathbb{R}^d$ be any point of the point process $\om$. As $(t,x)$ is almost surely the only point of $\om$ at time $t$, we can write that
\begin{align*}
Z_{t} &= P\left[ e^{\b}e^{\b \om(V_{t-})} \chi_{t,x}(B) \right] + P\left[ e^{\b \om(V_{t-})} \left(1-\chi_{t,x}(B)\right) \right]\\
& = P\left[ (e^{\b} -1) e^{\b \om(V_{t-})} \chi_{t,x}(B) \right] + P\left[ e^{\b \om(V_{t-})} \right] \\
& =  \left(\lambda(\b)  P_{t-}^{\b,\om} [\chi_{t,x}] + 1\right) Z_{t-}\;.
\end{align*}
Hence,
\begin{equation} \label{eq:logZtIntRepresentation}
\ln Z_{t} = \int \om_t(\rmd s \rmd x) \ln \big(1+\lambda P_{s-}^{\b,\om} [\chi_{s,x}]\big).
\end{equation}
Let $g$ be the function $g(u) = u-\ln(1+u)$, which is positive on $(-1,\infty)$. Then, recalling that $\int_{\rd} \chi_{s,x} dx=r^d$, we obtain  Doob's decomposition 
\begin{equation}\label{eq:DoobDecompLogZ}
-\ln W_t = \lambda \nu r^d t - \ln Z_t = M_t + A_t,
\end{equation}
where the martingale $M$ and the increasing process $A$ are given by
\begin{align} 
M_t & = -\int \bar{\om}_t (\rmd s \rmd x) \ln \big(1+\lambda P_{s-}^{\b,\om} [\chi_{s,x}]\big),\\
A_t & = \lambda \nu r^d t - \int_{[0,t]\times \mathbb{R}^d}  \ln \big(1+\lambda P_{s-}^{\b,\om} [\chi_{s,x}]\big) \, \nu \hspace{0.5mm} \rmd s \rmd x \nonumber\\
& = \int_{[0,t]\times \mathbb{R}^d}  g\big(\lambda P_{s-}^{\b,\om} [\chi_{s,x}]\big) \, \nu \hspace{0.5mm} \rmd s \rmd x. \label{eq:2773}
\end{align}
Moreover, the martingale $M$ is square-integrable, and its bracket is given by
\begin{equation}
\nn
\langle M \rangle_t = \int_{[0,t]\times \mathbb{R}^d} \big[ \ln \big(1+\lambda P_{s-}^{\b,\om} [\chi_{s,x}]\big)\big]^2 \, \nu \hspace{0.5mm} \rmd s \rmd x.
\end{equation}

\section{The replica overlap and quenched overlaps}
Things being clear from the context, we will use the same notation $|A|$ to denote the Lebesgue measure of Borel subset $A$ of $\R^d$ or $\R_+ \times \R^d$.
\begin{definition}
For any two paths $B$ and $\tilde{B}$, we define the \textbf{replica overlap} $R_t(B,\tilde{B})$ as the mean volume overlap of the two tubes around $B$ and $\tilde{B}$ in time $[0,t]$:
\begin{equation}\label{eq:ov1}
R_t(B,\tilde{B})=\frac{1}{tr^d} | V_t(B)\cap V_t(\tilde{B})|.
\end{equation}
In a similar way, we define the \textbf{quenched overlaps} $I_t$ and $J_t$ 
as\footnote{The product measure ${\gibbs}^{ \otimes 2} = \gibbs \otimes \gibbs$ makes the 2 replicas $B, \tilde B$ independent polymer paths sharing the same environment $\omega$.}
\begin{align} \label{eq:ov2}
I_t &= \frac {1}{r^{d}} 
{\gibbs}^{ \otimes 2} \big[ |U(B_t) \cap U(\tilde{B}_t)| \big]
,\\ \label{eq:ov3}
J_t &=   \frac {1}{tr^{d}}  
{\gibbs}^{ \otimes 2} \big[ |V_t(B) \cap V_t(\tilde{B})| \big],
\end{align}
\end{definition}

The  variable $I_t$ stands for the expected volume of overlap around the endpoints of two  independent polymer paths, while $J_t$ is the expected volume of overlap during the time interval $[0,t]$.
Note that both $\int_0^t I_s \rmd s$ and $t J_t$ represent an expected volume of overlap in time $[0,t]$, but they will emerge from different circumstances. 
Similar to  $\chi_{s,x}= {\mathbf 1} \{x \in U(B_s) \}$ from definition \eqref{def:chitx}, we write for short 
\begin{equation}\nn
\tilde \chi_{s,x}= {\mathbf 1} \{x \in U(\tilde B_s) \} .
\end{equation}
Writing 
$$
|U(B_t) \cap U(\tilde{B}_t)|  = \int_{\rd}  \chi_{t,x}\tilde \chi_{t,x} dx,
$$
we derive two useful formulas:
\begin{equation}\label{eq:useful}
I_t=\frac{1}{r^{d}}  \int_{\rd} \gibbs (\chi_{t,x})^2 dx\;,\qquad J_t = \frac{1}{r^{d}} \int_{\rd} dx \frac 1t  \int_0^t \gibbs (\chi_{s,x})^2 ds\;.
\end{equation}

\medskip

For better comparisons, we have normalized all quantities in \eqref{eq:ov1}, \eqref{eq:ov2} and \eqref{eq:ov3}  in such a way that
$$
0 \leq R_t, I_t, J_t \leq 1 \;,
$$
so that we can -- and we will -- view each of them as a \emph{localization index}:
\begin{itemize}
\item   $R_t$ close to 1 means that the two fixed paths $B, \tilde B$ are close on the interval $[0,t]$;
\item $I_t$  close to 1 means that the endpoints of two independent samples of the polymer measure are typically close one from the other;
\item $J_t$ close to 1 means that the paths of two independent polymers are close all along the time interval.
\end{itemize}
The second case corresponds to \emph{endpoint localization} whereas the third one is  \emph{path localization}. 
Mathematically, the quantity $I_t$ appears via It\^o's calculus (stochastic differentiation) and $J_t$ via Malliavin calculus (integration by parts).
On the contrary, small values of these indices correspond to
absence of localization: it means that the polymer spreads more or less uniformly in space without particular preference.

%

\begin{remark} 
When $\b = 0$, the Gibbs measure $\gibbs$ reduces to Wiener measure $P$, so that $J_t = \int_0^t I_s \rmd s = t P^{\otimes 2} [R_t(B,\tilde{B})]$, and
\[P^{\otimes 2} [R_t(B,\tilde{B})] = \frac{1}{t} \int_0^t P^{\otimes 2} \big[|U(B_s)\cap U(\tilde{B}_s)|\big] \rmd s \leq \frac{r^d}{t} \int_0^t P\left(B_{2s}\in U(0)\right) \rmd s \underset{t\to 0}{\longrightarrow} 0 .\] Thus, \[\lim_{t\to\infty}  J_t = \lim_{t\to\infty} \frac{1}{t} \int_0^t I_s \rmd s = 0,\] 
This indicates that, in absence of interaction with the inhomogeneous medium,  there is no localization of the Brownian path.
\end{remark} 
We now come to the core of the section: localization results.
\section{Endpoint localization}
A naive prediction is  that, for small $\b$,  the polymer is a small perturbation of Brownian motion for small $\b$, with a comparable behavior, whereas  for large $\b$ localization takes place and the limits are nonzero.
A first theorem shows that this is the case for the quantity $I_t$:
\begin{theorem} \label{th:loca} The following equivalence holds for $\b \neq 0$:
\begin{equation} \label{eq:intIsEqSD}
W_\infty = 0 \iff \int_0^\infty I_s \, \rmd s = \infty, \quad \IP\text{-a.s.}
\end{equation}
In particular, the above integral is a.s. finite for $\b \in (\bar{\b}_c^-, \bar{\b}_c^+)$, and a.s. infinite for $\b < \bar{\b}_c^-$ or $\b > \bar{\b}_c^+$.
Moreover, we have: 
\begin{align} 
&\lim_{t\to\infty} \frac{1}{t} \int_0^t I_s \, \rmd s = 0 \  \textup{ if } \b \in [\b_c^-,\b_c^+]\cap \mathbb{R},\label{eq:AsymptIntIs1}\\
&\liminf_{t\to\infty} \frac{1}{t} \int_0^t I_s \, \rmd s > 0 \  \textup{ if } \b \in \mathbb{R}\setminus [\b_c^-,\b_c^+]. \label{eq:AsymptIntIs2}
\end{align}
\end{theorem}
\begin{remark} Note that  $\int_0^\infty I_s \rmd s$ is a.s. finite or a.s. infinite,  by the dichotomy between strong and weak disorder. Similarly, 
strict positivity of $\,\liminf_{t\to\infty} t^{-1} \int_0^t I_s ds$  is equivalent to low temperature region $\psi(\b, \nu)>0$ in \eqref{eq:highlow}.
In fact, we will prove
in \eqref{eq:encadrementAtIs}--\eqref{eq:freeEnergyOverAt}
 that under strong disorder, there exist $c_1,c_2 \in (0,\infty)$, such that
\begin{equation} \label{eq:EncadrementFreeEnergy}
c_1 \int_0^t I_s \, \rmd s \leq  - \ln W_t \leq c_2 \int_0^t I_s \, \rmd s, \quad \text{for large t}, \quad \IP{-a.s.}
\end{equation}
\end{remark}
\begin{proof} (Theorem \ref{th:loca})
First observe that one can easily derive \eqref{eq:AsymptIntIs1} and \eqref{eq:AsymptIntIs2} from \eqref{eq:intIsEqSD} and \eqref{eq:EncadrementFreeEnergy}.

To show \eqref{eq:EncadrementFreeEnergy}, we will relate $\int_0^t I_s\rmd s$ to the variables $M_t$ and $A_t$ of the Doob decomposition \eqref{eq:DoobDecompLogZ} as follows. From 
\eqref{eq:useful}, we have
\begin{equation} \label{eq:IntIsEqGibbsSquared} \int_{[0,t]\times\mathbb{R}^d} \left(P_s^{\b,\om}[\chi_{s,x}]\right)^2 \rmd s \rmd x = 
 r^d \int_0^t I_s \, \rmd s .
\end{equation}
Then, looking at the behavior of $g(u)$ in \eqref{eq:2773} 
and $\ln(1+u)$ around $0$, it is clear that there are two constants $c_1,c_2>0$, depending only on $\b$ and $\nu$, such that
\[ c_1 u^2 \leq \nu g(u) \leq c_2 u^2, \quad \quad \nu\ln(1+u)^2 \leq c_2 u^2,\]
for all $u$ in $[0,\lambda]$ when $\b\geq 0$, (resp. $[\lambda,0]$ when $\b\leq 0$). Together with \eqref{eq:IntIsEqGibbsSquared}, this implies that 
\begin{align}\label{eq:encadrementAtIs}
c_1 \int_0^t I_s \rmd s \leq & A_t \leq  \, c_2 \int_0^t I_s \rmd s,\\
\langle & M \rangle_t \leq  \, c_2 \int_0^t I_s \rmd s.\label{eq:majCrochetMtIs}
\end{align}

We then recall two results about martingales. (These facts for the discrete martingales are standard (e.g. \cite[p. 255, (4.9),(4.10)]{Dur95}. It
is not difficult to adapt the proof for the discrete setting to our case.)
Let $\varepsilon >0$, then
\begin{align} &\langle M \rangle_\infty < \infty \implies (M_t)_{t\geq 0} \text{ converges\; a.s.}, \label{eq:BracketImpliesCV}\\  
&\langle M \rangle_\infty = \infty \implies \lim_{t\to\infty} M_t / \langle M \rangle_t^{\frac{1+\varepsilon}{2}} = 0 \quad {\rm a.s.}. \label{eq:InfiniteBracket}
\end{align}
Consequently, we get that $\IP$-almost surely,
\begin{align} \nn
\int_0^\infty I_s \rmd s < \infty & \iff A_\infty < \infty,  \langle M \rangle_\infty < \infty\\ \nn
& \implies  A_\infty < \infty, \, \lim_{t\to\infty} M_t \text{ exists and is finite}\\ \nn
& \implies W_\infty > 0,
\end{align}
where the last implication comes from the Doob decomposition \eqref{eq:DoobDecompLogZ}. By contraposition, this proves the first implication of \eqref{eq:intIsEqSD}.
The reverse implication will follow from the arguments below.

The next step is to show \eqref{eq:EncadrementFreeEnergy}, so we now suppose that we are in the strong disorder setting.
Using \eqref{eq:encadrementAtIs}, we see that it is enough to show that
\begin{equation} \label{eq:freeEnergyOverAt}
\lim_{t\to\infty} \frac{- \ln W_t}{A_t} = 1, \quad \IP\text{-a.s.}
\end{equation}
or equivalently by the Doob decomposition, that 
\begin{equation}\label{eq:MtOnAtCV0}
\lim_{t \to \infty} \frac{M_t}{A_t} = 0.
\end{equation}

As we just proved, the strong disorder implies that $\int_0^\infty I_s \rmd s = \infty$, which in turn implies that  $A_\infty = \infty$ by \eqref{eq:encadrementAtIs}. Thus, in the case that $\langle M \rangle_\infty < \infty$, the martingale $M_t$ converges so that the condition \eqref{eq:MtOnAtCV0} directly holds. When $\langle M \rangle_\infty = \infty$, this is still true as
\[\frac{M_t}{A_t}=\frac{M_t}{\langle M \rangle_t} \frac{\langle M \rangle_t}{A_t} \underset{t\to\infty}{\longrightarrow} 0,\]
by (\ref{eq:encadrementAtIs}, \ref{eq:majCrochetMtIs},  \ref{eq:InfiniteBracket}).

Finally, what is left to demonstrate is that $\IP$-almost surely   $W_\infty = 0$ 
on the event $\int_0^\infty I_s \rmd s=\infty$. In fact, we showed that this event implies both the limit \eqref{eq:freeEnergyOverAt} and $A_\infty = \infty$, so that $W_t \to 0$.
\end{proof}

{\bf Endpoint localization}: 
As in the discrete case, we can interpret the results from the present subsection, in terms
of localization for the path. Indeed, it is proven in Sect. 8 of \cite{CY05}, that for some constant $c_1$, 
\begin{equation} \label{eq:overlap-favorite}
c_1 \sup_{y \in \rd} \gibbs [ B_s \in U(y)]^2 \leq { \gibbs}^{\otimes 2} \left[ \big \vert U(B_s) \bigcap U(\tilde B_s)\big \vert \right]
\leq \sup_{y \in \rd} \gibbs [ B_s \in U(y)]
\end{equation}
(in fact, the inequality on the right is trivial, and the one on the left is the combination of \eqref{eq:num-1} and \eqref{eq:num-2}).

The maximum appearing in the above bounds should be viewed as the probability of the favorite location for $B_s$, under the polymer measure $\gibbs$; for $s=t$, the supremum is called the probability of the \emph{favorite endpoint}, and the maximizing $y$ is
the \emph{location of the favorite endpoint}.
Both Theorem \ref{th:loca} and Theorem \ref{th:asymptOfJt} are precise statements that the polymer localizes in the strong disorder
regime in a few specific corridors of width ${\cal O}(1)$, but spreads out in a diffuse way in the
weak disorder regime. If $\psi(\b,\nu)>0$, the Cesaro-limit of  probability of the {favorite endpoint} is strictly positive.

\section{Favorite path and path localization}

%
%
Recall that the excess free energy $\psi$ from 
\eqref{def:excessfreeenergy} is the difference of a smooth function and a convex function. Hence its right-derivative, resp. left-derivative, 
\begin{equation}\nn
 \left( \frac{\partial \psi}{\partial \b}\right)_+ (\b,\nu) = \lim_{\b' \searrow \b} \frac{\psi(\b',\nu)-\psi(\b,\nu)}{\b'-\b}, \qquad
 \left(   \frac{\partial \psi}{\partial \b}\right)_- (\b,\nu) = \lim_{\b' \nearrow \b} \frac{\psi(\b',\nu)-\psi(\b,\nu)}{\b'-\b},
\end{equation}
exists for all $\b$ and all $\nu$, and satisfy 
$ \left( \frac{\partial \psi}{\partial \b}\right)_+ (\b,\nu) \leq  \left(\frac{\partial \psi}{\partial \b}\right)_- (\b,\nu)$. For the same reason as above, $\psi(\cdot,\nu)$ is differentiable except on a set
which is at most countable, and we can write
 \begin{equation}\nn
 \left(  \frac{\partial \psi}{\partial \b}\right)_+ (\b,\nu) = \lim_{\b'\geq \b } 
  \frac{\partial \psi}{\partial \b} (\b',\nu) ,
\end{equation}
where the limit is over differentiability points $\b'$ tending to $\b$ by larger values. A similar statement holds for the left-derivative.
For further use, we note that for all fixed $\nu$, $\psi(\cdot, \nu)$ is absolutely continuous again for the same reason as above.

Now, we turn to the properties of the replica overlap $J_t$. The key fact is the following proposition:
\begin{proposition} \label{prop:boundJt}
There exist two constants $c_1,c_2 \in (0,\infty)$, depending only on $\b$ and $\nu$, such that
\begin{align} \label{eq:bound_PJt1}
\forall \b>0, & \quad c_1 \,  \left(  \frac{\partial \psi}{\partial \b}\right)_+\leq \liminf_{t\to\infty} \,  \IP[J_t] \leq \limsup_{t\to\infty} \,  \IP[J_t] \leq c_2\,\left(  \frac{\partial \psi}{\partial \b}\right)_-, \\
\label{eq:bound_PJt2} \forall \b<0, & \quad -c_1\,  \left(  \frac{\partial \psi}{\partial \b}\right)_-\leq \liminf_{t\to\infty} \,  \IP[J_t] \leq \limsup_{t\to\infty} \, \IP[J_t] \leq -c_2 \,
 \left(  \frac{\partial \psi}{\partial \b}\right)_+. 
\end{align}
\end{proposition}
\begin{proof} It is not difficult to see from the definition that $t J_t = \iint_{[0,t]\times \mathbb{R}^d} \gibbs[\chi_{s,x}]^2 \rmd s \rmd x$. Hence, using equation \eqref{eq:psi'} and the fact that
$e^{-|\b|} \leq 1+\lambda \gibbs[\chi_{s,x}] \leq e^{|\b|}$, we obtain that
\begin{equation} \label{eq:bound_dpsit}
\lambda \nu e^{\b-|\b|} t \IP[J_t] \leq t \frac{\partial}{\partial \b}\psi_t(\b,\nu) \leq \nu \lambda e^{\b + |\b|} t \IP[J_t].
\end{equation}
Moreover, the excess free energy writes $\psi(\b,\nu)=\nu \lambda(\b)r^d- p(\b, \nu)$, where $p(\b,\nu)$ is a convex function, defined as the limit, for $t\to\infty$, of the convex functions
$p_t(\b,\nu) = \frac{1}{t}\IP[\ln Z_t].$
By convexity properties, we know that
\[\left(  \frac{\partial p}{\partial \b}\right)_-\leq \liminf_{t\to \infty} \frac{\partial p_t}{\partial \b } \leq \limsup_{t\to \infty} \frac{\partial p_t}{\partial \b } \leq
\left(  \frac{\partial p}{\partial \b}\right)_+,
\]
which in turns implies that
\begin{equation} \left(  \frac{\partial \psi}{\partial \b}\right)_+\leq \liminf_{t\to \infty} \frac{\partial \psi_t}{\partial \b } \leq \limsup_{t\to \infty} \frac{\partial \psi_t}{\partial \b } \leq
\left(  \frac{\partial \psi}{\partial \b}\right)_-
.\end{equation}
The proposition is then a consequence of \eqref{eq:bound_dpsit} and these last inequalities.
\end{proof}
With this proposition, we can give a characterization of the critical values $\b_c^\pm$, in terms of the asymptotics of the overlap:
\begin{theorem} \label{th:asymptOfJt}
For all $\b \in [\b_c^-,\b_c^+]\cap \mathbb{R}$,
\begin{equation}
\lim_{t\to\infty}  \IP[J_t] = 
0.
\end{equation}
Furthermore,
\begin{align} 
\b_c^+ = \sup \{\b' \geq 0 : \forall \b \in [0,\b'], \, \lim_{t\to\infty} \IP[J_t] = 0\}= \inf\{ \b>0: \liminf_{t\to\infty} \IP[J_t] > 0\},\\
\b_c^- = \inf \{\b' \leq 0 : \forall \b \in [\b',0], \, \lim_{t\to\infty}  \IP[J_t] = 0\}= \sup\{ \b<0: \liminf_{t\to\infty} \IP[J_t] > 0\}.
\end{align}
\end{theorem}
\begin{proof} We will focus on the $\b\geq 0$ case, but the same arguments can be applied to the $\b \leq 0$ case. Define
\[ \delta_c^+ = \sup \{\b' \geq 0 : \forall \b \in [0,\b'], \, \lim_{t\to\infty}  \IP[J_t] = 0\},\]
so what we need to show in particular is $\b_c^+ = \delta_c^+$.

To prove the first claim of the theorem, from which $\b_c^+ \leq \delta_c^+$ follows directly, it is enough, using \eqref{eq:bound_PJt1}, to verify that
\begin{equation}
\forall \b \leq \b_c^+, \quad \left( \frac{\partial \psi}{\partial \b}\right)_-(\b,\nu) = 0.
\end{equation}
This property is true when $\b\in [0,\b_c^+)$, as $\psi$ is constant and set to $0$ in this interval. To prove that it extends to $\b_c^+$ if $\b_c^+ < \infty$, observe that $\psi$ is minimal at $\b_c^+$, so that
\[ \left( \frac{\partial \psi}{\partial \b}\right)_- (\b_c^+,\nu) \leq 0 \leq  \left(\frac{\partial \psi}{\partial \b}\right)_+ (\b_c^+,\nu).\]
As we saw earlier that $\frac{\partial \psi}{\partial \b_+} \leq  \frac{\partial \psi}{\partial \b_-}$ always holds, we finally get that $\frac{\partial \psi}{\partial \b_- }(\b_c^+) = \frac{\partial \psi}{\partial \b_+ }(\b_c^+)= 0$.


We now prove $\b_c^+ \geq \delta_c^+$. Let $\b > \b_c^+$ be finite,
so that, by definition, $\psi(\b,\nu)>0$. As $\psi$ is absolutely continuous 
with $\psi(0,\nu) = 0$,  one can write
\[\psi(\b,\nu) = \int_0^\b \frac{\partial \psi}{\partial \b_+ }(\b',\nu)\rmd \b' > 0,\]
which implies that there exists some $\b'\leq \b$ such that $\frac{\partial \psi}{\partial \b+ }(\b',\nu) > 0$. By equation \eqref{eq:bound_PJt1}, we get that $\liminf_{t\to\infty} \frac{1}{t} \IP[J_t(\b)] > 0$, hence $\b \geq \delta_c^+$. As it is true for all $\b > \beta_c^+$, we obtain that $\b_c^+ \geq \delta_c^+$.
\end{proof}

\medskip
\paragraph{The favorite path.}
Let $\mathcal{M}$ be the set of integer-valued Radon measures on $\mathbb{R}_+\times \mathbb{R}^d$, equipped with the sigma-field $\mathcal G$ generated by the variables $\om(A)$, $A\in \mathbb{R}_+\times \mathbb{R}^d$, so that we will consider $\om$ as a process of the probability space $(\mathcal{M},\mathcal{G},\IP)$.
It is possible to define, for all fixed time horizon $t>0$, a measurable function
\begin{equation} \begin{array}{cccc}
\Yt:&[0,t]\times\mathcal{M} & \to & \mathbb{R}^d\\
&(s,\om) & \mapsto & \Yt_s,
\end{array}
\end{equation}
which satisfies the property that, $\IP$-almost surely,
\begin{equation}
\forall s\in[0,t],\quad \gibbs \left(B_s \in U\left(\Yt_s\right)\right)=\max_{x\in\mathbb{R}^d} \gibbs \big(B_s \in U(x) \big).
\end{equation}
The reader may refer to \cite{CYBMPO2} for a proper (and rather technical) definition of $\Yt$. 

Here, the path $s\to \Yt_s$ stands for the "optimal path" or the "favorite path" of the polymer, although this path is neither necessarily continuous, nor necessarily unique.
Similarly to what we have done previously, we define the overlap with the favorite path $R_t^*$ as the fraction of time any path $B$ stays next to the favorite path:
\begin{equation}
R_t^* = R_t^*(B,\om) = \frac{1}{t} \int_0^t \mathbf{1}_{B_s\in \, U(\Yt_s)} \rmd s.
\end{equation}

As discussed before, a question of interest is the asymptotic behavior, as $t\to\infty$, of $R_t$ and of $R_t^*$, and we will see in Theorem \ref{th:localization} that they are related. In particular, we are interested in determining the regions were one can prove positivity in the limit of these quantities, which can be seen as localization properties of the polymer. 


Recall the notations
\begin{equation} \nn
{\cal D} = \{ (\b,\nu): \psi(\b, \nu)=0 \}, \quad {\cal L}  = \{ (\b,\nu): \psi(\b, \nu)>0\},
\end{equation}
of the high and low temperature regions, which are delimited by the critical lines $\b_c^-(\nu)$ and $\b_c^+(\nu)$ (cf. Definition \ref{defdef..}).
We saw in theorem \ref{th:asymptOfJt} that in the ${\cal D}$ region, $\lim_{t\to\infty} \IP[J_t] = 0$. On the other hand, Proposition \ref{prop:boundJt} tells us that the limit inferior of $\IP[J_t]$ is always positive in the region $\mathcal{L}'$, where
\begin{equation} \label{eq:Lregion}
\mathcal{L}' = \left\{ \b > 0,\nu > 0 :  \left(  \frac{\partial \psi}{\partial \b}\right)_+ > 0 \right\}\cup\left\{ \b < 0, \nu > 0 :  \left(  \frac{\partial \psi}{\partial \b}\right)_- < 0 \right\}.
\end{equation}
From the preceeding considerations, we know that
\begin{equation}
\mathcal{L}' \subset \mathcal{L}\;,
\end{equation}
and a still open question is whether 
$\mathcal{L}' = \mathcal{L}$ or not.

\begin{remark}
It is a direct consequence of the monotonicity of $\psi$ (point \ref{eq:p3} of proposition \ref{prop:p}), that the inequalities on the derivatives of $\psi$ appearing in \eqref{eq:Lregion}, once replaced by large inequalities, are always verified:
$$ \b<0 \implies  \left(  \frac{\partial \psi}{\partial \b}\right)_- \leq 0,
\quad
 \b>0 \implies  \left(  \frac{\partial \psi}{\partial \b}\right)_+ \geq 0.$$
\end{remark}

We first state some results about the localized region:
\begin{proposition}
\begin{enumerate}[label=(\roman*)]
\item For any fixed $\nu>0$ and for large enough positive $\b$, $(\b,\nu) \in \mathcal{L}'$.
\item For all $(\b,\nu) \in \mathcal{L}'$, $\liminf_{t\to\infty} \IP[J_t] > 0$.
\end{enumerate}
\end{proposition}
\begin{proof}
To prove (i), we use a result of \cite[Th. 2.2.2.(b)]{CY05}
where it is shown that there exists a positive constant $C_1=C_1(r,\nu,d)$, such that, for fixed $\nu, r$ and $\b$ large enough,
\[p(\b,\nu) \leq C_1 \lambda^{1/2}.\]
By convexity of $p$ in $\b$, we get that for large enough $\b$,
$$\left(  \frac{\partial p}{\partial \b}\right)_+ \leq p(\b+1,\nu) - p(\b,\nu) 
 < \nu r^d e^\b,
$$
so that
$\left(  \frac{\partial \psi}{\partial \b}\right)_+$ is indeed positive when $\b$ is big enough.

The property (ii) is given by proposition \ref{prop:boundJt}.
\end{proof}
\begin{remark}
We stress on how strong is the above claim (ii). For $(\b,\nu) \in \mathcal{L}'$, it implies that   there exist $C>0, \delta>0$ such that 
$$\liminf_{t\to\infty} \IP  {\gibbs}^{ \otimes 2} \big[ R_t(B, \tilde{B}) \geq \delta \big] \geq C\;.$$
In contrast, if $\b=0$, there is  some $C'>0$ such that
$$
P^{\otimes 2}  \big[ R_t(B, \tilde{B}) \geq \delta \big] \leq e^{-C't}
$$
for all large enough $t$.
\end{remark}
Observe that the properties of $R_t$ and $R_t^*$ are comparable in the following sense:
\begin{theorem} \label{th:localization}
There exists a constant $c=c(d,r)$ in $(0,1)$, such that
\begin{equation} \label{eq:2to1}
c\left(\IP\gibbs \big[R_t^*\big] \right)^2 \leq \IP[J_t] \leq r^d \, \IP \gibbs\big[R_t^*\big].
\end{equation}
In particular, we get that for all $\b\in\mathcal{L}'$,
\begin{equation} \label{eq:localization}
\liminf_{t\to\infty} \IP \gibbs \big[R_t^*\big] \geq r^{-d}\liminf_{t\to\infty} \IP[J_t] >0.
\end{equation}
\end{theorem}
\begin{remark}
Note that equation \eqref{eq:localization}  gives another feature of path localization of the polymer in the $\mathcal{L}'$ region: we can find a "path" depending only on the environment (here, we found that $\Yt$ does the job) such that the {\rm expected proportion of time the random polymer spends in the neighborhood of that "path" 
is bounded away from 0} as $t \to \8$. Under the Gibbs measure, the random polymer sticks to that particular "path". Even though that "path" is not smooth -- in fact, it has long jumps -- it is an interesting object which sumarizes the attractive effect  of the medium.
\end{remark}
\begin{proof}
To prove the first point, it is enough to show that there is a $c\in(0,1)$, such that
\begin{equation} \label{eq:boundOnRt}
c\gibbs\big[R_t^*\big]^2 \leq J_t \leq r^d\gibbs\big[R_t^*\big],
\end{equation}
the proposition being then a simple consequence of Jensen's inequality. 
For the right-hand side inequality of \eqref{eq:boundOnRt}, observe that by Fubini's theorem,
\begin{align*}
\gibbs\otimes\gibbs\big[R_t\big] &=\frac{1}{t}\int_0^t \int \gibbs[\chi_{s,x}]^2 \rmd s \rmd x.\\
& \leq \frac{1}{t}\int_0^t \max_{x\in\mathbb{R}^d} \gibbs\left( B_s \in U(x)\right)\rmd s \times \int_{\mathbb{R}^d} \gibbs\big[\chi_{s,x}\big] \rmd x\\
& = r^d \gibbs\left[ R_t^* \right].
\end{align*}

In order to obtain the left-hand side inequality, let $r_d$ denote the radius of the ball $U(x)$ and $y$ be any point of $\mathbb{R}^d$. By Cauchy-Schwarz's inequality,
\[\left(\int_{\mathrm{B}(y,r_d/2)}  \gibbs\big(B_s \in U(z) \big) \rmd z \right)^2 \leq \big|\mathrm{B}(y,r_d/2)\big| \int_{\mathbb{R}^d} \gibbs\big(B_s \in U(z)\big)^2,\]
and since for every $z$ in $\mathrm{B}(y,r_d/2)$, the ball $\mathrm{B}(y,r_d/2)$ is included in $U(z)$, this inequality leads to the following:
\begin{align}\nn
\int_{\mathbb{R}^d} \gibbs\big(B_s \in U(z)\big)^2 \rmd z &\geq \frac{2^d}{r^d} \left(\int_{\mathrm{B}(y,r_d/2)}  \gibbs\big(B_s \in \mathrm{B}(y,r_d/2) \big) \rmd z \right)^2\\
&= \frac{r^d}{2^d} \, \gibbs\big(B_s \in \mathrm{B}(y,r_d/2)\big)^2. \label{eq:num-1}
\end{align}

Now, let $c'=c'(d)$ be the minimal number of copies of $\mathrm{B}(y,r_d/2)$ necessary to cover $U(y)$. Then, by additivity of $\gibbs$,
\begin{equation} \label{eq:num-2}
\max_{y\in\mathbb{R}^d} \gibbs\big(B_s\in U(y)\big) \leq c' \max_{y\in\mathbb{R}^d} \gibbs\big(B_s\in \mathrm{B}(y,r_d/2)\big).
\end{equation}
Putting things together and integrating on $[0,t]$, we finally get that
\[ c\,\frac{1}{t} \int_0^t \max_{y\in\mathbb{R}^d} \gibbs\big(B_s\in U(y)\big)^2 \rmd s \leq  \frac{1}{t}\int_0^t \int \gibbs[\chi_{s,x}]^2 \rmd s \rmd x,\]
where $c=(c')^{-2} {r^d}/{2^d}$, from which the left-hand side inequality of \eqref{eq:boundOnRt} can be obtained by applying Jensen's inequality with probability measure $\rmd s /t$ on [0,t].

The second point of the theorem is then a consequence of the first point and proposition \ref{prop:boundJt}.
\end{proof}
\begin{remark}
Formulas like \eqref{eq:2to1} and \eqref{eq:boundOnRt} can be called 2-to-1 formulas since they relate quenched expectations for two independent polymers to
expectations for only one polymer (and involving the optimal path). As we have seen from the computations, stochastic analysis brings in second moments, involving 2 replicas of the polymer path. Then, using such formulas, the information is reduced to one polymer path interacting with the favorite path.
\end{remark}

\chapter{Formulas for variance and concentration}
In this chapter, we introduce the  critical exponents of the model and relations between them. The starting points are precise formulas for  fluctuations  of the partition function (variance and large deviations).
\section{The critical exponents}
 There are different ways of defining the critical exponents, see for example \cite{Chatterjee-scaling,LiceaNewPiza}. We will not enter the finest details, and we stay at an intuitive level.
 Although it is not clear that these definitions are all equivalent, the main idea is that the critical exponents are two reals $\xi^{\perp}$ and $\xi^{\parallel}$ such that
\begin{equation}
\sup_{0\leq s \leq t} |B_s| \approx t^{\xi^{\perp}(d)} \quad \text{and} \quad \ln Z_t - \IP[\ln Z_t] \approx t^{\xi^{\parallel}(d)} \quad \text{as} \  t\to\infty.
\end{equation}
The "wandering exponent" $\xi^{\perp}$ is the exponent for the asympotic transversal (or "perpendicular") fluctuations of the path, with respect to the time axis.
 The polymer is said to be \emph{diffusive} when $\xi^{\perp} = 1/2$ (as for the brownian motion), and it is said to be \emph{super-diffusive} when $\xi^{\perp} > 1/2$. One of the conjectures in polymers is that diffusivity should occure in weak disorder, while super-diffusivity should take place in the strong disorder setting. The number $\xi^{\parallel}$ denotes the critical exponent for the longitudinal fluctuation of the free energy.

The study of these exponents goes beyond the polymer framework. 
The reason is that they are expected to take the same value in many different  statistical physics models describing growth phenomena.  In dimension $d=1$ 
this family is called the KPZ universality class (see Section \ref{sec:IR}). It is conjectured in the physics literature \cite{KrugSpohn} that the two exponents should depend on one other, in the way that
\begin{equation}
\xi^{\parallel}(d) = 2 \xi^{\perp}(d) - 1, \forall d\geq 1
\end{equation}
Under a certain definition of the exponents, the relation was proved by Chaterjee \cite{Chatterjee-scaling} for first-passage percolation. Auffinger and Damron were able to simplify Chatterjee's proof and extend the result to directed polymers \cite{AuffingerDamron-sc,AuffingerDamron13}.

In dimension $d=1$, it is conjectured that $\xi^{\perp}=2/3$ and $\xi^{\parallel}=1/3$ for any positive $\b$. For now, this has only been proven for solvable models of polymers: Sepp\"al\"ainen's discrete log-gamma polymer \cite{Sepp12}, O'Connell-Yor semi-discrete polymer \cite{MorenoSeppValko,SeppValko10}, and also for the the KPZ polymer \cite{BalaszQuastelSeppalainen11}.

In dimension $d \geq 2$, essentially nothing is known. Let's simply mention the (rough) bounds
$$
0 \leq \xi^\parallel \leq 1/2\,,\quad 1/2 \leq \xi^\perp \leq 3/4\,,
$$
where the last one will be proved in section \ref{ch:CM}.

A way to approach $\xi^{\parallel}$ is to consider the variance of  $\ln Z_t$.
In what follows, we give a formula to express the variance of $\ln Z_t$ in terms of a stochastic integral, which is obtained through a Clark-Ocone type martingale representation.

\section{The Clark-Ocone representation}
It is a consequence of It\^o's work on iterated stochastic integrals \cite{ito1951multiple} any that square-integrable functionals of the Brownian motion can be written as the sum of a constant and an It\^o integral. In \cite{clark1970representation}, Clark extended this result to a wider range of functionals, and showed that any martingale that is measurable with respect to the Brownian motion filtration, could be represented as a stochastic integral martingale. Clark was also able to compute the integrand of the representation, for a special class of functionals. Ocone then showed \cite{ocone1984malliavin} that this computation was linked to Malliavin's calculus, and generalised this idea to a larger class of functionals.

Such representations - called Clark-Ocone representations - also exist in the framework of functionals of a Poisson processes. 
Denote by $\om_{s-}$ the restriction of $\om$ on $[0,s)\times\mathbb{R}^d$, and consider the derivative operator
\begin{equation}
D_{(s,x)}F(\om) := F(\om + \delta_{s,x}) - F(\om).
\end{equation}
We have:
\begin{theorem} \textup{\cite[theorem 3.1]{LAST20111588}}
Let $F=F(\om)$ be a functional of the Poisson process, such that $\IP[F^2]<\infty$. Then,
\begin{equation}
\IP \int \IP[D_{(s,x)}F(\om)|\om_{s-}]^2 \rmd s \rmd x < \infty,
\end{equation}
and we have for all $u\geq 0$, that $\IP$-a.s.
\begin{equation}
\IP[F(\om)|\om_u] = \IP[F(\om)] + \int_{[0,u]\times \mathbb{R}^d} \IP[D_{(s,x)} F(\om) | \om_{s-}] \bar \om (\rmd s \rmd x).
\end{equation}
\end{theorem}
This proves that the square integrable martingale $(\IP[F(\om)|\om_u])_{u\geq 0}$ admits a stochastic integral martingale representation, with predictable integrand $\IP[D_{(s,x)} F(\om) | \om_{s-}]$.

\section{The variance formula}
To lighten the writing, we will denote by $\mathcal{G}_{s-}$ the sigma-field generated by $\om_{s-}$ and $\IP^{\mathcal{G}_{s-}}$ will stand for the expectation knowing $\mathcal{G}_{s-}$.

Using Jensen's inequality and Tonelli's theorem, it is easy to check that $\ln Z_t$ is a square integrable function of $\om$. Hence, the process $\big(\IP[\ln Z_t|\om_u]\big)_{u\in[0,t]}$ is a martingale which admits a Clark-Ocone type representation:
\begin{equation} \label{eq:ClarKRepresentationLogZt}
\IP[\ln Z_t|\om_u] = \IP\left[\ln Z_t\right] + \int_{[0,u]\times\mathbb{R}^d} \IP^{\mathcal{G}_{s-}}\left[D_{(s,x)}\ln Z_t\right] \bar{\om}(\rmd s \rmd x),
\end{equation}
where
\[D_{(s,x)}F(\om) = \ln \frac{P\left[e^{\b\om(V_t)}e^{\b \chi_{s,x}}\right]}{Z_t}  = \ln \left(1+\lambda \gibbs[\chi_{s,x}]\right).\]
As a consequence, one can express the variance of $\ln Z_t$ via \eqref{eq:ClarKRepresentationLogZt}, using the formula for the variance of a Poisson integral \eqref{eq:MartingaleVarianceFormula}, and find that
\begin{equation} \label{eq:varFormula}
\mathrm{Var} (\ln Z_t) = \IP \int_{[0,t]\times\mathbb{R}^d} \IP^{\mathcal{G}_{s-}}\left[\ln \left(1+\lambda \gibbs[\chi_{s,x}]\right)\right]^2 \nu \hspace{0.5mm} \rmd s \rmd x.
\end{equation}

This variance formula leads us to the following theorem:
\begin{theorem} \label{th:var}
\begin{enumerate}[label=(\roman*)]
  \item The following lower and upper bounds on the variance hold:
  \begin{align}
\mathrm{Var} (\ln Z_t) &\geq c_-^2 \IP \int_{[0,t]\times\mathbb{R}^d} \IP^{\mathcal{G}_{s-}}\big[ \gibbs [\chi_{s,x}] \big]^2 \, \nu \hspace{0.5mm} \rmd s \rmd x,\\
\mathrm{Var} (\ln Z_t) &\leq c_+^2 \IP \int_{[0,t]\times\mathbb{R}^d} \IP^{\mathcal{G}_{s-}}\big[ \gibbs [\chi_{s,x}] \big]^2 \, \nu \hspace{0.5mm} \rmd s \rmd x, \label{eq:IneqVar}
  \end{align}
where $c_- = 1 - e^{-|\b|}$ and $c_+ = e^{|\b|}-1$.

In particular,
  \begin{equation}
  \mathrm{Var} (\ln Z_t) \leq c_+^2 \, t \nu \, \IP[J_t]. \label{eq:var_bound}
  \end{equation} 
  \item Letting $c=\nu c_+^2 \exp(c_+)$, the following concentration estimate holds:
  \begin{equation} \label{eq:concentrationlogZt}
  \IP\left(\big|\ln Z_t - \IP[\ln Z_t]\big| > u\right) \leq 2\exp\left(-\frac{1}{2}(u\wedge\frac{u^2}{ct})\right).
  \end{equation}
\end{enumerate}
\end{theorem}
\begin{remark} Recalling Theorem \ref{th:asymptOfJt}, the inequality \eqref{eq:var_bound} suggests that the variance should be smaller in weak disorder than in strong disorder. It also shows that for all $d\geq 1$, we have $\xi^{\parallel}(d)\leq 1/2$.
\end{remark}
\begin{proof}
The two first bounds on the variance are a consequence of the fact that, for all $u\in [0,1]$, we have 
\begin{equation}\label{eq:LogBound}
c_- u \leq |\ln(1+\lambda u)| \leq c_+ u.\end{equation} Then, apply Jensen's inequality to the conditional expectation in the right-hand side of \eqref{eq:IneqVar} and use Fubini's theorem such that
\begin{align*}
\mathrm{Var} (\ln Z_t) &\leq c_+^2 \IP \int_{[0,t]\times\mathbb{R}^d} \IP^{\mathcal{G}_{s-}}\big[ \gibbs [\chi_{s,x}] \big]^2 \, \nu \hspace{0.5mm} \rmd s \rmd x \\
&\leq c_+^2 \IP \int_{[0,t]\times\mathbb{R}^d} \gibbs [\chi_{s,x}]^2 \, \nu \hspace{0.5mm} \rmd s \rmd x \\
& = c_+^2 t \nu \, \IP[J_t],
\end{align*}
by definition of $J_t$. This completes the proof of (i).
To prove \eqref{eq:concentrationlogZt}, we first denote by $Y_{t,u}$ the mean-zero martingale part appearing in \eqref{eq:ClarKRepresentationLogZt}, i.e.
\[Y_{t,u} := \int_{[0,u]\times\mathbb{R}^d} \IP^{\mathcal{G}_{s-}}\ln \left(1+\lambda \gibbs[\chi_{s,x}]\right) \bar{\om}(\rmd s \rmd x).\]
Then, letting $\varphi(v)=e^v - v - 1$ and $a\in [-1,1]$, we define $(M_{t,u})_{u\in[0,t]}$ as the exponential martingale associated to $(Y_{t,u})_{u\in[0,t]}$:
\[M_{t,u}=\exp\left(aY_{t,u}-\int_{[0,u]\times \mathbb{R}^d}  \varphi\left(a \cdot\IP^{\mathcal{G}_{s-}}\ln \big(1+\lambda \gibbs[\chi_{s,x}]\big)\right) \, \nu \hspace{0.5mm} \rmd s \rmd x\right).\]
By \eqref{eq:LogBound} and the observations that $\chi$ is less than $1$ and that $|\varphi(v)|\leq e^{|v|} v^2/2$ for all $v$, we have for $a\in[-1,1]$
\begin{align*}\left| \int_{[0,t]\times \mathbb{R}^d}  \varphi\left(a \cdot\IP^{\mathcal{G}_{s-}}\ln \left(1+\lambda \gibbs[\chi_{s,x}]\right)\right) \, \nu \, \rmd s \rmd x \right| 
&\leq e^{c_+} \frac{c_+^2 a^2}{2} \int_{[0,t]\times\mathbb{R}^d} \IP^{\mathcal{G}_{s-}}\big[ \gibbs [\chi_{s,x}] \big]^2 \, \nu \hspace{0.5mm} \rmd s \rmd x\\
&\leq c \frac{a^2}{2} \int_{[0,t]\times\mathbb{R}^d} \IP^{\mathcal{G}_{s-}}\big[ \gibbs [\chi_{s,x}] \big] \, \rmd s \rmd x\\
&= c \frac{a^2}{2}t,
\end{align*}
where $c = \nu c_+^2 e^{c_+}$ and where the second inequality was obtained using Jensen's inequality. 

If one denotes by $b_{t,u}$ the integral term in the definition of $M_{t,u}$, we just showed that $b_{t,t} \leq c a^2t/2$, so by Markov's inequality and the martingale property, we obtain that
\begin{equation}
\IP\big(\ln Z_t - \IP[Z_t] > u\big) = \IP\big(M_{t,t} > \exp(au-b_{t,t})\big)
\leq \exp(c a^2t/2 - au).
\end{equation}
This implies \eqref{eq:concentrationlogZt} after minimizing the bound for $a\in[-1,1]$, and repeating the same procedure for the lower deviation.
\end{proof}

From the concentration estimate, one can derive the following almost sure behavior:
\begin{corollary} \label{cor:var1}
For all $\varepsilon>0$ and as $t\to \infty$,
\begin{equation} \label{eq:OrderOfCenteredlogZt}
\ln Z_t - \IP[\ln Z_t] = \mathcal{O}\left(t^{\frac{1+\varepsilon}{2}}\right), \quad \IP\text{-a.s.} 
\end{equation}
\end{corollary}
\begin{proof}
Equation \eqref{eq:concentrationlogZt} implies that for large enough $k\in\mathbb{N}$,
\begin{equation}
\IP\left[|\ln Z_k - \IP \ln Z_k|> k^{\frac{1+\varepsilon}{2}}\right] \leq 2\exp\big(-\frac{t^{\varepsilon}}{2c}\big),
\end{equation}
which is summable. By Borel-Cantelli lemma, we obtain that,  $\IP$-almost surely,
$$|\ln Z_k - \IP[\ln Z_k]| \leq k^{\frac{1+\varepsilon}{2}}\qquad {\rm for\ } k {\rm \ large\ enough.}
$$ 
To extend this to any $t\geq 0$ and prove \eqref{eq:OrderOfCenteredlogZt}, it suffices to apply the next lemma.
\end{proof}
\begin{lemma} Let $h>0$. For all $0\leq s \leq h$,
\begin{equation} \label{eq:BoundOnZt+s-Zt}
-c_+ \delta_t(h) \leq \ln Z_{t+s} - \ln Z_t \leq c_+\delta_t(h),
\end{equation}
where
\[\delta_t(h)=\int_{[t,t+h]\times\mathbb{R}^d} P_{s-}^{\b,\om}[\chi_{s,x}] \om(\rmd s \rmd x),\]
is such that for any $\varepsilon > 0$,
\begin{equation}\label{eq:deltah}\delta_t(h)=\mathcal{O}\big(t^{\frac{1+\varepsilon}{2}}\big), \quad \IP\text{-a.s.}\end{equation}
\end{lemma}
\begin{proof} We get from the integral writing of $\ln Z_t$ \eqref{eq:logZtIntRepresentation} that
\[\ln Z_{t+s} - \ln Z_t = \int_{[t,t+h]\times\mathbb{R}^d} \om(\rmd s \rmd x) \ln \big(1+\lambda P_{s-}^{\b,\om} [\chi_{s,x}]\big).\]
Hence, \eqref{eq:BoundOnZt+s-Zt} is simply obtained with \eqref{eq:LogBound}.

Now, introduce the martingale
\[M_t = \int_0^t \om(\rmd s \rmd x) P_{s-}^{\b,\om}[\chi_{s,x}] - \nu r^dt,\]
which has bracket ${\langle M \rangle_t}=\nu\int_0^t I_s \rmd s \leq \nu r^d t$. Note that $\delta_t(h) = M_{t+h}-M_t + h$, so that \eqref{eq:deltah} is thus a consequence of the martingale properties
\eqref{eq:BracketImpliesCV} and
\eqref{eq:InfiniteBracket}, since in the case where $\langle M \rangle_\infty$ is infinite, then $|M_t| = o\big(\langle M \rangle_t^{\frac{1+\varepsilon}{2}}\big)$ as $t\to\infty$.
\end{proof}
In order to illustrate the general strategy, we now mention a consequence of theorem \ref{th:var} (i).

\begin{corollary} \label{cor:var2} Let $\b\neq 0$, $\xi$ and $C>0$. There exists a constant $c_1 = c_1(d,C)\in(0,\infty)$, such that 
\begin{align*} \liminf_{t\to\infty}t^{-(1-d\xi)}\mathrm{Var}(\ln Z_t) &\geq c_1 \liminf_{t\to\infty} \inf_{0\leq s \leq t} \left(\IP \gibbs\left(|B_s|\leq C + Ct^\xi\right)\right)^2\\
&\geq c_1 \liminf_{t\to\infty} \left(\IP \gibbs\left(\sup_{0\leq s \leq t} |B_s| \leq C + Ct^\xi\right) \right)^2.
\end{align*}
\end{corollary}
The result suggests that 
$$
\chi(d) \geq \frac{1-d\xi(d)}{2}.
$$
For more details on the result and the proof, we refer the reader to Corollary 2.4.3 in \cite{CY05}.
%

\newcommand{\vpat}{\varphi_{a,t}}

 \chapter{Cameron-Martin transform and applications} \label{ch:CM}
 
 In this section we extensively use the property that the a-priori measure for the polymer path is Wiener measure. A tilt on the polymer path reflects into a shift on the environment. 
 
 \section{Tilting the polymer}
 
 We extend the shear transformation (\ref{eq:tauxif} -- \ref{eq:tau}) to non-linear shifts $\varphi: \R_+ \to \rd$ by defining 
 \begin{equation} \label{eq:htau}
\hat \tau_\varphi f : s \mapsto f(s)+\varphi(s) \;, \quad \hat \tau_\varphi \circ \big(\sum_i \delta_{(t_i, x_i)}\big) = \sum_i \delta_{(t_i, x_i+\varphi( t_i))} \;,
\end{equation}
so that $\hat \tau_\varphi=\tau_\xi$ when $\varphi(t)=t \xi$ and 
that $\hat \tau_\varphi \circ \om$ has same law as $\om$.
\medskip

Let $\varphi \in H_{0, {\rm loc}}^1:=\left\{ \varphi \in {\cal C}(\R_+,\rd); \varphi(0)=0, \dot \varphi \in L^2_{\rm loc} \right\}$  where the dot denotes time derivative.  
Introduce the probability measure on the path space $P^{\varphi}$ which restriction to ${\cal F}_t$ has density relative to $P$ given by
$$
\big(\frac{d P^{\varphi}}{d P}\big)_{| {\cal F}_t}= \exp \left\{ \int_0^t \dot \varphi (s) dB(s) - \frac{1}{2} \int_0^t |\dot \varphi(s)|^2 ds\right\}
$$
for all $t >0$,
with $B$ the canonical process. Then, by Cameron-Martin theorem, under the measure $P^{\varphi}$
the process $W(t)=B(t)-\varphi(t)$ is a standard Brownian motion, and as in \eqref{eq:CMh} we write
 \begin{eqnarray}
 P\left[ \exp \{ \b \om(V_t(B))\} e^{ \int_0^t \dot \varphi (s) dB(s) - \frac{1}{2} \int_0^t |\dot \varphi(s)|^2 ds}
  \right] 
\nn &=& P^\varphi \left[ \exp \{ \b \om(V_t(\hat \tau_{\varphi}(W))\}\right] \\
\nn &=& P\left[ \exp \{ \b \om(V_t(\hat \tau_{\varphi}(B)))\}\right] \\
\nn &=& P\left[ \exp \{ \b \om(\hat \tau_{\varphi} (V_t (B)))\}\right] \\
\nn &=& P\left[ \exp \{ \b (\hat \tau_{-\varphi} \circ \om)(V_t(B))\}\right] \\
&=&   Z_t(\hat \tau_{-\varphi} \circ \om,\b) \nn \;,
\end{eqnarray}
yielding
\begin{equation} \label{eq:CMphi}
\gibbs \left[ \exp \{  \int_0^t \dot \varphi (s) dB(s) \}
  \right] = \exp\{ \frac{1}{2} \int_0^t |\dot \varphi(s)|^2 ds\} \times \frac{ W_t(\hat \tau_{-\varphi} \circ \om,\b)}{ W_t( \om,\b)} \;.
\end{equation}
\section{Consequences for weak disorder regime}
In this section we assume that $\b \in (\bar \b_c^-, \bar \b_c^+)$, more precisely that $ W_\8( \om,\b)=\lim_{t } W_t( \om,\b)  >0$ a.s. 
In particular, for a fixed $\varphi \in H_0^1$ (i.e., $\dot \varphi \in L^2$), we derive from \eqref{eq:CMphi} that 
\begin{equation} \label{eq:phiH1}
\gibbs \left[ \exp \{  \int_0^t \dot \varphi (s) dB(s) \}
  \right]  \longrightarrow 
\exp\{ \frac{1}{2} \|\dot \varphi\|_2^2 \} \times \frac{ W_\8(\hat \tau_{-\varphi} \circ \om,\b)}{ W_\8( \om,\b)} 
\end{equation}
a.s. as $t \to \8$.

In view of \eqref{eq:CMphi}, a natural question is continuity 
of $W_\8$ in the $\om$ variable: how does the limit depend on the environment ?
\begin{lemma} \label{lem:cW8}
Assume that $W_\8(\om, \b)>0$. Let $\varphi_T \in {\cal C}(\R_+, \rd)$ be a family indexed 
by $ T>0$ with $\varphi_T(0)=0$ which vanishes locally uniformly, i.e.,
$$ \forall t>0, \quad  \|\varphi_T\|_{\8,t} := \sup\{ |\varphi_T(s)|; s \in [0,t]\}  \to 0 \quad {\rm as \ } T \to \8. $$
Then we have, as $T \to \8$, 
\begin{eqnarray} \label{eq:l721-1}
{ W_\8(\hat \tau_{\varphi_T} \circ \om,\b)} &\longrightarrow& { W_\8( \om,\b)} \quad {\rm in\ } L^1{\rm -norm,}\\
 \label{eq:l721-2}
{ W_T(\hat \tau_{\varphi_T} \circ \om,\b)} &\longrightarrow& { W_\8( \om,\b)} \quad {\rm in\ } L^1{\rm -norm.}
\end{eqnarray}
\end{lemma}
The result can be compared to lemma \ref{lm:distanceWtxWt0}, where 
we have already considered the effect of a shift on the environment. In the notation of \eqref{eq:htau}, that lemma deals with constant shifts 
$\varphi=\varphi^{(x)}$ such that
$\varphi^{(x)}(t) \equiv x$ for all $t$, and implies that 
$$
x \mapsto W_\8 ( \hat \tau_{\varphi^{(x)}} \circ \omega) \quad {\rm is\  Lipschitz\ continuous\ from\ } \rd {\rm \ to\ } L^1.
$$
In the above lemma \ref{lem:cW8}, the shift is not anymore constant.

\begin{proof} Fix $t>0$ and decompose the difference $ W_\8(\hat \tau_{\varphi_T}\! \circ \om,\b) \!-\! { W_\8( \om,\b)}$ as
$$ 
 \{ W_\8(\hat \tau_{\varphi_T} \!\circ  \om,\b) \!-\! W_t(\hat \tau_{\varphi_T}\! \circ  \om,\b)\} \!+\! \{W_t(\hat \tau_{\varphi_T}\! \circ \om,\b) \!-\!  W_t( \om,\b)\} \!+\! \{ W_t( \om,\b)\! - \! W_\8( \om,\b)\}.
$$
Now, using triangular inequality and invariance in law of $\om$ under the shear transformation, we get
\begin{eqnarray}
\| { W_\8(\hat \tau_{\varphi_T}\! \circ \om,\b)} \!-\! { W_\8( \om,\b)}\|_1\! &\!\leq  \!&\! 2 
\| W_\8( \om,\b) \!- \! W_t( \om,\b)\|_1 \!+\! 
\| W_t(\hat \tau_{\varphi_T} \!\circ \om,\b) \!-\!  W_t( \om,\b)\|_1  \nn \\
\label{eq:contW8}
&=: & 2 \eps(t)+ \eps_t ( \varphi_T)
\end{eqnarray}
with $\eps(t)= \|W_\8-W_t\|_\8$ and 
$\eps_t(\cdot)$ defined by the above formula.
By assumption on $\b$ and proposition \ref{prop:UI},
we have $\lim_{t \to \8} \eps(t)=0$. On the other hand,  for fixed $t$, $W_t(\hat \tau_{\varphi_T} \circ \om,\b) \to  W_t( \om,\b)$ a.s. as $T \to \8$, 
and the variables $(W_t(\hat \tau_{\varphi_T} \circ \om,\b); T >0)$, $W_t( \om,\b)$ are uniformly integrable. Thus, the above convergence holds in $L^1$, which, combined with \eqref{eq:contW8}, completes the proof
of \eqref{eq:l721-1}. 

The proof of \eqref{eq:l721-2} is quite similar, writing this time the difference
$$
W_T(\hat \tau_{\varphi_T} \circ \om,\b)-  { W_\8( \om,\b)} = \{ W_T(\hat \tau_{\varphi_T} \circ \om,\b)- W_\8( \hat \tau_{\varphi_T} \circ \om,\b)\} +
\{ W_\8( \hat \tau_{\varphi_T} \circ \om,\b)- W_\8( \om,\b)\},
$$
observing that the first term in the right-hand side has $L^1$-norm equal to $\|W_T-W_\8\|_1$, that the second one vanishes a.s. and is uniformly integrable. This completes the proof.
\end{proof} 

For $a \in \rd, t>0$ define the function
\begin{equation} \nn
\vpat(s)=\frac{s \wedge t}{\sqrt t} a \;.
\end{equation}
From \eqref{eq:CMphi} with $\varphi=\vpat$, we have
\begin{eqnarray}\label{eq:valdor}
\gibbs \left[ e^{a \cdot \frac{B(t)}{\sqrt t}} \right] &=& e^{|a|^2/2} 
\times \frac{ W_t(\hat \tau_{-\vpat} \circ \om,\b)}{ W_t( \om,\b)} 
\;.
\end{eqnarray}
\begin{theorem} \label{prop:diff}
Assume weak disorder, i.e., that $W_\8>0$. Then, as $t \to \8$,
$$
\gibbs \left[ e^{a \cdot \frac{B(t)}{\sqrt t}} \right] \stackrel{\P}{\longrightarrow}  e^{|a|^2/2}.
$$
\end{theorem}
\begin{proof}
 Note that the family $(-\vpat, t >0)$ satisfies the assumptions of lemma \ref{lem:cW8}. Writing  formula \eqref{eq:valdor} as 
 $$
 \gibbs \left[ e^{a \cdot \frac{B(t)}{\sqrt t}} \right] - e^{|a|^2/2} =  e^{|a|^2/2} 
\times \frac{ W_t(\hat \tau_{-\vpat} \circ \om,\b) -  W_t( \om,\b)}{ W_t( \om,\b)} \;,
 $$
we see from Slutsky's lemma that this quantity vanishes as $t \to \8$.
\end{proof}
\begin{remark}
In a suitable sense, this result shows that the polymer is diffusive if $W_\8>0$. Its interest is that it covers  the {\em full} weak disorder region, in contrast to \cite[Th. 2.1.1]{CYkokyuroku} which only applies to the $L^2$ region. It can be viewed as a step to prove  diffusivity at weak disorder region. 
\end{remark}
\section{Moderate and large deviations at all temperature}
 
 In this section, $\b \in \R$ is arbitrary.
 
 \begin{theorem} \label{th:LD}
 For $\P$-a.e. realization of the environment $\om$, the following holds.
 \medskip
 
 (i) For all Borel $A \subset \rd$, 
 \begin{eqnarray}\nn
 - \inf \{  \frac{|\xi|^2}{2}; \xi \in \mathring A\} & \leq \liminf_{t \to \8} \frac{1}{t} \ln \gibbs \left[ \frac{B(t)}{t} \in A \right] & \\ \nn
 & \leq \limsup_{t \to \8} \frac{1}{t} \ln \gibbs \left[ \frac{B(t)}{t} \in A \right]  & \leq - \inf \{  \frac{|\xi|^2}{2}; \xi \in \overline A\} . 
\end{eqnarray}
 \medskip

 (ii)  Let $t_n$ be a positive sequence increasing to $\8$, let $\chi >0$ such that
 \begin{equation} \label{eq:hypMD}
\sum_{n \geq 1} \P( | \ln Z_{t_n}(\om, \b) - \P[\ln Z_{t_n}(\om, \b)] | > t_n^\chi) < \8, 
\end{equation}
and let $\xi > (1+\chi)/2$. Then, for all Borel $A \subset \rd$, 
 \begin{eqnarray}\nn
 - \inf \{  \frac{|a|^2}{2}; a \in \mathring A\} & \leq \liminf_{t \to \8} t_n^{-(2\xi-1)} \ln P_{t_n}^{\b,\om} \left[ \frac{B({t_n})}{t_n^\xi} \in A \right] & \\ \nn
 & \leq \limsup_{t \to \8}  t_n^{-(2\xi-1)} \ln P_{t_n}^{\b,\om} \left[ \frac{B({t_n})}{t_n^\xi} \in A \right]  & \leq - \inf \{  \frac{|a|^2}{2}; a \in \overline A\} . 
\end{eqnarray}
 \end{theorem}
 Part (i) is the almost sure large deviation principle for the polymer endpoint. The rate function $a \mapsto |a|^2/2$ is called the shape function in the polymer framework.
 In view of (\ref{eq:ph=}--\ref{eq:p0=}) it is no surprise. 
 
 Part (ii) is an almost sure moderate deviation principle. The rate function is the same as before. Since \eqref{eq:hypMD} is expected to hold for all $\chi > \xi^\parallel$, we derive that the polymer endpoint lies 
 at time $t_n$ within a distance $t_n^{ \xi^\parallel + \eps}$  with overwhelming probability for all positive $\eps$. Hence we get a relation between characteristic exponents
 \begin{equation} \label{eq:exprelation}
 \xi^\perp  \leq \frac{1+  \xi^\parallel}{2} \;.
\end{equation}
Since $ \xi^\parallel \leq 1/2$ from theorem \ref{th:var} and corollary \ref{cor:var1}, we derive a bound for all values of $d$:
 \begin{equation} \label{eq:exprelation3/4}
 \xi^\perp  \leq 3/4  \;.
\end{equation}
 \begin{proof} We start with (i). 
 Using \eqref{eq:CMphi} with $\varphi=\varphi_t$ given by $\varphi_t(s)= (s \wedge t )a$, we get
\begin{eqnarray}
\ln \gibbs[ e^{a \cdot B(t)}] &=& \frac{t |a|^2}{2} + \ln W_t(\tau_a \circ \om, \b) -  \ln W_t( \om, \b) \nn \\
 &=& \frac{t |a|^2}{2} + \big(\ln W_t(\tau_a \!\circ \om, \b) - \P[\ln W_t(\tau_a\! \circ \om, \b)]\big)+ \big( \ln W_t( \om, \b) - \P[\ln W_t(\om, \b)]\big),\nn 
 \end{eqnarray}
and by theorem \ref{th:freeenergy} the set $\Omega_a$ of environments such that 
\begin{equation}\label{eq:huawei}
\lim_{t\to \8}  \frac{1}{t}\ln \gibbs[ e^{a \cdot B(t)}] = \frac{ |a|^2}{2} 
\end{equation}
has full $\P$-measure. We now claim that on the event $\overline \Omega= \bigcap_{a \in \Q^d} \Omega_a$ the above limit holds for all $a \in \rd$.  
Indeed, the maps $a \mapsto \frac{1}{t}\ln \gibbs [ e^{a \cdot B(t)}] $ are convex, and the convergence is locally uniform on the closure $\rd$ of $\Q^d$, see Th. 10.8 in \cite{Roc70}.
Then, the large deviation principle  (i) follows from the G\"artner-Ellis-Baldi theorem (\cite{DeZe98}, p.44, Th.2.3.6), with the rate function 
$$|\xi|^2/2 = \sup \{ |a|^2/2; a \in \rd\}.$$

For the proof of (ii) we proceed as above. By \eqref{eq:CMphi} with $\varphi_t(s)= (s \wedge t )t^{\xi-1}a$, we first write
\begin{eqnarray}\nn
t^{-(2\xi-1)}\ln \gibbs[ e^{t^{\xi-1}a \cdot B(t)}] &=& \frac{|a|^2}{2} + t^{-(2\xi-1)} \big(\ln W_t(\hat \tau_{-\varphi_t} \!\circ \om, \b) - \P[\ln W_t(\om, \b)]\big)\\
&& \qquad \qquad +  t^{-(2\xi-1)} \big( \ln W_t( \om, \b) - \P[\ln W_t(\om, \b)]\big).\nn 
 \end{eqnarray}
 Now, in order to take the limit we restrict to the sequence $t=t_n$: using Borel-Cantelli lemma, we indeed get the a.s. limit \eqref{eq:huawei} along the sequence $t_n$. 
 From then on, the other arguments go through without any change.
 \end{proof}

 \chapter{Phase diagram in the $(\b, \nu)$-plane, $d\geq 3$} \label{ch:PD}
 
 In this section, $r>0$ is kept fixed, and we discuss the phase diagram in the remaining parameters $\b, \nu$. 
 Recall $\psi$ form \eqref{def:excessfreeenergy} and $\psi_t(\b,\nu)= \nu \lambda(\b) r^d - t^{-1} \P[ \ln Z_t(\om, \b)]$ from \eqref{def:psit}. Recall the notations
 \begin{equation} \nn
 {\cal D} = \{ (\b,\nu): \psi(\b, \nu)=0 \}, \quad {\cal L}  = \{ (\b,\nu): \psi(\b, \nu)>0\}, \quad {\tt Crit} = {\cal D} \bigcap \overline{\cal L}.
\end{equation}
$\cal D$ is is the delocalized phase, $\cal L$ the localized phase.
They are also high temperature/low density and low temperature/high density phases respectively.
They are separated by the critical curve ${\tt Crit} =\{(\beta_c^+(\nu), \nu); \nu >0\} \cup \{(\beta_c^-(\nu), \nu); \nu > \nu_c\}$. 
 We have seen that, in dimension $d=1$ or $ 2$,  ${\cal D}$ reduces to the axis $\b=0$, so  we assume $d \geq 3$ in this section.
 
 We introduce  
  \begin{equation}\label{def:nuc}
 \nu_c = \sup\{ \nu >0: \b_c^- = -\8\}\;.
\end{equation}
Then, $ \nu_c \in (0, \8)$, and 
$$
\nu > \nu_c \iff \b_c^- > -\8 \;.
$$

 A central question in polymer models is to estimate the critical curve \cite{dHollanderStFl, Giacomin}. 
 
 \section{Strategy for critical curve estimates}
 
 We follow the idea of \cite{CYBMPO2}, that is to
 \begin{quotation}
find curves $\nu(\b)$ along which $\psi$ is monotone.
\end{quotation}
We now sketch  the strategy. By computing the derivative  with the chain rule
\begin{equation}\nn
\frac{d}{d \b}  \psi_t(\b,\nu(\b)) = \nu' \frac{\partial}{\partial \nu}  \psi_t + \frac{\partial}{\partial \b}  \psi_t,
\end{equation}
one sees that,  along the smooth curve ${\cal C}_a^\a$
\begin{equation}\label{eq:curve}
\nu(\b)= a | \lambda (\b)|^{-\alpha}
\end{equation}
for positive constants $a, \alpha$, it takes the amenable form
\begin{equation} \label{eq:curved}
t \frac{\partial}{\partial \b}  \psi_t(\b,\nu(\b)) = \nu' \times \P \left[ \int_{(0,t] \times \rd} h_\alpha ( \lambda \gibbs[ \chi_{s,x}])dsdx\right] \;,
\end{equation}
where
\begin{equation} \label{eq:halpha}
h_\alpha ( u) = u -\frac{u^2}{\alpha(1+u)} - \ln(1+u)
\end{equation}
for $u >-1$. Now, the questions boils down to controling the sign of the function on the relevant interval with endpoints 0 and $\lambda(\b)$:
\begin{equation}\nn
h_\alpha ( u) \quad
\left\{
\begin{array}{ccc}
 \geq 0 & {\rm if \ } \alpha=2 {\rm \ and }  & u \geq 0,  \\
 \leq 0 &  {\rm if \ }  \alpha=2 {\rm \ and }  & u \in (-1,0],   \\
 \leq 0 & {\rm if \ }   \b>0, \alpha\leq \alpha(\b) {\rm \ and }  & u \in [0, \lambda],     \\
 \geq 0 &  {\rm if \ }   \b<0, \alpha\geq \alpha(\b) {\rm \ and }  & u \in [\lambda,0],
\end{array}
\right.
\end{equation}
with 
$\alpha(\b) = \frac{ (e^\b-1)^2}{e^\b(e^\b-1-\b)}, \quad \alpha(0)=2$. With this at hand, we can bound the critical curve from above and below with 
curves of the form \eqref{eq:curve} and specific $\alpha$'s, as indicated on Figure \ref{f-diag2}.
\begin{figure} \centering
\includegraphics[scale=0.6]{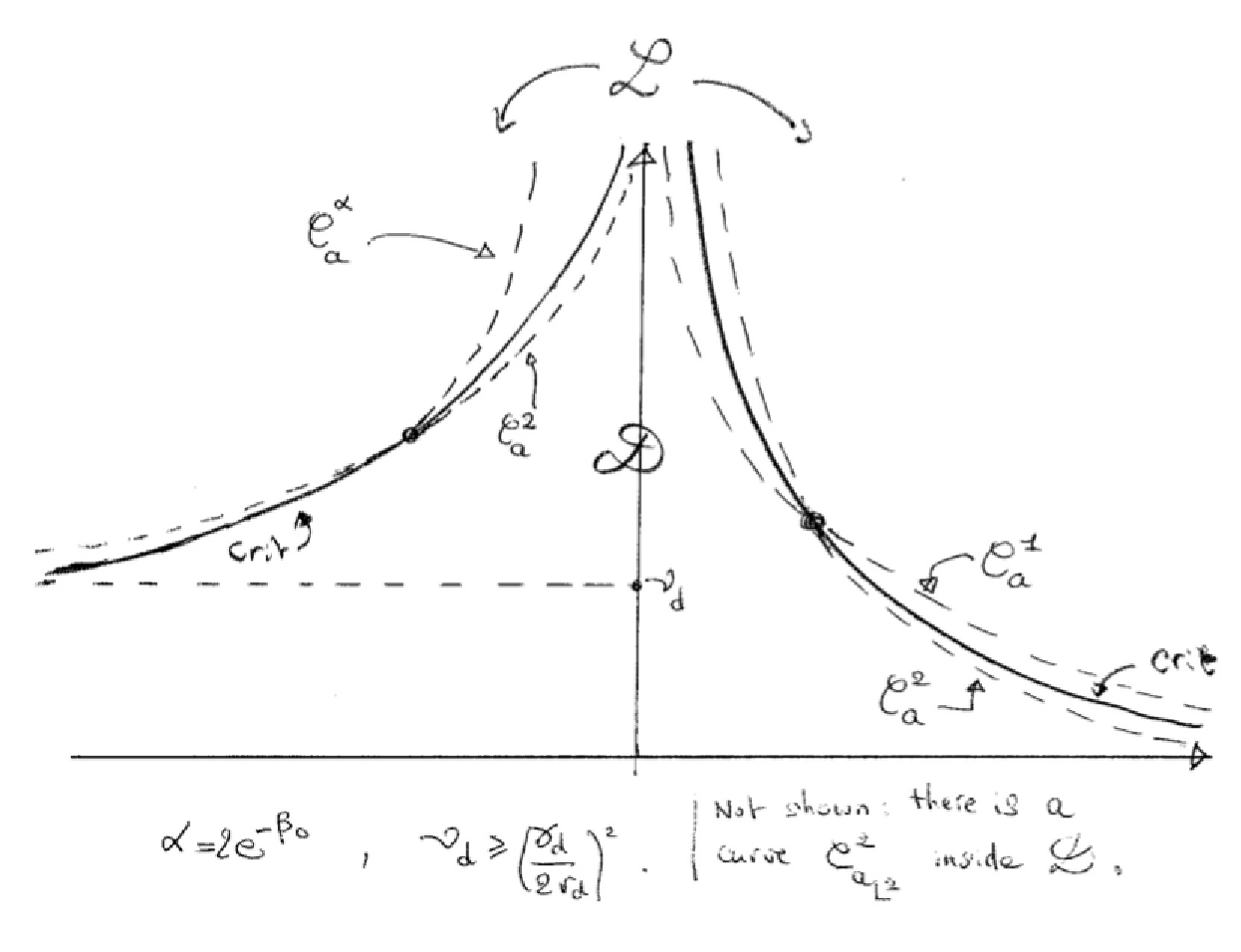}
\caption{Estimating the critical curve. ${\cal C}^\alpha_a$ denotes the curve \eqref{eq:curve}. } 
\label{f-diag2}
\end{figure}

\section{Main results}

We start with an estimate of the free energy.
In \cite[Th. 5.3.1]{CYBMPO2}, the asymptotics of the free energy is determined as $\nu  \b^2 $ diverges and $\b$ remains bounded. We state it below, it is key for our results.
\begin{lemma}\label{lem:boundfe}
Let $\b_0 \in (0,\8)$ arbitrary. Then, as $\nu  \b^2 \to \8$ and $|\b| \leq \b_0$, we have
$$ p(\nu,\b) = \b \nu r^d + {\cal O}((\nu  \b^2 r^d)^{5/6}\;.$$
\end{lemma}
\noindent We refer to the above paper for the involved, technical proof.

Among all the results, we mention:
\begin{theorem}
Let $d \geq 3$.

(i) the functions $\be_c^\pm(\nu)$ are locally Lipshitz and strictly monotone;

(ii) we have\footnote{For positive functions we write $f(x) \asymp g(x)$ as $x \to x_0$ if the ratio remains bounded  from 0 and from $\8$  as $x \to x_0$.}
 \begin{equation} \label{eq:decay*of*b_c+}
\b^+_c (d, \nu) \asymp \ln (1/\nu)
\; \; \mbox{as $\nu \searrow 0$};
\end{equation}

(ii) we have 
\begin{equation} \label{eq:decay*of*b_c2}
|\b^\pm_c (d,\nu)| \asymp 1/\sqrt{\nu} \quad {\rm as } \quad \nu \nearrow \8.
\end{equation}
\end{theorem} 
\section{Main steps}

The derivative of $\psi_t$ in $\b$ has been obtained in \eqref{eq:meanenergy}. We explain how to obtain the derivative in $\nu$, and for clarity,
in the sequel of the section we write $\P=\P_\nu$ to make the dependence in $\nu$ explicit.
 \begin{lemma} \label{lem:deriveenu}
We have 
\begin{equation}
\label{eq:dnp}
\frac{\6 }{\6 \nu} \P_\nu[\ln Z_t(\b, \om)]
 =\int_{[0,t]\times \rd}  dsdx
\; \P_\nu \;  \ln [1+\ln \gibbs( \chi_{s,x})]  \;.
\end{equation}
\end{lemma}
\begin{proof}:
For $k \geq 1$, let 
$$
Z_{t,k}=P[ e^{\b \om(V_t)}; A_k]\;,
\; \; \; 
\mbox{with} \; \; A_k=\{B_s \in [-k,k]^d, \;\forall s \leq t\},
$$
and $p_{t,k}(\b,\n )=t^{-1} \P_\nu \ln Z_{t,k}$. 
Let $K^r=[-k-r,k+r]^d$ and $K_t=(0,t]\times K^r$.
By Proposition \ref{prop:PoissonGirsanov},
$$
\P_\nu [ \ln Z_{t,k}]= \P_1 \left[ \rho_{t,\nu} \ln Z_{t,k}\right] 
\; \; \;\mbox{with}\; \; \; 
\rho_{t,\nu}=
\exp \left( \om_t(K_t)\ln \nu 
-  (\nu-1)t|K^r|\ri).
$$
Thus, $t p_{t,k}(\b,\nu)$ is differentiable in $\nu$, with derivative 
\begin{eqnarray*} 
\frac{1}{\nu} \P_\nu
\left[  \int_{K_t}
 \overline{\om}(dsdx)   \ln Z_{t,k}\ri]
&\stackrel{\eqref{eq:IPP}}{=}& 
\int_{K_t} dsdx \; \P_\nu    \ln 
\frac{Z_{t,k}(\om+\delta_{s,x})}{Z_{t,k}(\om) }  \\
&=& 
\int_{[0,t]\times \rd} dsdx \; \P_\nu  
\ln \left( 1+\ln \gibbs [\chi_{s,x} | A_k]\right).
\end{eqnarray*}
 Now, we write
\begin{equation}\nn
tp_{t,k}(\b,\nu)-tp_{t,k}(\b,1)
= \int_1^\nu d\nu' \int_{[0,t]\times \rd} dsdx \; \P_{\nu'}   \ln 
\left( 1+\ln \gibbs [\chi_{s,x} | A_k]\right).
\end{equation}
By dominated convergence theorem (see details in \cite{CYBMPO2}, Lemma 7.2.1), we can take the limit $k \to \8$, and obtain the desired statement.
\end{proof} 
We come to the core of the proof. With the derivatives of $\psi_t$  in both variables, from Lemma \ref{lem:deriveenu} and   \eqref{eq:meanenergy}, we obtain
\begin{eqnarray}\nn
t \frac{d}{d\b}  \psi_t(\b,\nu(\b)) & \!=& \! \nu' t \frac{\partial}{\partial \nu}  \psi_t + t \frac{\partial}{\partial \b}  \psi_t \\
& \!=& \!
\nu' \times \P \int_{[0,t]\times \rd} ds dx \left\{ \lambda \gibbs [\chi_{s,x}] +\frac{\nu}{\nu'} e^\b \lambda \frac{ \gibbs [\chi_{s,x}]^2 }{1\! +\!  \lambda  \gibbs [\chi_{s,x}] } 
-\ln \big( 1 \! + \! \lambda  \gibbs [\chi_{s,x}] \big)\right\}. \nn
\end{eqnarray}
Recall that $e^\b=\lambda'$; we recover the simpler formula \eqref{eq:curved} along the curves of equation
$$
\frac{\lambda' \nu}{\lambda \nu'} = - \frac{1}{\alpha},
$$
that is, the curves $h_\alpha$ from \eqref{eq:halpha}.
From then on,  the rest of the proof is a tedious but elementary exercise in calculus, performed in  \cite{CYBMPO2}.
We will not dive any further in the details of the proof, that the reader can find in this reference together with many fine estimates.
We summarize the section by giving a qualitative picture of the phase diagram in Figure \ref{f-diag}.
\begin{figure} \centering
\includegraphics[width=10cm,height=70mm]{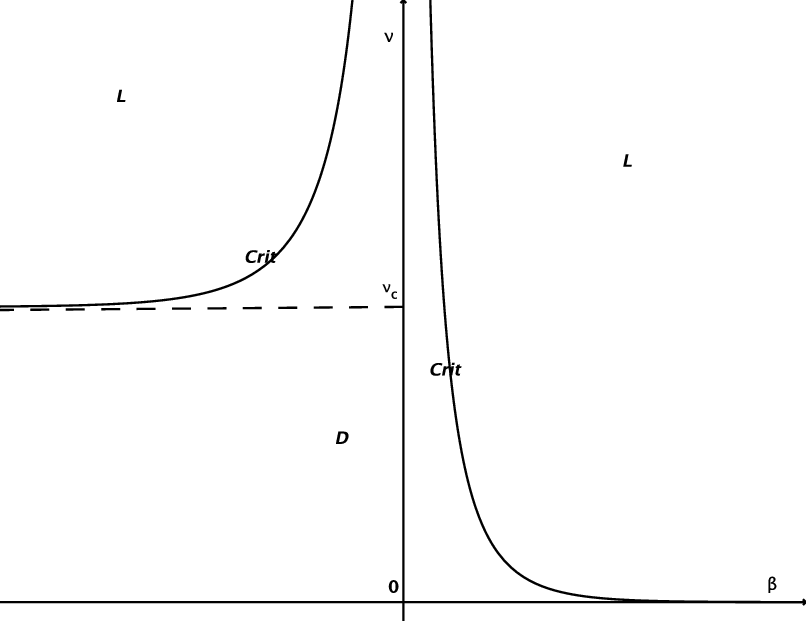}
\caption{Shape of the phase diagram, $d \geq 3$. }
\label{f-diag}
\end{figure}


 \chapter{Complete localization}
 
As we send some parameters to 0 or $\8$, the present model converges to other related polymer models. A first instance is the intermediate disorder regime of section \ref{sec:IR},
where parameters $\beta, \nu, r$ are scaled  with the polymer size, see \eqref{eq:paramRI}.
 
\section{A mean field limit}

Another  instance is the mean field limit: independently of the polymer length, we let 
\begin{center} $\nu \to \8$ and $\beta \to 0$ in such a way that $\nu \beta^2 \to b^2 \in (0,\8)$.
\end{center}
Then, the rewards given by the Poisson medium get denser and weaker in this asymptotics so that they turn into a Gaussian environment, given by 
generalized Gaussian process $g(t,x)$ with mean 0 and covariance
$$
\E[g(t,x)g(s,y)] 
= b^2 \delta (t-s) |U(x) \cap U(y)|,
$$
where $|\;\cdot \; |$ above denotes the Lebesgue measure. In other words, the environment is gaussian, which is correlated in space but not in time -- it is Brownian-like. 
Here, the limit of our model is a Brownian directed polymer in a Gaussian environment, introduced in
\cite{RoTi05}, with partition function
$$
{\frak Z}_t= P \left[ \exp \{ \int_0^t g(s, B(s)) ds\} \right] . 
$$
We do not elaborate this asymptotics, but instead we focus on the case $b=\8$.

\section{The regime of complete localization}
This corresponds to letting, independently of the polymer length, 
\begin{equation} \label{eq:completeloc}
\nu \to \8, \;|\beta|\leq \b_0, \quad {\rm such\ that} \quad \nu \beta^2 \to \8.
\end{equation}
(The parameter $r$ is kept fixed.) Precisely, we first let $t \to \8$ and then take the limit \eqref{eq:completeloc}.
\begin{theorem}
Under the assumption \eqref{eq:completeloc}, 
\begin{eqnarray*}
 1-\cO \lef( (\n \b^2)^{-1/6}\ri) \leq 
\liminf_{t \to \8}  \P\big[  \gibbs  (R_t^*) \big]   \qquad \qquad \qquad \qquad \qquad\\ \qquad\qquad \qquad \qquad \qquad 
 \leq  \limsup_{t \to \8}  \P\big[  \gibbs  (R_t^*) \big]  \leq 1-\cO \lef( (\n \b^2)^{-1/6}\ri).
\end{eqnarray*}
\end{theorem}
This statement describes the strong localization properties of the polymer path. 
The time-average $\frac 1t \int_0^t  
{\mathbf 1}_{B_s\in U(\Yt(s))}ds$ is the {\em time fraction} the polymer spends together with the favourite path. We know that when $(\partial \psi)/(\partial \beta) \neq 0$
the time fraction is positive. The claim here is that it is almost  the maximal value 1, in the limit  \eqref{eq:completeloc}. 
For a benchmark, we recall that, for the free measure $P$, 
for all smooth path ${\mathsf Y}$ and all $\delta>0$, there exists 
a positive $C$ such that for large $t$,
\bdnl{bench}
P\left( \frac{1}{t} \int_0^t  
{\mathbf 1}_{B_s\in U( {\mathsf Y}(s))}ds  \geq \delta \right)  \leq e^{-Ct}
\edn
(In fact, it is not difficult to see (\ref{bench}) for $Y \equiv 0$ by 
applying Donsker-Varadhan's large deviations \cite{donskervaradhan-wienersaus} for the occupation measure of Brownian motion. 
Then, one can 
use Girsanov transformation to extend (\ref{bench}) to the case of 
smooth path ${\mathsf Y}$.)
\begin{proof} Let $\hat \phi_t(\b)= t^{-1} ( \P \ln Z_t - t \nu \b r^d)$. We assume $\b \geq 0$, the other case being similar.  By convexity of $\hat \phi_t$,
\begin{eqnarray}\nn
\hat \phi_t(2\b) - \hat \phi_t(\b) &\stackrel{\rm convex}{\geq}& \b \frac{\partial \hat \phi_t }{ \partial \b} (\b) \\ \nn
&\stackrel{\eqref{eq:meanenergy}}{=}&  \frac{\b \nu \lambda(\b)}{t} \int_{(0,t]\times \rd}ds dx\; \P \;
 \frac{\gibbs[  \chi_{s,x}]- \gibbs[  \chi_{s,x}]^2}{1+ \lambda(\b) \gibbs[\chi_{s,x}]} \;.
 \end{eqnarray}
Bounding from above the denominator in the integral by $e^{\b_0}$, we get
\begin{equation} \label{eq:rase}
1- \P {\gibbs}^{\otimes 2} [R_t] \leq e^{\b_0} \frac{\hat \phi_t(2\b) - \hat \phi_t(\b)}{\b \lambda \nu r^d}\;.
\end{equation}
Now, using the bound in Lemma \ref{lem:boundfe}, we derive
$$
{\rm R.H.S.} \eqref{eq:rase} = {\cal O}((\nu \b^2)^{-1/6}\;.
$$
Similarly,
\begin{eqnarray}\nn
\hat \phi_t(\b) - \hat \phi_t(\b/2) &\stackrel{\rm convex}{\leq}& (\b/2) \frac{\partial \hat \phi_t }{ \partial \b} (\b) \\ \nn
&{=}&  \frac{\b \nu \lambda(\b)}{2t} \int_{(0,t]\times \rd}ds dx\; \P \;
 \frac{\gibbs[  \chi_{s,x}]- \gibbs[  \chi_{s,x}]^2}{1+ \lambda(\b) \gibbs[\chi_{s,x}]} \;,
 \end{eqnarray}
leading to 
$$
1- \P {\gibbs}^{\otimes 2} [R_t] \geq \frac{\hat \phi_t(\b) - \hat \phi_t(\b/2)}{\b \lambda \nu r^d}  = {\cal O}((\nu \b^2)^{-1/6}\;,
$$
and to the desired result.
\end{proof} 
We can 
extract fine additional  information and geometric properties of the Gibbs measure.
For $\delta \in (0,1/2)$ define the $(\delta,t)$-negligible set as
$$
{\cal N}_{\delta,t}^\h =  \Big\{(s,x) \in [0,t]\times \rd  : \gibbs(\chi_{s,x}) \leq \delta\Big\}, 
$$ 
and the  $(\delta,t)$-predominant set as
$$
{\cal P}_{\delta,t}^\h =  \Big\{(s,x) \in [0,t]\times \rd  : \gibbs(\chi_{s,x}) \geq 1-\delta\Big\}.
$$ 
As suggested by the names, ${\cal N}_{\delta,t}^\h$ is the set of space-time locations the polymer wants to stay
away from, and ${\cal P}_{\delta,t}^\h$ is  the set of  locations the polymer likes to visit. Both sets depend on the environment.
\begin{corollary}\label{th:pathlocfull2}
For all $0<\delta <1/2$, we have, under the assumption \eqref{eq:completeloc},  
\begin{equation}
\label{eq:r2to1-3}
\limsup_{t\to\8}\P \left[ \frac{1}{t}  \Big\vert  ({\cal N}_{\delta,t}^\h \cup  {\cal P}_{\delta,t}^\h)^\complement
 \Big\vert  \right]
=\cO \lef( (\n \b^2)^{-1/6}\ri),
\end{equation}
\begin{equation}
\label{eq:r2to1-4}
\limsup_{t\to\8}\P \gibbs \left[ \frac{1}{t}  \Big\vert  V_t(B) \bigcap {\cal N}_{\delta,t}^\h
\Big\vert 
\right]
=\cO \lef( (\n \b^2)^{-1/6}\ri),
\end{equation}
\begin{equation}
\label{eq:r2to1-5}
\limsup_{t\to\8}\P \gibbs \left[ \frac{1}{t}  \Big\vert  
V_t(B)^\complement   \bigcap 
{\cal P}_{\delta,t}^\h
  \Big\vert  
\right]
=\cO \lef( (\n \b^2)^{-1/6}\ri).
\end{equation}
\end{corollary}
Recall that $\vert \cdot \vert$ denotes the Lebesgue measure on $\R_+ \times \R^d$, and  note that 
$ \vert  {\cal N}_{\delta,t}^\h  \vert =  \vert  V_t(B)^\complement  \vert = \8$.

The limits (\ref{eq:r2to1-3}),  (\ref{eq:r2to1-4}),  (\ref{eq:r2to1-5}), bring 
information on how is the corridor around the favourite path where the measure concentrates for large $\nu \b^2$.
We depict the main features  for large $\nu \b^2$:
\begin{itemize}
\item
 most (in Lebesgue measure) time-space  locations become negligible or predominant, 
\item most  (in Lebesgue and Gibbs measures)  negligible  locations are outside the tube around the polymer path, 
\item
 most  (in Lebesgue and Gibbs measures)  predominant  locations are inside the tube around the polymer path.
\end{itemize}
The trace $ \big\{x \in \R^d: \gibbs(\chi_{s,x}) \geq 1-\delta\big\}$ at time $t$ of the $(\delta,t)$-predominant set is reminiscent of the $\epsilon$-atoms discovered  in \cite{VargasAtoms07}, with $\epsilon=1-\delta$, and discussed in \cite{BatesChatterjee}. These references study 
the time and space discrete setting, and restrict to the end point of the polymer.

\begin{proof}
It suffices to prove, for all $\delta \in (0,1/2]$,
\begin{eqnarray}
\label{eq:2to1-3}
\frac{1}{tr^d}  \Big\vert  \Big\{(s\!,\!x)\! \in \![0,t]\!\times \!\rd: \gibbs(\chi_{s,x})\! \in \![\delta, 1\!-\!\delta]\Big\} \Big\vert  &\!\leq\!&
\frac{1}{\delta (1\!-\!\delta)} \!
{\gibbs}^{\otimes 2} \Big(1\!-\! R_t\Big)\\
 \label{eq:2to1-4}
\gibbs \left[ \frac{1}{tr^d}  \Big\vert  V_t(B) \bigcap \Big\{(s,x) : \gibbs(\chi_{s,x}) \leq \delta\Big\} 
\Big\vert  
\right]  &\!\leq\!&
\frac{1}{1-\delta} \;
 {\gibbs}^{\otimes 2} \Big(1- R_t\Big)\\
\label{eq:2to1-5}
\gibbs \left[ \frac{1}{tr^d}  \Big\vert  
V_t(B)^\complement   \bigcap 
 \Big\{(s,x) : \gibbs(\chi_{s,x}) \geq 1-\delta\Big\} 
  \Big\vert  
\right] &\!\leq\!&
\frac{1}{1-\delta} \;
 {\gibbs}^{\otimes 2} \Big(1- R_t\Big).
\end{eqnarray}
Note that
$$
u(1-u) \geq (1-\delta)u {\mathbf 1}_{u < \delta} + 
\delta (1-\delta) {\mathbf 1}_{u \in[\delta,1-\delta]} +(1-\delta)(1-u) {\mathbf 1}_{u > 1-\delta} .
$$
Setting $A_s=\{x : \gibbs(\chi_{s,x} )\in [\delta, 1-\delta]\}$ and writing
\begin{eqnarray*}
{\gibbs}^{\otimes 2} (1- R_t) &=& \frac{1}{tr^d} \int_0^t \int_{\rd} \left[ {\gibbs}( \chi_{s,x})- {\gibbs}( \chi_{s,x})^2\right] dx ds \\
&\geq &  \frac{1}{tr^d} \int_0^t \int_{A_s} \left[ {\gibbs}( \chi_{s,x})- {\gibbs}( \chi_{s,x})^2\right] dx ds \\
&\geq & 
\delta (1-\delta) 
\frac{1}{tr^d} \int_0^t \left| \{x : {\gibbs}(\chi_{s,x}) \in [\delta, 1-\delta]\} \right|
 ds, 
\end{eqnarray*}
which yields (\ref{eq:2to1-3}). For the next one, we write
\begin{eqnarray*}
{\gibbs}^{\otimes 2} (1- R_t) &=& \frac{1}{tr^d} \int_0^t \int_{\rd} \left[ {\gibbs}( \chi_{s,x})- {\gibbs}( \chi_{s,x})^2\right] dx ds \\
&\geq& (1-\delta) \frac{1}{tr^d} \int_0^t \int_{\rd}  {\gibbs}( \chi_{s,x})  {\mathbf 1}_{
 {\gibbs} \chi_{s,x})  < \delta} ds dx\\
 &=& (1-\delta) {\gibbs} \left[ \frac{1}{tr^d} \int_0^t \int_{\rd}   {\mathbf 1}_{
{\gibbs}( \chi_{s,x})  < \delta, B_s \in U(x)}   \right] ds dx,
\end{eqnarray*}
which is (\ref{eq:2to1-4}). The last claim can be proved similarly.
\end{proof} 
\chapter{The Intermediate Regime ($d=1$)} \label{sec:IR}
In this chapter, we focus on dimension $d=1$, where the polymer is in the strong disorder phase as soon as $\b \neq 0$ is kept fixed (Remark \ref{rk:criticd=12}).
\section{Introduction}
Although it is believed that the model satisfies the following non-standard critical exponents: 
\begin{equation} \label{eq:expected_critical_exponents}
\sup_{0\leq s \leq t} |B_s| \approx t^{2/3} \quad \text{and} \quad \ln Z_t - \IP[\ln Z_t] \approx t^{1/3} \quad \text{as} \  t\to\infty,
\end{equation}
proofs are missing at this moment. It is also expected that the fluctuations of the free energy around its mean are of Tracy-Widom type:
\begin{conjecture} For all non-zero $\b,\nu$ and $r$, there exists some constant $\sigma(\beta,\nu)$ such that, as $t\to \infty$,
\begin{equation}
\frac{\ln Z_t-p(\b,\nu)t}{\sigma(\beta,\nu)t^{1/3}} \cvlaw F_{\mathrm{GOE}}
\end{equation}
where the $F_{\mathrm{GOE}}$ is the Tracy-Widom GOE distribution \textup{\cite{tracy1994level}}.
\end{conjecture}
These properties are characteristics of the KPZ universality class. They are in sharp contrast to the weak disorder regime, where one knows to a large extent that $B_t \approx t^{1/2}$ (Theorem \ref{prop:diff}), and where the free energy $\ln Z_t$ has order one fluctuations around its mean \eqref{eq:WD}, which are features of the Edward-Wilkinson universality class.

The KPZ universality class is a family of models of random surfaces dynamics that share non-gaussian statistics, non-standard critical exponents and scaling relations (3-2-1 in time, space and fluctuations, as in \eqref{eq:expected_critical_exponents}). Members of this class include some interacting particles systems (asymmetric simple exclusion prosses (ASEP), interacting Brownian motions), paths in random environment (directed polymers, first and last passage percolation), stochastic PDEs (KPZ equation, stochastic Burgers equation, stochastic reaction-diffusion equations). The reader may refer to \cite{corwin2016kardar} for a non-technical review on the KPZ universality class.

The Kardar-Parisi-Zhang (KPZ) equation is the non-linear stochastic partial differential equation:
\begin{equation} \label{KPZ} \frac{\partial \mathcal{H}}{\partial T}(T,X) = \frac{1}{2} \frac{\partial^2 \mathcal{H}}{\partial X^2}(T,X) + \frac{1}{2} \left(\frac{\partial \mathcal{H}}{\partial X}(T,X) \right)^2 + \beta \eta(T,X),
\end{equation}
where $\b\in\mathbb{R}$ and $\eta$ is a random measure on $[0,1]\times\mathbb{R}$ called the space-time \emph{Gaussian white noise}, which verifies that:
\begin{enumerate}[label=(\roman*)]
\item For all measurable sets $A_1,\dots,A_k$ of $[0,1]\times\mathbb{R}$, $\big(\eta(A_1),\dots,\eta(A_k)\big)$ is a centered Gaussian vector.
\item For all measurable sets $A,B$ of $\,[0,1]\times\mathbb{R}$, then $\mathbb P[\eta(A)\eta(B)] = |A\cap B|.$
\end{enumerate}

The KPZ equation models the behavior of a random interface growth and was introduced by Kardar, Parisi and Zhang \cite{kardar1986dynamic} in 1986. It is difficult to make sense of this equation and Bertini-Cancrini \cite{BertiniCancrini95} argued that a possible definition of $\mathcal{H}_\b$ could be given by the so-called \emph{Hopf-Cole transformation}:
\begin{equation} \label{eq:Hopf-Cole}
\mathcal{H}_\b(T,X) = \ln \mathcal{Z}_\b(T,X),\end{equation}
where $\mathcal{Z}_\b$ is the solution of the stochastic heat equation (SHE):
\begin{equation} \label{SHE}
\frac{\partial \mathcal{Z}_\b}{\partial T}(T,X) = \frac{1}{2} \frac{\partial^2 \mathcal{Z}_\b}{\partial X^2}(T,X) + \beta \mathcal{Z}_\b(T,X) \eta(T,X).
\end{equation}

In a breakthrough paper \cite{ACQ11}, Amir, Corwin and Quastel were able to describe the pointwise distribution of $\mathcal{H}_\b(T,X)$ by exploiting the weak universality of the ASEP model. It results from this that the KPZ equation lies in the KPZ universality class.

The \emph{weak KPZ universality conjecture} states that the KPZ equation is a universal object of the KPZ class. As a general idea, the KPZ equation should appear as a scaling limit at critical parameters for models that feature a phase transition between the Edward-Wilkinson class (4-2-1 scaling) and the KPZ class. This was first verified for the model of ASEP \cite{BertiniGiacomin97}, and more recently for the discrete and Brownian directed polymers \cite{AlbertsKhaninQuastelAP,CoscoIntermediate}. The proofs rely on the Hopf-Cole transformation, which enables one to switch between the KPZ equation and the stochastic heat equation. In this chapter, we essentially summarize the arguments of \cite{CoscoIntermediate} to explain why the Brownian polymer in Poisson environment model verifies the weak KPZ universality.


\section{Connections between stochastic heat equation(s) and directed polymers 
}
\subsection{The continuum case}
\label{continuumPolymerSHE} 
A special case of interest for the SHE, where $\mathcal{Z}_\b(T,X)$ can be seen as the point-to-point partition function of a directed polymer, placed at $X=0$ at time $T=0$, is when
\begin{equation} \label{eq:ICofSHE}
\mathcal{Z}_\b(0,X) = \delta_0(X).
\end{equation}
In this case, $\mathcal{Z}_\b(T,X)$ can be expressed through the following shortcut (cf. Section \ref{subsec:constructionofP2P}):
\begin{equation} \label{eq:FeynmanKac0}
 \mathcal{Z}_\b(T,X) = \rho(T,X) \  P_{0,0}^{T,X} \left[ :\mathrm{exp}:\left(\beta \int_0^T \eta(u,B_u)\mathrm{d}u\right) \right],
\end{equation}
where $\rho(t,x)=e^{-x^2/2t}/\sqrt{2\pi t}$.

This equation is similar to the definition of the point-to-point partition function a polymer with Brownian path and white noise environment. Alberts, Khanin and Quastel \cite{AlbertsKhaninQuastelCRP} were in fact able to construct a polymer measure with P2P partition function given by $\mathcal{Z}_\b(T,X)$.
As both the environment and the paths of the polymer are continuous, it was named \emph{the continuum directed random polymer}.

Similarly to the Poisson polymer, the P2P free energy $\mathcal{F}_\b(T,X)$ can be defined as
\begin{equation}
\mathcal{F}_\b(T,X) = \ln \frac{\mathcal{Z}_\b(T,X)}{\rho(T,X)},
\end{equation}
so that the free energy of the polymer and the solution of the KPZ equation follow the relation:
\begin{equation}
\mathcal{F}_\b(T,X) = \mathcal{H}_\b(T,X) + X^2/2T + \ln \sqrt{2 \pi T}.
\end{equation}

\subsection{The Poisson case}
Introduce the renormalized point-to-point partition function:
\begin{equation} \label{eq:renormalizedP2Pdef}
W(t,x;\om,\b,r) = \rho(t,x) P_{0,0}^{t,x}\left[ \exp \{ \b \om(V_t)-\lambda(\b) \nu r^d t\}\right].
\end{equation}
We will often shorten the notation $W(t,x;\om,\b,r) =W(t,x) $ when no confusion can arise. Compared to $Z_t(\om, \b; x)$ of \eqref{PtoPdefinition}, a major difference is that it encorporates the Gaussian kernel as a factor.
In the next theorem, we state that the renormalized P2P partition function verifies a weak formulation of the following \emph{stochastic heat equation with multiplicative Poisson noise}:
\begin{equation}
\partial_t W(t,x) = \frac{1}{2} \Delta W(t,x) + \lambda W(t-,x) \bar{\om}(\rmd t \times U(x)).
\end{equation}
When $\b=0$, it reduces to the usual heat equation.
\begin{theorem}[Weak solution]\label{th:she}
 For all $\varphi\in \mathcal{D}(\mathbb R)$ and $t\geq 0$, we have $\IP$-almost surely
\begin{multline} \label{eq:sheForPoisson}
\int_\mathbb{R} W(t,x)\varphi(x) \rmd x = \varphi(0) + \frac{1}{2}\int_0^t \rmd s \int_\mathbb{R} W(s,x) \Delta \varphi (x) \rmd x \\
+ \lambda \int_{\mathbb{R}} \rmd x \varphi(x) \int_{(0,t]\times \mathbb R} \bar{\om}(\rmd s, \rmd y) W(s-,x) \mathbf{1}_{|y-x|\leq r/2} \,.
\end{multline}
\end{theorem}
\begin{proof} Let $\xi_t = \exp(\b \om(V_t(B)) - \lambda(\b)\nu r^d t)$ and observe that 
\begin{equation*}\int_\mathbb{R} W(t,x)\varphi(x)\rmd x = P[\xi_t \varphi(B_t)].
\end{equation*} 
Then, recalling that $\om(V_t(B)) = \int \chi_{s,x} \, \om_t(\rmd s \rmd x)$, we use It\^o's formula \cite[Section II.5]{IkedaWatanabe} for fixed $B$ to get that
\begin{align} \label{eq:finalecoupedefrance}
\nonumber \xi_t &= 1 -\lambda \nu r^d \int_0^t \xi_s \rmd s + \lambda \int_{(0,t]\times \mathbb{R}} \xi_{s-} \chi_{s,x} \, \om(\rmd s \rmd x)\\
& = 1 + \lambda\int_{(0,t]\times \mathbb{R}} \xi_{s-} \chi_{s,x} \, \bar{\om}(\rmd s \rmd x),
\end{align} as almost surely, $\IP$-a.s. $\xi_s = \xi_{s-}$ a.e.

As a difference of two increasing processes, $\xi$ is of finite variation over all bounded time intervals. Also note that one can get an expression to the measure associated to $\xi$ from the last equation. By the integration by part formula \cite[p.52]{jacod2013limit},
\[ \xi_t \varphi(B_t) = \xi_0 \varphi(B_0) + \int_0^t \xi_{s-} \rmd \varphi(B_s) + \int_0^t \varphi(B_s) \rmd \xi_s + [\xi,\varphi(B)]_t,\]
where $[\xi,\varphi(B)]_t=0$ since $\varphi(B)$ is continuous. Applying It\^o 's formula on $\rmd \varphi(B)$ and then taking $P$-expectation (which cancels the martingale term in the It\^o formula), one obtains by \eqref{eq:finalecoupedefrance} that $\IP$-a.s.
\begin{align*}
&\int_\mathbb{R} W(t,x)\varphi(x)\rmd x\\
&  = \varphi(0) + \frac{1}{2}\int_0^t P[\xi_{s-} \Delta \varphi(B_s)] \rmd s + \lambda \int_{(0,t]\times\mathbb R} P[\varphi(B_s) \xi_{s-} \chi_{s,y}]\bar{\om}(\rmd s \rmd y)\\
& = \varphi(0) + \frac{1}{2}\int_0^t \int_\mathbb R \Delta \varphi(x) W(s-,x) \rmd x \rmd s + \lambda \int_{(0,t]\times\mathbb R} \left(\int_\mathbb R \varphi(x) \mathbf{1}_{|y-x|\leq r/2} W(s-,x) \rmd x\right) \bar{\om}(\rmd s \rmd y).
\end{align*}
To conclude the proof, observe that we can apply Fubini's theorem to the last integral since for all $t>0$,
\begin{align*}
\IP \int_{(0,t] \times \mathbb{R}} P[|\varphi(B_s)|\xi_{s-} \chi_{s,y}] \om(\rmd s \rmd y) &= \nu e^\beta \int_{(0,t] \times \mathbb{R}} \IP[\xi_{s-}] P[|\varphi(B_s)| \chi_{s,y}] \rmd s \rmd y\\
& = \nu e^\beta r \int_0^t P[|\varphi(B_s)|] \rmd s < \infty,
\end{align*}
where we have used the Mecke equation, cf. \eqref{eq:IPP} or \cite[4.1]{LastPenrose} in the first equality.
\end{proof}

\section{Chaos expansions}
Let us first introduce some notations.
For any $k\geq 1$, $s_1,\dots,s_k \in \mathbb{R}_+$ and $x_1,\dots,x_k \in \mathbb{R}$, write $\mathbf{s}=(s_1,\dots,s_k)$ and $\mathbf{x}=(x_1,\dots,x_k)$. Let 
\begin{equation}\Delta_k(0,t)=\{\mathbf{s}\in [u,t]^k \,|\, 0<s_1<\dots<s_k\leq t\},
\end{equation}
be the $k$-dimensional simplex and $\Delta_k=\Delta_k(0,1)$.

\subsection{The continuum case}
\label{subsec:constructionofP2P}
We give here the definition of a \emph{mild} solution to the stochastic heat equation, and we will see how this leads to an expression of the solution as a Wiener chaos expansion.
We first mention that it is possible (cf. \cite{janson_1997}) to extend the integral over the space time white noise to any square integrable function:
\begin{proposition}
There exists an isometry $I_1 : L^2\big([0,1]\times\mathbb{R}\big) \mapsto L^2(\Omega,\mathcal{G},\IP)$ verifying:
\begin{enumerate}[label=(\roman*)]
\item For all measurable set $A$ of $[0,1]\times\mathbb{R}$, we have $I_1(A) = \eta(A)$.
\item For all $g\in L^2$, the variable $I_1(g)$ is a centered Gaussian variable of variance $\Vert g \Vert^2_{L^2([0,1]\times \mathbb{R})}$.
\end{enumerate}
\end{proposition}
\noindent We call $I_1(g)$ the \emph{Wiener integral} which also writes $I_1(g) = \int_{[0,1]}\int_{\mathbb{R}} g(s,x) \, \eta(\rmd s,\rmd x).$

It is said that $\mathcal Z$ is a \emph{mild} solution to the stochastic heat equation \eqref{SHE} if, for all $0 \leq S <T\leq 1$,
\begin{equation} \label{eq:mildSHE}
\begin{split}
\mathcal{Z}(T,X) = & \int_{\mathbb{R}} \rho(T-S,X-Y)\mathcal{Z}(S,Y)\mathrm{d}Y \\
&+\beta \int_S^T \int_{\mathbb{R}} \rho(T-U,X-Y) \mathcal{Z}(U,Y) \eta(U,Y) \mathrm{d}U \mathrm{d}Y,
\end{split}
\end{equation}
and if for all $T\geq 0$, $\mathcal{Z}(T,X)$ is measurable with respect to the white noise on $[0,T]\times\mathbb{R}$.

\begin{remark}
As a motivation to look at this form of the equation, one can check that if $\mathcal{Z}(T,X)$ satisfies \eqref{eq:mildSHE} with a smooth deterministic function $\eta(U,Y)$, then $\mathcal{Z}(T,X)$ is a solution to the SHE \eqref{SHE} with smooth noise.
\end{remark}

\begin{remark}
Under some integrability condition, it can be shown that there is a unique mild solution - up to indistinguishability - to the SHE with Dirac initial condition \textup{\cite{BertiniCancrini95}}. This solution is continuous in time and space for $(T,X)\in (0,1]\times \mathbb{R}$, and it is continuous in $T=0$ in the space of distributions. Furthermore, $\mathcal{Z}_\b(T,X)$ can be shown to be positive for all $T>0$ \textup{\cite{morenoflores2014,Mueller}}.
\end{remark}

Using the initial condition $\mathcal{Z}(0,X) = \delta_X$, we get by iterating equation \eqref{eq:mildSHE} that
\begin{align*}
&\mathcal{Z}(T,X) = \rho(T,X) + \beta \int_0^T \int \rho(T-U,X-Y) \rho(U,Y) \eta(U,Y) \mathrm{d}U \mathrm{d}Y \\
&+ \beta^2 \iint_{0<R<U\leq T} \iint_{\mathbb{R}^2} \rho(T-U,X-Y) \rho(U-R,Y-Z) \mathcal{Z}(R,Z) \\
& \qquad \qquad \qquad \qquad \qquad \times \eta(U,Y) \eta(R,Z) \mathrm{d}U \mathrm{d}Y  \mathrm{d}R \mathrm{d}Z.
\end{align*}
It is possible to give a proper definition of these iterated integrals, and one can find the details of such a procedure in \cite[Chapter 7]{janson_1997}. We give a few properties of these integrals:
\begin{property}  For all $k>0$, there exists a map $I_k:L^2(\Delta_k \times \mathbb{R}^k)\mapsto L^2(\Omega,\mathcal{G},\IP)$, which has the following properties:
\begin{enumerate}[label = (\roman*)]
\item For all $g\in L^2(\Delta_k \times \mathbb{R}^k)$ and $h\in L^2(\Delta_j \times \mathbb{R}^j)$, the variable $I_k(g)$ is centered and
\begin{equation} \label{eq:Wiener_cov_struct}
\IP \left[{I}_k(g) {I}_j(h)\right] =\ \delta_{k,j} \, <g,h>_{L^2(\Delta_k \times \mathbb{R}^k)}.
\end{equation}
\item The map $I_k$ is linear, in the sense that for all square-integrable $f,g$ and reals $\lambda,\mu$,
\begin{equation*}
\IP \text{-a.s.} \quad I_k(\lambda f + \mu g) = \lambda \, I_k(f) + \mu \,I_k(g).
\end{equation*}

The operator $I_k$ is called the \textbf{multiple Wiener integral}, and for $g\in L^2(\Delta_k \times \mathbb{R}^k)$, we also write
\[I_k(g) = \int_{\Delta_k} \int_{\mathbb{R}^k} g(\mathbf{t},\mathbf{x}) \eta^{\otimes k}(\mathrm{d}\mathbf{t},\mathrm{d}\mathbf{x}).\]
\end{enumerate}
\end{property}
\begin{remark} As a justification of the "iterated integral" property, it can be shown that the map $I_k$ extends to $L^2([0,1]^k\times \mathbb{R}^k)$, where it verifies that for all orthogonal family $(g_1,\dots,g_k)$ of functions in $L^2([0,1] \times \mathbb{R})$:
\begin{equation} \label{eq:IkOrthog} {I}_k\left(\bigotimes_{j=1}^k g_j \right) =  \prod_{j=1}^k I_1(g_j),
\end{equation}
where $\bigotimes$ denotes the tensor product: $(\bigotimes_{j=1}^k g_j)(\mathbf{s},\mathbf{x}) = \prod_{j=1}^k g_j(s_j,x_j)$.
\end{remark}

By repeating the above iteration procedure, one gets that:
\begin{equation} \label{eq:ChaosExpantionSHE}
\mathcal{Z}(T,X) = \rho(T,X) + \sum_{k=1}^\infty \b^k \int_{\Delta_k} \int_{\mathbb{R}^k} \rho^k(\mathbf{S},\mathbf{Y};T,X) \eta^{\otimes k}(\mathrm{d}\mathbf{S},\mathrm{d}\mathbf{Y}),
\end{equation}
where we have used the notation, for $\mathbf{s}\in \Delta_k(s,t)$ and $\mathbf{y}\in\mathbb{R}^d$,
\begin{equation*}
\rho^k(\mathbf{s},\mathbf{y} \, ; t,x)= \rho(s_1,y_1) \left( \prod_{j=1}^{k-1} \rho(s_{j+1}-s_j,y_{j+1}-y_j) \right) \rho(t-s_k,x-y_k).
\end{equation*}

The infinite sum \eqref{eq:ChaosExpantionSHE} is called a \emph{Wiener chaos expansion}. By the covariance structure of the Wiener integrals \eqref{eq:Wiener_cov_struct}, all the integrals in the sum are orthogonal and to prove that $\eqref{eq:ChaosExpantionSHE}$ converges in $L^2$, it suffices to check that (see \cite{AlbertsKhaninQuastelCRP}):
\[\sum_{k=0}^{\infty} \Vert \rho^k(\cdot,\cdot;T,X) \Vert^2_{L^2(\Delta_k\times\mathbb{R}^k)} < \infty.\]

The ratio $\frac{\rho^k(\mathbf{s},\mathbf{y} \, ; t,x)} {\rho(t,x)} $ is the $k$-steps transition function of a Brownian bridge,  starting from $(0,0)$ and ending at $(t,x)$. From this observation, it is possible to introduce an alternative expression of 
the mild solution $\mathcal{Z}(T,X)$ of SHE equation, via a Feynman-Kac formula:
\begin{equation} \label{eq:FeynmanKac}
 \mathcal{Z}(T,X) = \rho(T,X) \  P_{0,0}^{T,X} \left[ :\mathrm{exp}:\left(\beta \int_0^T \eta(u,B_u)\mathrm{d}u\right) \right],
\end{equation}
The Wick exponential $:\mathrm{exp}:$ of a Gaussian random variable $\xi$ is defined by 
\[:\mathrm{exp}(\xi): = \sum_{k=0}^\infty \frac{1}{k!} :\xi^k:
 \]
where the $:\xi^k:$ notation stands for the Wick power of a random variable (cf.\cite{janson_1997}).
The integral $\int_0^T \eta(u,B_u)\mathrm{d}u$, on the other hand, is not well defined, and to understand how to go from \eqref{eq:FeynmanKac} to  \eqref{eq:ChaosExpantionSHE}, one should use the following identification:
\[ P_{0,0}^{T,X} \left[ :\left( \b \int_0^T \eta(u,B_u)\mathrm{d}u \right)^k : \right]  \ = \beta^k k! \int_{\Delta_k} \int_{\mathbb{R}^k} \frac{\rho^k(\mathbf{S},\mathbf{Y}\, ; T,X)}{\rho(T,X)} \eta^{\otimes k}(\mathrm{d}\mathbf{S} \mathrm{d}\mathbf{Y}).\]

From now on, we suppose that $\mathcal{Z}_\b(T,X)$ is defined through equation \eqref{eq:ChaosExpantionSHE}. Integrating over $X$ this equation leads to the definition of the partition function of the continuum polymer:
\begin{equation} \label{eq:defOfP2LZBeta}
\mathcal{Z}_\b = \sum_{k=0}^\infty \b^k I_k( \rho^k),
\end{equation}
where $\rho^k$ is the $k$-th dimensional Brownian transition function, defined for $(\mathbf{s},\mathbf{x}) \in \Delta_k\times\mathbb{R}^k$ by:
\begin{equation}
\begin{split}
\rho^k(\mathbf{s},\mathbf{x}) &= \rho(s_1,x_1) \left( \prod_{j=1}^{k-1} \rho( s_{j+1}-s_j,x_{j+1}-x_j) \right)\\
& =  P\left(B_{s_1} \in \rmd x_1,\dots, B_{s_k} \in \rmd x_k\right),
\end{split}
\end{equation}
with the convention that $\rho^0 = 1$. The motivation for writing the partition function  as in  \eqref{eq:defOfP2LZBeta} is that \eqref{eq:ChaosExpantionSHE}
writes
$\mathcal{Z}_\b (T,X)= \sum_{k=0}^\infty \b^k I_k( \rho^k(\cdot; T,X))$.

\subsection{The Poisson case}
We want to express $W_t$ in a similar way as \eqref{eq:defOfP2LZBeta}, this time with Poisson iterated integrals. We give here the basic definitions of these integrals and one can refer to \cite{LastPenrose} for more details\footnote{Note that for simplicity, we choosed to define here the integrals for functions of the simplex, so that some normalizing $k!$ terms and symmetrisation of some objects should be added to match the definitions in \cite{LastPenrose}.}.
\begin{definition}
For any positive integer $k$, define the $\mathbf k$\textbf{-th factorial measure} $\,\om_t^{(k)}$ to be the point process on $\Delta_k(0,t)\times\mathbb{R}^k$, such that, for any measurable set $\mathbf{A}\subset \Delta_k(0,t)\times\mathbb{R}^k$,
\begin{equation}
\om_t^{(k)}(\mathbf{A}) = \sum_{\substack{
(s_1,x_1),\dots,(s_k,x_k) \in\, \om_t\\
s_1<\dots<s_k}} \mathbf{1}_{((s_1,x_1),\dots,(s_k,x_k)) \in \mathbf{A}}.
\end{equation}
Otherwise stated,
\begin{equation}
\om_t^{(k)} = \sum_{\substack{
(s_1,x_1),\dots,(s_k,x_k) \in\, \om_t\\
s_1<\dots<s_k}} \delta_{((s_1,x_1),\dots,(s_k,x_k))}\;.
\end{equation}
\end{definition}
These factorial measures define naturaly a multiple integral for the point process $\om_t$. Contrary to the Wiener integrals, these integrals are not centered, so what we really want is to define a multiple integral for the compensated process $\bar{\om}_t$. This is done as follows:
\begin{definition}
For $k\geq1$ and $g\in L^1(\Delta_k(0,t)\times \mathbb{R}^k)$, denote
the \textbf{multiple Wiener-It\^o integral} of $g$ as
\begin{equation} \label{eq:def_omk}
\bar{\om}^{(k)}_t(g)
:=\sum_{J\subset [k]} (-1)^{k-|J|} \int_{\Delta_k\times \mathbb{R}^k} 
g( \mathbf{s}, \mathbf{x})\, \om_t^{(|J|)}(\rmd \mathbf{s}_J,\rmd \mathbf{x}_J) \, \nu^{k-|J|} \,\rmd \mathbf{s}_{J^c} \, \rmd \mathbf{x}_{J^c}.
\end{equation}
When $k=0$, define $\bar{\om}^{(0)}_t$ to be the identity on $\mathbb{R}$.
\end{definition}

The two following results can be found in \cite{LastPenrose}:
\begin{proposition}
For $k\geq 1$, the map $\bar{\om_t}^{(k)}$ can be extended to a map \[\begin{array}{ccccc}
\bar{\om}_t^{(k)} & : & L^2(\Delta_k(0,t)\times \mathbb{R}^k) & \to & L^2(\Omega,\mathcal{G},\IP)\\
& & g &\mapsto & \bar{\om}_t^{(k)}(g),
\end{array} \] which coincides with the above definition of $\,\bar{\om}_t^{(k)}$ on the functions of $L^1 \cap L^2(\Delta_k(0,t)\times \mathbb{R}^k)$.
\end{proposition}
\begin{property}
\begin{enumerate}[label=(\roman*)]
\item For any $k\geq 1$ and $g\in L^2(\Delta_k(0,t)\times \mathbb{R}^k)$, we have $\IP\big[\bar{\om}_t^{(k)} (g)\big] = 0$.
\item For any $k\geq 1$ and $l\geq 1$, $g\in L^2(\Delta_k(0,t)\times \mathbb{R}^k)$ and $h\in L^2(\Delta_l(0,t)\times \mathbb{R}^l)$, the following covariance structure holds:
\begin{equation} \label{eq:covariance_structure}
\IP\left[\bar{\om}_t^{(k)}(g) \, \bar{\om}_t^{(l)}(h) \right] = \delta_{k,l} \ \nu^k < g ,h >_{L^2(\Delta_k(0,t)\times \mathbb{R}^k)}.
\end{equation}
\item The map $\bar{\om}_t^{(k)}$ is linear, in the sense that for all square-integrable $f,g$ and reals $\lambda,\mu$,
\begin{equation*}
\IP \text{-a.s.} \quad \bar{\om}_t^{(k)}(\lambda f + \mu g) = \lambda \, \bar{\om}_t^{(k)}(f) + \mu \,\bar{\om}_t^{(k)}(g).
\end{equation*}
\end{enumerate}
\end{property} 
\begin{proposition} \textup{\cite{CoscoIntermediate}} \label{prop:Poisson_chaos_decomposition}
The renormalized partition function admits the following Wiener-It\^o chaos expansion:
\begin{equation} \label{eq:wiener_ito_wt_expansion}
W_t = \sum_{k=0}^\infty \, \bar{\om}^{(k)}_t (\Psi^k),
\end{equation} where the sum converges in $L^2$ and where, for all $\mathbf{s}\in \Delta_k$, $\mathbf{x}\in\mathbb{R}^k$ and $k\geq 0$, we have set:
\begin{equation} \label{eq:TkWtExpression}
\Psi^k(\mathbf{s},\mathbf{x}) = \lambda(\b)^k\, P\left[\prod_{i=1}^k \chi_{s_i,x_i}(B) \right],
\end{equation}
with the convention that an empty product equals $1$.
\end{proposition}

\begin{proof}[Sketch of proof] We follow the proof of Lemma 18.9 in \cite{LastPenrose}. 
By definition we have that $W_t = P\left[e^{\b\om(V_t(B)-t\lambda r\nu}\right]$. Hence, assuming that Fubini's theorem applies to the RHS of \eqref{eq:wiener_ito_wt_expansion}, it is enough to show that $\IP\times P$-almost surely:
\begin{equation} \label{eq:exp_expansion}
e^{\b\om(V_t(B))-t\lambda r\nu} = \sum_{k=0}^\infty \, \bar{\om}^{(k)}_t \left((\lambda\chi)^{\otimes k}\right),
\end{equation}
where, for all $\mathbf{s}\in \Delta_k(0,t)$, $\mathbf{x}\in\mathbb{R}^k$, we have defined $(\lambda\chi)^{\otimes k}(\mathbf{s},\mathbf{x}) = \prod_{j=1}^k \lambda(\b)\chi_{s_j,x_j}(B)$.
Then, observe that:
\begin{align*}
&\sum_{k=0}^\infty \, \bar{\om}^{(k)}_t \left((\lambda\chi)^{\otimes k}\right)\\
& = \sum_{k=0}^\infty \sum_{J\subset [k]} (-1)^{k-|J|} \int_{\Delta_k(0,t)\times \mathbb{R}^k} \prod_{i=1}^k \lambda\chi_{s_i,x_i}\, \om_t^{(|J|)}(\rmd \mathbf{s}_J,\rmd \mathbf{x}_J) \, \nu^{k-|J|} \,\rmd \mathbf{s}_{J^c} \, \rmd \mathbf{x}_{J^c}\\
& = \sum_{k=0}^\infty \sum_{j=0}^k \binom{k}{j}(-1)^{k-j} \frac{1}{k!} \int_{[0,t]^k\times \mathbb{R}^k} \prod_{i=1}^k \lambda\chi_{s_i,x_i}\, \om_t^{(j)}(\rmd \mathbf{s}_{[j]},\rmd \mathbf{x}_{[j]}) \, \nu^{k-j} \,\rmd \mathbf{s}_{[j]^c} \, \rmd \mathbf{x}_{[j]^c}\\
& = \sum_{j=0}^\infty \frac{1}{j!}\, j!\, \om^{(j)} \left((\lambda\chi)^{\otimes j}\right) \sum_{k=j}^\infty \frac{1}{(k-j)!} (-t\lambda \nu r)^{k-j}\\
& = e^{-t\lambda r\nu} \sum_{j=0}^\infty \om^{(j)} \left((\lambda\chi)^{\otimes j}\right).
\end{align*}
Then, if we let $(s_1,x_1),\dots,(s_N,x_N)$ with $s_1<\dots<s_N$ be the points of $\om$ that lie in the tube $V_t(B)$, we get by definition of the $\om_t^{(j)}$'s that
\begin{equation*}
\sum_{j=0}^\infty \om_t^{(j)} \left((\lambda\chi)^{\otimes j}\right) = \sum_{J\subset [N]} \prod_{i\in J} \left(e^{\b \chi_{s_i,x_i}}-1 \right) = \prod_{i=1}^N e^{\b \chi_{s_i,x_i}} = e^{\b\om(V_t(B))},
\end{equation*}
where the last equality comes from a telescopic sum (and the convention that an empty product is $1$). This implies \eqref{eq:exp_expansion}. 

To prove convergence in $L^2$ of the sum in the RHS of \eqref{eq:wiener_ito_wt_expansion}, notice that the terms are pairwise orthogonal and verify:
\begin{align*}
&\IP\left[ \bar{\om}^{(k)}_t \left(T_k \, W_t\right)^2\right] = \lambda^{2k}\int_{\Delta_k(0,t)\times \mathbb{R}^k} P\left[\prod_{i=1}^k \chi_{s_i,x_i}(B) \right]^2\rmd \mathbf{s} \rmd \mathbf{x} \\
&\leq \frac{\lambda^{2k}}{k!} \int_{[0,t]^k\times \mathbb{R}^k} P\left[\prod_{i=1}^k \chi_{s_i,x_i}(B) \right]\rmd \mathbf{s} \rmd \mathbf{x}\\
&= \frac{(\lambda^{2}t\nu r)^k}{k!},
\end{align*}
whose sum converges.
\end{proof}

\section{The intermediate regime} \label{sec:intermediate_regime}
We now consider parameters $\beta_t\in\mathbb{R}$, $\nu_t > 0$ and $r_t > 0$ that depend on time $t$, and we fix a parameter $\b^*\in\mathbb{R}^*$. We assume that they verify the following asymptotic relations, as $t\to\infty$:
\begin{equation} \label{eq:paramRI}
\begin{gathered}
\text{(a)} \  \nu_t r_t^2 \lambda(\b_t)^2 \sim (\b^*)^{2}t^{-1/2}, \quad \text{(b)} \  \nu_t r_t^3 \lambda(\b_t)^3 \to 0,\\
\text{(c)} \  r_t/\sqrt{t} \to 0.
\end{gathered}
\end{equation}

Suppose for example that $r_t = \nu_t = 1$. Then, the scaling conditions are equivalent to $\b_t = \b^* t^{-1/4}$, so that we can see the scaling as a limit from strong disorder ($\b>0$) to weak disorder ($\b=0$). For general parameters, one can observe that in dimension $d\geq 3$, Theorem \ref{th:LTwoRegion} implies that there exists a positive constant $c(d)$, such that the polymer lies in the $L^2$ region as soon as
$\nu r^{d+2} \lambda(\b)^2 < c(d)$.  Since conditions (a) and (c) imply that $\nu_t r_t^3 \lambda(\b_t)^2 \to 0$, the scaling for $d=1$ should again be interpreted as a crossover between strong and weak disorder.

\begin{remark}
The asymptotics are in contrast to the regime of  complete localization \eqref{eq:completeloc}, where, for fixed $r$, one let $\nu \beta^2 \to \infty$.
\end{remark}

The following theorem states that under the above scaling, the P2L and P2P partition functions of the Poisson polymer converge to the one of the continuum polymer:
\begin{theorem} \label{th:cv_Wt_interReg} Suppose conditions (a), (b) and (c) hold. Then, as $t\to\infty$:
\begin{equation} \label{eq:P2PcvlawTh}
W_t(\om^{\nu_t},\b_t,r_t) \cvlaw \mathcal{Z}_{\beta^*},
\end{equation}
where $\om^{\nu_t}$ is the Poisson point process with intensity measure $\nu_t \rmd s \rmd x$.
Moreover, for all $S,Y,T,X \in [0,1]$, we have
\begin{equation}
\sqrt{t} W\left(tS,\sqrt{t} \hspace{0.3mm}Y;t\hspace{0.3mm}T,\sqrt{t}\hspace{0.3mm}X;\om^{\nu_t},\b_t,r_t\right) \cvlaw \mathcal{Z}_{\beta^*}\left(S,Y;T,X\right),
\end{equation}
where the renormalized P2P partition function from $(S,Y)$ to $(T,X)$ is defined by
\begin{equation}
W(s,y;t,x;\om,\b,r) = W(t-s,x-y;\om,\b,r)\circ \theta_{s,y},
\end{equation}
and similarly for $\mathcal{Z}_{\beta^*}\left(S,Y;T,X\right)$.
\end{theorem}
\begin{remark} \label{rk:sqrtt}
The $\sqrt{t}$ term appears here as a renormalization in the scaling of the heat kernel: $\sqrt{t}\rho\left(tT,\sqrt{t}X\right) = \rho(T,X)$.
\end{remark}
\begin{proof}[Sketch of proof]
We focus on showing \eqref{eq:P2PcvlawTh}, as the result for the P2P partition function follows from the same technique and remark \ref{rk:sqrtt}. Let $\gamma_t$ be proportional to the vanishing parameter appearing in scaling relation (b):
\begin{equation}
\gamma_t := (\b^*)^{-3} \nu_t r_t^3 \lambda(\b_t)^3 \to 0.
\end{equation}
and we now specify the radius for the indicator $\chi_{s,x}^{\delta}(B) = \mathbf{1}_{|B_s-y|\leq \delta/2}$.
Introduce the following time-depending functions of $\Delta_k(0,t)\times\mathbb{R}^k$:
\begin{equation}
\phi_t^k(\mathbf{s},\mathbf{x}) = \gamma_t^{-k} \, \lambda(\b_t) ^k P\left[\prod_{i=1}^k \chi_{s_i,x_i}^{r_t/\sqrt{t}} (B)\right].
\end{equation} 
Note that for all $(s,x)$, the diffusive scaling property of the Brownian motion implies that
\[\chi_{s/t,x/\sqrt{t}}^{r_t/\sqrt{t}} = \mathbf{1}_{|B_{s/t} - x/\sqrt{t}| \leq r_t/2\sqrt{t}} \eqlaw \chi^{r_t}_{s,x}.\]
Therefore, using notation $\widetilde{\phi}_t^k = \phi_t^k(\cdot/t,\cdot/\sqrt{t})$, we see that after simple rescaling, equation \eqref{eq:TkWtExpression} becomes
\begin{equation} \label{eq:phitktilde} \gamma_t^k \, \widetilde{\phi}_t^k(\mathbf{s},\mathbf{x}) = \lambda(\b)^k\, P\left[\prod_{i=1}^k \chi_{s_i,x_i}(B) \right].
\end{equation}
Hence, Proposition \eqref{prop:Poisson_chaos_decomposition} and equation \eqref{eq:phitktilde} lead to the following expression of $W_t$:\footnote{Note that from now on, we will always assume that $\om\eqlaw\om^{\nu_t}$, although we will drop the superscript notation.}
\begin{equation} \label{eq:WtExpansionTwo}
W_t = \sum_{k=0}^\infty \gamma_t^k \, \bar{\om}^{(k)}_t \left(\widetilde{\phi}_t^k\right).
\end{equation}

Now, we also define the rescaled functions $\tilde{\rho}^{\,k}_t = \rho^{k}(\cdot/t,\cdot/\sqrt{t})$ and we make two claims:
\begin{itemize}
\item \textbf{Claim 1:} For all $k\geq 0$ and as $t\to\infty$, $\phi^k_t \cvLdeux (\b^*)^k\rho^k$, 
\item \textbf{Claim 2:} As $t\to\infty$, $\sum_{k=0}^\infty (\beta^*)^k \gamma_t^k\,\bar{\om}_t^{(k)}(\tilde{\rho}^{\,k}_t) \cvlaw \sum_{k=0}^\infty (\b^*)^k I_k( \rho^k) = \mathcal{Z}_{\beta^*}$.
\end{itemize}
Claim 1 follows from the scaling relations and the fact that $\varepsilon^{-k} P[\prod_{i=1}^k \chi_{s_i,x_i}^\varepsilon(B)] \to \rho^k(\mathbf{s},\mathbf{x})$ as $\varepsilon\to 0$. For claim 2, we only present the argument for the convergence in law of the $k=1$ term of the sum. The complete argument relies on this case to extend the convergence to all $k\geq 1$ terms (see \cite{CoscoIntermediate}).
As $\tilde{\om}^{(1)}=\tilde{\om}$ and $\rho^1 = \rho$, we can apply the complex exponential formula (see equation \eqref{eq:PoissonExpMoment}) to compute the characteristic function of $\gamma_t \,\bar{\om}_t(\tilde{\rho}_t)$. For $u\in \mathbb{R}$, we obtain:
\begin{align*}&\IP\left[e^{iu\gamma_t\bar{\om}_t(\tilde{\rho}_t)}\right]\\
&=\exp\left(\int_{[0,t]} \int_\mathbb{R}  \big(e^{iu\gamma_t \rho(s/t,x/\sqrt{t})}-1 -iu\gamma_t \, \rho(s/t,x/\sqrt{t})\, \big) \nu_t \, \mathrm{d}s \mathrm{d}x \right)\\
&=\exp\left(\int_{[0,1]} \int_\mathbb{R} \nu_t t^{3/2} \big(e^{iu\gamma_t \rho(s,x)}-1 -iu\gamma_t \, \rho(s,x)\big)\mathrm{d}s \mathrm{d}x \right).
\end{align*}
By Taylor-Lagrange formula, we get that
\begin{equation*}
\forall (s,x)\in [0,1]\times\mathbb{R}, \  \nu_t t^{3/2}\left|e^{iu\gamma_t \rho(s,x)}-1 -iu\gamma_t \, \rho(s,x)\right|\leq \nu_t t^{3/2} \gamma_t^2 \frac{u^2}{2} \rho(s,x)^2,
\end{equation*}
This gives $L^1$ domination since since $\rho\in L^2([0,1]\times\mathbb{R})$ and since relations (a) and (b) imply that $ \nu_t \gamma_t^2 \sim t^{-3/2}$. Moreover, as $\gamma_t \to 0$, the integrand of the above integral converges pointwise to the function $(s,x) \mapsto -\frac{u^2}{2}\rho^2(s,x)$. Therefore, by dominated convergence, we obtain that as $t\to\infty$,
\[\IP\left[e^{iu\gamma_t\tilde{\om}_t(\tilde{\rho}_t)}\right] \to \exp\left(-\frac{u^2}{2} \Vert \rho\Vert^2_2\right).\] This is the Fourier transform of a centered Gaussian random variable of variance $\Vert \rho\Vert^2_2$, which has the same law as $I_1(\rho)$, so that indeed $\bar{\om}_t^{(1)}(\tilde{\rho}^{\,1}_t) \cvlaw I_1(\rho^1)$.

Now, if $X_n$ and $Y_n$ are random variables such that $Y_n \cvlaw Y$ and $\Vert Y_n-X_n\Vert_2 \longrightarrow 0$, then $X_n \cvlaw Y$.
Therefore, to prove that $W_t \cvlaw \mathcal{Z}_{\beta^*}$, it is enough by claim 2 to show that
\begin{equation}\left\Vert \sum_{k=0}^\infty \gamma_t^k \, \bar{\om}_t^{(k)}(\tilde{\phi}_t^k) - \sum_{k=0}^\infty (\b^*)^k \gamma_t^k\,\bar{\om}_t^{(k)}(\tilde{\rho}^{\,k}_t) \right\Vert_2^2 \underset{{t} \to {\infty}}{\longrightarrow} 0.
\end{equation}
By Pythagoras' identity and linearity of $\tilde{\om}_t^{(k)}$, we obtain that the above norm writes:
\[\sum_{k=0}^\infty \gamma_t^{2k} \, \Vert\bar{\om}_t^{(k)}\big(\phi_t^k(\cdot/t,\cdot/\sqrt{t})- (\b^*)^k\rho^{k}(\cdot/t,\cdot/\sqrt{t})\big)\Vert_2^2.\]
For all $g\in L^2(\Delta_k\times \mathbb{R}^k)$, we get from a substitution of variables that
\begin{equation*}\Vert\bar{\om}_t^{(k)}\big(g(\cdot/t,\cdot/\sqrt{t})\Vert_2^2 = \nu_t^k\Vert g(\cdot/t,\cdot/\sqrt{t})\Vert^2_{L^2(\Delta_k(0,t)\times \mathbb{R}^k)}  =  \nu_t^k t^{3k/2} \Vert g\Vert^2_{L^2(\Delta_k\times \mathbb{R}^k)},
\end{equation*}
so that the above sum is given by
\[\sum_{k=0}^\infty \gamma_t^{2k} \nu_t^k t^{3k/2} \, \Vert \phi_t^k-(\b^*)^k\rho^{k}\Vert_{L^2(\Delta_k\times \mathbb{R}^k)}^2.\]
Conditions (a) and (b) imply that $\gamma_t^2 \nu_t^k t^{3/2} \sim 1$, so the proof can be concluded by claim 1 and by showing that $\Vert \phi_t^k\Vert_2^2$ is dominated by the summable sequence $C^{2k} \Vert\rho^k\Vert^2_2$, where $C=C(\b^*)$ is some positive constant, so that the dominated convergence theorem applies.
\end{proof}

\section{Convergence in terms of processes of the P2P partition function}
\label{sec:proofCVProcesses}
Let $\mathcal{D}'(\mathcal R)$ denote the space of distributions on $\mathbb{R}$, and $D\big([0,1],\mathcal{D}'(\mathcal R)\big)$ the space of c\`adl\`ag function with values in the space of distributions, equipped with the topology defined in \cite{mitoma1983tightness}.
We also define the rescaling of the renormalized P2P partition function \eqref{eq:renormalizedP2Pdef}:
\begin{equation}
\mathcal{Y}_t\left(T,X\right) = \rho(T,X) W\left(tT,\sqrt{t}X;\om^{\nu_t},\b_t,r_t\right).
\end{equation}
The two variables function $\mathcal{Y}_t$ can be seen as an element of $D\big([0,1],\mathcal{D}'(\mathcal R)\big)$, through the mapping $\mathcal{Y}_t : T \mapsto \big(\varphi \mapsto \int \mathcal{Y}_t(T,X)\varphi(X)\rmd X\big)$. The next theorem states that the rescaled partition function $\mathcal{Y}_t$ converges, in terms of processes, to the solution of the stochastic heat equation:
\begin{theorem} \label{th:cvlaw} \textup{\cite{CoscoIntermediate}}
Suppose that $(\b_t)_{t\geq 0}$ is bounded by above. As $t\to\infty$, the following convergence of processes holds:
\begin{equation}
\mathcal{Y}_t \cvlaw \big(T\mapsto \mathcal{Z}_{\b^*}(T,\cdot)\big),
\end{equation}
where the convergence in distribution holds in $D\big([0,1],\mathcal{D}'(\mathcal R)\big)$.
\end{theorem}
For any function $F\in D\left([0,1],\mathcal{D}'(\mathbb{R})\right)$ and $\varphi\in \mathcal{D}(\mathbb{R})$, set
\begin{equation}
F(T,\varphi) := \int F(T,X)\varphi(X)\rmd X.
\end{equation}
In order to show tightness of $\mathcal{Y}_t$, the tool used in \cite{CoscoIntermediate} is Mitoma's criterion \cite{mitoma1983tightness,Walsh}:
\begin{proposition} \label{th:Mitoma}
Let $(F_t)_{t\geq 0}$ be a family of processes in $D\left([0,1],\mathcal{D}'(\mathbb{R})\right)$. If, for all $\varphi \in \mathcal{D}(\mathbb{R})$, the family $T\to F_t(T,\varphi),t\geq 0$ is tight in the real cadl\`ag functions space $D([0,1],\mathbb{R})$, then $(F_t)_{t\geq 0}$ is tight in $D\left([0,1],\mathcal{D}'(\mathbb{R})\right)$.
\end{proposition}
\noindent Then, to prove uniqueness of the limit, one can rely on the following proposition:
\begin{proposition}[\textup{\cite{mitoma1983tightness}}] \label{prop:mitomaIdentification}
Let $(F_t)_{t\geq 0}$ be a tight family of processes in the space $D\left([0,1],\mathcal{D}'(\mathbb{R})\right)$.  If there exists a process $F\in D\left([0,1],\mathcal{D}'(\mathbb{R})\right)$ such that, for all $n\geq 1$, $T_1,\dots,T_n\in [0,1]$ and $\varphi_1,\dots,\varphi_n \in \mathcal{D}(\mathbb R)$, we have as $t\to\infty$:
\[\left(F_t(T_1,\varphi_1),\dots,F_t(T_n,\varphi_n)\right) \cvlaw \left(F_t(T_1,\varphi_1),\dots,F_t(T_n,\varphi_n)\right),
\]
then $F_t \cvlaw F$.
\end{proposition}

{\small
\bibliographystyle{plain}
\bibliography{polymeres-bib.bib}

\begin{thebibliography}{10}

\bibitem{AlbertsKhaninQuastelCRP}
Tom Alberts, Konstantin Khanin, and Jeremy Quastel.
\newblock The continuum directed random polymer.
\newblock {\em J. Stat. Phys.}, 154(1-2):305--326, 2014.

\bibitem{AlbertsKhaninQuastelAP}
Tom Alberts, Konstantin Khanin, and Jeremy Quastel.
\newblock The intermediate disorder regime for directed polymers in dimension
  {$1+1$}.
\newblock {\em Ann. Probab.}, 42(3):1212--1256, 2014.

\bibitem{ACQ11}
Gideon Amir, Ivan Corwin, and Jeremy Quastel.
\newblock Probability distribution of the free energy of the continuum directed
  random polymer in {$1+1$} dimensions.
\newblock {\em Comm. Pure Appl. Math.}, 64(4):466--537, 2011.

\bibitem{AuffingerDamron13}
Antonio Auffinger and Michael Damron.
\newblock The scaling relation {$\chi=2\xi-1$} for directed polymers in a
  random environment.
\newblock {\em ALEA Lat. Am. J. Probab. Math. Stat.}, 10(2):857--880, 2013.

\bibitem{AuffingerDamron-sc}
Antonio Auffinger and Michael Damron.
\newblock A simplified proof of the relation between scaling exponents in
  first-passage percolation.
\newblock {\em Ann. Probab.}, 42(3):1197--1211, 2014.

\bibitem{BalaszQuastelSeppalainen11}
M{{\'a}}rton Bal{{\'a}}zs, Jeremy Quastel, and Timo Sepp{{\"a}}l{{\"a}}inen.
\newblock Fluctuation exponent of the {KPZ}/stochastic {B}urgers equation.
\newblock {\em J. Amer. Math. Soc.}, 24(3):683--708, 2011.

\bibitem{BatesChatterjee}
Erik Bates and Sourav Chatterjee.
\newblock The endpoint distribution of directed polymers.
\newblock {\em Ann. Probab.}, 48(2):817--871, 2020.

\bibitem{BeLa15}
Quentin Berger and Hubert Lacoin.
\newblock The high-temperature behavior for the directed polymer in dimension
  1+ 2.
\newblock {\em Ann. Inst. Henri Poincar{\'e} Probab. Stat.}, 53:430--450, 2017.

\bibitem{BertinUnp}
Pierre Bertin.
\newblock Free energy for brownian directed polymers in random environment in
  dimension one and two.
\newblock 2008.

\bibitem{Bertin}
Pierre Bertin.
\newblock Very strong disorder for the parabolic {A}nderson model in low
  dimensions.
\newblock {\em Indag. Math. (N.S.)}, 26(1):50--63, 2015.

\bibitem{BertiniCancrini95}
Lorenzo Bertini and Nicoletta Cancrini.
\newblock The stochastic heat equation: {F}eynman-{K}ac formula and
  intermittence.
\newblock {\em J. Statist. Phys.}, 78(5-6):1377--1401, 1995.

\bibitem{BertiniGiacomin97}
Lorenzo Bertini and Giambattista Giacomin.
\newblock Stochastic {B}urgers and {KPZ} equations from particle systems.
\newblock {\em Comm. Math. Phys.}, 183(3):571--607, 1997.

\bibitem{CarmonaMolchanov}
Ren{{\'e}}~A. Carmona and S.~A. Molchanov.
\newblock Parabolic {A}nderson problem and intermittency.
\newblock {\em Mem. Amer. Math. Soc.}, 108(518):viii+125, 1994.

\bibitem{Smoluchowski}
Subrahmanyan Chandrasekhar, Mark Kac, and R.~Smoluchowski.
\newblock {\em Marian {S}moluchowski: his life and scientific work.}
\newblock PWN---Polish Scientific Publishers, Warsaw, 2000.

\bibitem{Chatterjee-scaling}
Sourav Chatterjee.
\newblock The universal relation between scaling exponents in first-passage
  percolation.
\newblock {\em Ann. of Math. (2)}, 177(2):663--697, 2013.

\bibitem{clark1970representation}
John~MC Clark.
\newblock The representation of functionals of brownian motion by stochastic
  integrals.
\newblock {\em The Annals of Mathematical Statistics}, pages 1282--1295, 1970.

\bibitem{CStFlour}
Francis Comets.
\newblock {\em {Directed polymers in random environments. \'Ecole d'\'Et\'e de
  Probabilit\'es de Saint-Flour XLVI -- 2016.}}
\newblock Cham: Springer, 2017.

\bibitem{CYkokyuroku}
Francis Comets and Nobuo Yoshida.
\newblock Some new results on brownian directed polymers in random environment.
\newblock {\em RIMS Kokyuroku}, 1386:50--66, 2004.

\bibitem{CY05}
Francis Comets and Nobuo Yoshida.
\newblock Brownian directed polymers in random environment.
\newblock {\em Comm. Math. Phys.}, 254(2):257--287, 2005.

\bibitem{CYBMPO2}
Francis Comets and Nobuo Yoshida.
\newblock Localization transition for polymers in {P}oissonian medium.
\newblock {\em Comm. Math. Phys.}, 323(1):417--447, 2013.

\bibitem{corwin2016kardar}
Ivan Corwin.
\newblock Kardar-parisi-zhang universality.
\newblock {\em Notices of the AMS}, 63(3):230--239, 2016.

\bibitem{CoscoIntermediate}
Cl\'{e}ment Cosco.
\newblock The intermediate disorder regime for {B}rownian directed polymers in
  {P}oisson environment.
\newblock {\em Indag. Math. (N.S.)}, 30(5):805--839, 2019.

\bibitem{DeZe98}
Amir Dembo and Ofer Zeitouni.
\newblock {\em Large deviations techniques and applications}, volume~38 of {\em
  Applications of Mathematics (New York)}.
\newblock Springer-Verlag, New York, second edition, 1998.

\bibitem{dHollanderStFl}
Frank den Hollander.
\newblock {\em Random polymers}, volume 1974 of {\em Lecture Notes in
  Mathematics}.
\newblock Springer-Verlag, Berlin, 2009.
\newblock Lectures from the 37th Probability Summer School held in Saint-Flour,
  2007.

\bibitem{donskervaradhan-wienersaus}
Monroe Donsker and Srinivasa Varadhan.
\newblock Asymptotics for the wiener sausage.
\newblock {\em Communications on Pure and Applied Mathematics}, 28(4):525--565,
  1975.

\bibitem{Dur95}
Richard Durrett.
\newblock {\em Probability: theory and examples}.
\newblock Duxbury Press, Belmont, CA, second edition, 1996.

\bibitem{Giacomin}
Giambattista Giacomin.
\newblock {\em Random polymer models}.
\newblock Imperial College Press, London, 2007.

\bibitem{Howard00}
C~Douglas Howard.
\newblock Lower bounds for point-to-point wandering exponents in euclidean
  first-passage percolation.
\newblock {\em Journal of applied probability}, 37(4):1061--1073, 2000.

\bibitem{HowardNewman97}
C~Douglas Howard and Charles~M Newman.
\newblock Euclidean models of first-passage percolation.
\newblock {\em Probability Theory and Related Fields}, 108(2):153--170, 1997.

\bibitem{Hsu2013}
Elton~P. Hsu and Karl-Theodor Sturm.
\newblock Maximal coupling of euclidean brownian motions.
\newblock {\em Communications in Mathematics and Statistics}, 1(1):93--104, Mar
  2013.

\bibitem{IkedaWatanabe}
N.~Ikeda and S.~Watanabe.
\newblock {\em Stochastic differential equations and diffusion processes.
  Second Edition}.
\newblock Kodansha scientific books. North-Holland, 1989.

\bibitem{ito1951multiple}
Kiyosi It{\^o}.
\newblock Multiple wiener integral.
\newblock {\em Journal of the Mathematical Society of Japan}, 3(1):157--169,
  1951.

\bibitem{jacod2013limit}
Jean Jacod and Albert Shiryaev.
\newblock {\em Limit theorems for stochastic processes, 2nd edition}, volume
  288.
\newblock Springer Science \& Business Media, 2003.

\bibitem{janson_1997}
Svante Janson.
\newblock {\em Gaussian Hilbert Spaces}.
\newblock Cambridge Tracts in Mathematics. Cambridge University Press, 1997.

\bibitem{KS91}
Ioannis {Karatzas} and Steven~E. {Shreve}.
\newblock {\em {Brownian motion and stochastic calculus. 2nd ed.}}
\newblock Springer-Verlag (New York), 1991.

\bibitem{kardar1986dynamic}
Mehran Kardar, Giorgio Parisi, and Yi-Cheng Zhang.
\newblock Dynamic scaling of growing interfaces.
\newblock {\em Physical Review Letters}, 56(9):889, 1986.

\bibitem{khoshnevisan}
Davar Khoshnevisan.
\newblock {\em Analysis of stochastic partial differential equations.}, volume
  119 of {\em CBMS Regional Conference Series in Mathematics}.
\newblock Published for the Conference Board of the Mathematical Sciences,
  Washington, DC, 2014.

\bibitem{KrugSpohn}
Joaquim Krug and Herbert Spohn.
\newblock {\em Kinetic roughening of growing surfaces (in: Solids far from
  equilibrium, Godr{\`e}che ed.)}, pages 477--582.
\newblock Cambridge University Press, Cambridge, 1992.

\bibitem{Lacoin10}
Hubert Lacoin.
\newblock New bounds for the free energy of directed polymers in dimension
  {$1+1$} and {$1+2$}.
\newblock {\em Comm. Math. Phys.}, 294(2):471--503, 2010.

\bibitem{Last2016}
G{\"u}nter Last.
\newblock {\em Stochastic Analysis for Poisson Processes}, pages 1--36.
\newblock Springer International Publishing, Cham, 2016.

\bibitem{LastPenrose}
G\"unter Last and Mathew Penrose.
\newblock {\em Lectures on the Poisson process.}
\newblock Cambridge University Press (IMS Textbook), 2017.

\bibitem{LAST20111588}
G\"unter Last and Mathew~D. Penrose.
\newblock Martingale representation for poisson processes with applications to
  minimal variance hedging.
\newblock {\em Stochastic Processes and their Applications}, 121(7):1588 --
  1606, 2011.

\bibitem{LiceaNewPiza}
C.~Licea, C.~M. Newman, and M.~S.~T. Piza.
\newblock Superdiffusivity in first-passage percolation.
\newblock {\em Probab. Theory Related Fields}, 106(4):559--591, 1996.

\bibitem{mitoma1983tightness}
Itaru Mitoma.
\newblock Tightness of probabilities on c ([0, 1]; y') and d ([0, 1]; y').
\newblock {\em The Annals of Probability}, pages 989--999, 1983.

\bibitem{Zeldovich}
S.~A. Molchanov, A.~A. Ruzmaikin, D.~D. Sokoloff, and Ya.~B. Zeldovich.
\newblock {\em Intermittency, diffusion and generation in a nonstationary
  random medium}.
\newblock Reviews in Mathematics and Mathematical Physics, 15/1. Cambridge
  Scientific Publishers, Cambridge, 2014.

\bibitem{morenoflores2014}
Gregorio~R. Moreno~Flores.
\newblock On the (strict) positivity of solutions of the stochastic heat
  equation.
\newblock {\em Ann. Probab.}, 42(4):1635--1643, 07 2014.

\bibitem{MorenoSeppValko}
Gregorio~R. Moreno~Flores, Timo Sepp{{\"a}}l{{\"a}}inen, and Benedek
  Valk{{\'o}}.
\newblock Fluctuation exponents for directed polymers in the intermediate
  disorder regime.
\newblock {\em Electron. J. Probab.}, 19:no. 89, 28, 2014.

\bibitem{Mueller}
Carl Mueller.
\newblock On the support of solutions to the heat equation with noise.
\newblock {\em Stochastics Stochastics Rep.}, 37(4):225--245, 1991.

\bibitem{ocone1984malliavin}
Daniel Ocone.
\newblock Malliavin's calculus and stochastic integral representations of
  functional of diffusion processes.
\newblock {\em Stochastics: An International Journal of Probability and
  Stochastic Processes}, 12(3-4):161--185, 1984.

\bibitem{RASaihp11}
Firas Rassoul-Agha and Timo Sepp{\"a}l{\"a}inen.
\newblock Process-level quenched large deviations for random walk in random
  environment.
\newblock In {\em Annales de l'institut Henri Poincar{\'e} (B)}, volume~47,
  pages 214--242, 2011.

\bibitem{revuzYor}
Daniel Revuz and Marc Yor.
\newblock {\em Continuous martingales and Brownian motion, 3d edition}, volume
  293.
\newblock Springer Science \& Business Media, 2005.

\bibitem{Roc70}
R.~Tyrrell Rockafellar.
\newblock {\em Convex analysis}.
\newblock Princeton Mathematical Series, No. 28. Princeton University Press,
  Princeton, N.J., 1970.

\bibitem{RogersWilliams}
L.~Chris Rogers and David Williams.
\newblock {\em Diffusions, Markov processes and Martingales, vol 2: Ito
  calculus}, volume Reprint of the second (1994) edition.
\newblock Cambridge University Press, 2000.

\bibitem{RoTi05}
Carles Rovira and Samy Tindel.
\newblock On the brownian-directed polymer in a gaussian random environment.
\newblock {\em Journal of Functional Analysis}, 222:178--201, 2005.

\bibitem{Sepp12}
Timo Sepp\"al\"ainen.
\newblock Scaling for a one-dimensional directed polymer with boundary
  conditions.
\newblock {\em Ann. Probab.}, 40:19--73, 2012.

\bibitem{SeppValko10}
Timo Sepp\"al\"ainen and Benedek Valk\'o.
\newblock Bounds for scaling exponents for a {$1+1$} dimensional directed
  polymer in a {B}rownian environment.
\newblock {\em ALEA Lat. Am. J. Probab. Math. Stat.}, 7:451--476, 2010.

\bibitem{Shiozawa-clt}
Yuichi Shiozawa.
\newblock Central limit theorem for branching {B}rownian motions in random
  environment.
\newblock {\em J. Stat. Phys.}, 136(1):145--163, 2009.

\bibitem{Shiozawa-loc}
Yuichi Shiozawa.
\newblock Localization for branching {B}rownian motions in random environment.
\newblock {\em Tohoku Math. J. (2)}, 61(4):483--497, 2009.

\bibitem{Si95}
Yakov~G. Sinai.
\newblock A remark concerning random walks with random potentials.
\newblock {\em Fund. Math.}, 147(2):173--180, 1995.

\bibitem{KendallStoyanMecke}
Dietrich Stoyan, Wilfrid Kendall, and Joseph Mecke.
\newblock {\em Stochastic geometry and its applications}.
\newblock Akademie-Verlag, Berlin, 1987.

\bibitem{Sznitman}
Alain-Sol Sznitman.
\newblock {\em Brownian motion, obstacles and random media}.
\newblock Springer Monographs in Mathematics. Springer-Verlag, Berlin, 1998.

\bibitem{tracy1994level}
Craig~A Tracy and Harold Widom.
\newblock Level-spacing distributions and the airy kernel.
\newblock {\em Communications in Mathematical Physics}, 159(1):151--174, 1994.

\bibitem{Va04}
Vincent Vargas.
\newblock A local limit theorem for directed polymers in random media: the
  continuous and the discrete case.
\newblock {\em Ann. Inst. H. Poincar{\'e} Probab. Statist.}, 42(5):521--534,
  2006.

\bibitem{VargasAtoms07}
Vincent Vargas.
\newblock Strong localization and macroscopic atoms for directed polymers.
\newblock {\em Probability theory and related fields}, 138(3-4):391--410, 2007.

\bibitem{Walsh}
John~B. Walsh.
\newblock An introduction to stochastic partial differential equations.
\newblock In {\em \'{E}cole d'{\'e}t{\'e} de probabilit{\'e}s de
  {S}aint-{F}lour, {XIV}---1984}, volume 1180 of {\em Lecture Notes in Math.},
  pages 265--439. Springer, Berlin, 1986.

\bibitem{Wuthrich2}
Mario~V. W{{\"u}}thrich.
\newblock Scaling identity for crossing {B}rownian motion in a {P}oissonian
  potential.
\newblock {\em Probab. Theory Related Fields}, 112(3):299--319, 1998.

\bibitem{Wuthrich1}
Mario~V. W{{\"u}}thrich.
\newblock Superdiffusive behavior of two-dimensional {B}rownian motion in a
  {P}oissonian potential.
\newblock {\em Ann. Probab.}, 26(3):1000--1015, 1998.

\bibitem{Wuthrich-geodesics}
Mario~V. W{{\"u}}thrich.
\newblock Geodesics and crossing {B}rownian motion in a soft {P}oissonian
  potential.
\newblock {\em Ann. Inst. H. Poincar{\'e} Probab. Statist.}, 35(4):509--529,
  1999.

\bibitem{Wuthrich-bornes}
Mario~V. W{{\"u}}thrich.
\newblock Numerical bounds for critical exponents of crossing {B}rownian
  motion.
\newblock {\em Proc. Amer. Math. Soc.}, 130(1):217--225, 2002.

\end{thebibliography}
}

\end{document}